\documentclass[11pt, reqno]{amsart}
\usepackage{amsmath, amssymb, amsthm, graphicx,marvosym}
\usepackage{cancel}
\usepackage[nobysame]{amsrefs}[11pts]

\usepackage{color}

\newtheorem{theorem}{Theorem}[section]
\newtheorem{definition}{Definition}[section]
\newtheorem{lemma}{Lemma}[section]
\newtheorem{remark}{Remark}[section]

\newtheorem{corollary}{Corollary}[section]
\numberwithin{equation}{section}
\numberwithin{figure}{section}

\makeatletter \@addtoreset{equation}{section} \makeatother

\newcommand{\RR}{\mathbb{R}}
\newcommand{\R}{\mathbb{R}}
\newcommand{\supp}{{\mathrm {supp}}}
\newcommand{\beq}{\begin{equation}}
\newcommand{\eeq}{\end{equation}}

\newcommand{\dd}{{\rm d}}

\begin{document}
\title[Vanishing Viscosity Limit of the Compressible Navier-Stokes Equations]{Vanishing Viscosity Limit
of the Three-Dimensional Barotropic Compressible Navier-Stokes Equations with Degenerate
Viscosities and Far-Field Vacuum}

\author{Geng Chen}
\address{Geng Chen:\, School of Mathematics, University of Kansas, Lawrence, KS 66045, USA.}
\email{\tt gengchen@ku.edu}

\author{Gui-Qiang G. Chen}
\address{Gui-Qiang G. Chen:\, Mathematical Institute,
University of Oxford, Oxford,  OX2 6GG, UK.}
\email{\tt chengq@maths.ox.ac.uk}

\author{Shengguo Zhu}
\address{Shengguo Zhu:\, School of Mathematical Sciences, and MOE-LSC, Shanghai Jiao Tong University, Shanghai, 200240, China;
Mathematical Institute, University of Oxford,  Oxford, OX2 6GG, UK.}
\email{\tt zhushengguo@sjtu.edu.ck}


\date{\today}

\begin{abstract}
We are concerned with the inviscid limit of the Navier-Stokes equations to
the Euler equations for barotropic compressible
fluids in $\RR^3$.
When the viscosity coefficients obey a lower power-law of the density
({\it i.e.}, $\rho^\delta$ with $0<\delta<1$),
we identify a  {\it quasi-symmetric hyperbolic--singular elliptic coupled structure}
of the Navier-Stokes equations to control the behavior of the velocity of the fluids
near a vacuum.
Then this structure is employed to prove that there exists a unique regular solution
to the corresponding Cauchy problem
with arbitrarily large initial data and far-field vacuum,
whose life span is uniformly positive in the vanishing viscosity limit.
Some uniform estimates on both the local sound speed and
the velocity in $H^3(\mathbb{R}^3)$ with respect to the viscosity coefficients
are also obtained, which lead to the strong convergence of the regular solutions of the Navier-Stokes equations
with finite mass and energy to the corresponding regular solutions of the Euler equations
in $L^{\infty}([0, T]; H^{s}_{\rm loc}(\mathbb{R}^3))$ for any $s\in [2, 3)$.
As a consequence,
we show that, for both viscous and inviscid flows, it is impossible that
the $L^\infty$ norm of any global regular solution with vacuum  decays  to zero asymptotically,
as $t$ tends to infinity.
Our framework developed here is applicable to the same problem for the other physical
dimensions via
some minor modifications.
\end{abstract}
\maketitle

\tableofcontents

\section{Introduction}
The time evolution of the density $\rho^\epsilon \geq 0$ and velocity
$u^\epsilon=\left((u^\epsilon)^{(1)},(u^\epsilon)^{(2)},(u^\epsilon)^{(3)}\right)^\top$
$\in \mathbb{R}^3$  of a general viscous barotropic
compressible  fluid in $ \mathbb{R}^3$ is governed by the following
compressible Navier-Stokes equations (\textbf{CNS}):
\begin{equation}
\label{eq:1.1}
\begin{cases}
\rho^\epsilon_t+\text{div}(\rho^\epsilon u^\epsilon)=0,\\[2mm]
(\rho^\epsilon u^\epsilon)_t+\text{div}(\rho^\epsilon u^\epsilon\otimes u^\epsilon)
  +\nabla
   p(\rho^\epsilon) =\epsilon\,\text{div} \mathbb{T}(\rho^\epsilon, \nabla u^\epsilon),
\end{cases}
\end{equation}
where $\epsilon>0$ is a constant measuring the strength of viscosity which is assumed  to be $\epsilon\in (0,1]$
without loss of generality,
and $x=(x_1,x_2,x_3)\in \mathbb{R}^3$ and $t\geq 0$ are the space and time variables, respectively.
For polytropic gases, the constitutive relation is given by
\begin{equation}
\label{eq:1.2}
p(\rho)=A \rho^{\gamma} \qquad \mbox{with $A>0$ and $\gamma> 1$},
\end{equation}
where $A$ is an entropy constant and  $\gamma$ is the adiabatic exponent.
The viscous stress tensor $\mathbb{T}(\rho, \nabla u)$ has the form:
\begin{equation}
\label{eq:1.3}
\mathbb{T}(\rho, \nabla u)
=\mu(\rho)\big(\nabla u+(\nabla u)^\top\big)+\lambda(\rho)\text{div} u\,\mathbb{I}_3,
\end{equation}
where $\mathbb{I}_3$ is the $3\times 3$ identity matrix,
\begin{equation}
\label{fandan}
\mu(\rho)= \alpha \rho^\delta,\qquad \lambda(\rho)= \beta  \rho^\delta
\end{equation}
for some constant $\delta>0$,
$\mu(\rho)$ is the shear viscosity coefficient,
$\lambda(\rho)+\frac{2}{3}\mu(\rho)$ is the bulk viscosity coefficient,
$\alpha$ and $\beta$ are both constants satisfying
 \begin{equation}\label{10000}
 \alpha>0, \qquad \alpha+\beta\geq 0.
 \end{equation}

Formally, when $\epsilon=0$, the Navier-Stokes equations \eqref{eq:1.1}
reduce to the compressible Euler equations for inviscid flow:
\begin{equation}
\label{eq:1.1E}
\begin{cases}
\displaystyle
\rho_t+\text{div}(\rho u)=0,\\[2mm]
\displaystyle
(\rho u)_t+\text{div}(\rho u\otimes u)
  +\nabla p(\rho) =0,
\end{cases}
\end{equation}
where $\rho$ and $u$ are the mass density and velocity of the inviscid fluid.

In this paper, we focus on the lower power-law case $\delta\in (0, 1)$ to analyze
the asymptotic behavior of the smooth solution $(\rho^\epsilon,u^\epsilon)$ with finite mass
and energy as $\epsilon \rightarrow 0$
for the Cauchy problem \eqref{eq:1.1}--\eqref{10000} with the following initial data and far-field vacuum:
\begin{align}
&(\rho^\epsilon,u^\epsilon)|_{t=0}=(\rho^\epsilon_0,  u^\epsilon_0)(x)
   &&\mbox{with $\rho_0^\epsilon(x)>0$ for $x\in \mathbb{R}^3$}, \label{initial}\\
&(\rho^\epsilon_0,u^\epsilon_0)(x)\rightarrow  (0,0)  &&\text{as $|x|\rightarrow \infty$}.\label{far}
\end{align}
Our results show that the inviscid flow (\ref{eq:1.1E}) can be regarded as
the viscous flow  \eqref{eq:1.1} with vanishing viscosity for the regular solutions
with far-field vacuum in the sense of Definition \ref{d1}.
The far-field behavior
\begin{align}
(\rho^\epsilon,u^\epsilon)(t,x)\rightarrow  (0,0) \quad\,\, \text{as $|x|\rightarrow \infty$ for $t\ge 0$}
\end{align}
is natural when some physical requirements are imposed,
for example, when the total mass is finite in $\mathbb{R}^3$.

In the theory of gas dynamics, system \eqref{eq:1.1} may be derived from the Boltzmann equation
through the Chapman-Enskog expansion; see Chapman-Cowling \cite{chap} and Li-Qin  \cite{tlt}.
For some physical situations, the viscosity coefficients $\mu$ and $\lambda$
and the heat conductivity coefficient $\kappa$
are functions of the absolute temperature $\theta$ ({\it cf.} \cite{chap}) such as
\begin{equation}
\label{eq:1.5g}
\mu(\theta)= a_1 \theta^{\frac{1}{2}}F(\theta),\quad\,  \lambda(\theta)=a_2 \theta^{\frac{1}{2}}F(\theta),
\quad\, \kappa(\theta)=a_3 \theta^{\frac{1}{2}}F(\theta)
\end{equation}
for some constants $a_i, i=1,2,3$.
In fact, for the cut-off inverse power force model,
if the intermolecular potential varies as $r^{-a}$, then
$$
F(\theta)=\theta^b\qquad  \text{with $b=\frac{2}{a} \in [0,\infty)$},
$$
where $ r$ is the intermolecular distance.
In particular, $a=1$ and $b=2$ for the ionized gas,
$a = 4$  and $b=\frac{1}{2}$ for Maxwellian molecules,
and $a=\infty$ and  $b=0$ for rigid elastic spherical molecules; see \cite[\S 10]{chap}.
As a typical example where $F$ is not a power function of $\theta$, the Sutherland's model is well known:
\begin{equation}\label{sutherland}
F(\theta)=\frac{\theta}{\theta+s_0},
\end{equation}
where $s_0>0$ is the Sutherland constant.
According to Liu-Xin-Yang \cite{taiping},
for barotropic and polytropic fluids,
such a dependence is inherited through the laws of Boyle and Gay-Lussac:
$$
p=R\rho\theta =A \rho^\gamma  \qquad \text{for constant $R>0$};
$$
that is, $\theta=AR^{-1}\rho^{\gamma-1}$ so that the viscosity coefficients are functions
of the density.
For most physical processes, $\gamma$ is in interval $(1,3)$,
which implies that $\delta\in (0,1)$ for the rigid elastic spherical molecules.
In this sense, the lower power-law case is the most physically relevant
for the degenerate viscous flow. In this paper, we focus on this case.
In fact, similar models with density dependent viscosity coefficients arise from various physical situations,
such as the Korteweg system, the shallow water equations, the lake equations,
and the quantum Navier-Stokes systems ({\it cf}. \cites{bd6,bd3, bd2, bd, BN1, BN2, ansgar,decayd}).

Another motivation of our study is that the mathematical structure of system \eqref{eq:1.1}
is an excellent prototype
of nonlinear degenerate systems of partial differential equations,
since the equations on the time evolution of the fluid velocity
become very singular near the far-field due to the decay of the density
at a certain rate, which will be further discussed later.

\subsection{Well-posedness for compressible flow with vacuum}
We first recall some related  frameworks on the well-posedness for
strong solutions with vacuum
of the Cauchy problem of the hydrodynamics equations mentioned above.
For the inviscid flow, in Makino-Ukai-Kawashima \cite{tms1},
the local sound speed $c$ was first introduced
to rewrite the system in a symmetric hyperbolic form,
and the local existence of the unique regular solution
of the compressible full Euler equations with vacuum was established;
see also Makino-Perthame \cite{mp} for the Euler-Poisson equations.
For the isentropic flow in $\mathbb{R}^3$, the sound speed $c$ is defined by
\begin{equation}\label{sonic-speed}
c:=\sqrt{p'(\rho)}=\sqrt{A\gamma} \rho^{\frac{\gamma-1}{2}}.
\end{equation}
Then the result in \cite{tms1} can be read as follows:

\begin{theorem}[\cite{tms1}]\label{thmakio}
Let $\gamma >1$. For the Cauchy problem \eqref{eq:1.1E} with
\begin{align}
(\rho,u)|_{t=0}=(\rho_0, u_0)(x)  \qquad  \text{for $x\in \mathbb{R}^3$},\label{winitial}
\end{align}
if the initial data $( \rho_0, u_0)$ satisfy
\begin{equation}\label{th78asd}
\rho_0(x)\geq 0,\qquad \big(c_0, u_0\big)\in
H^3(\mathbb{R}^3),
\end{equation}
then there exist $T_0>0$ and a unique regular solution
$(\rho, u)$ of the Cauchy problem \eqref{eq:1.1E} with
\eqref{winitial}
satisfying
\begin{equation}\label{reg11asd}
(c,u) \in C([0,T_0];H^3(\mathbb{R}^3)),\qquad (c_t,u_t) \in C([0,T_0];H^2(\mathbb{R}^3)),
\end{equation}
where the regular solution $(\rho, u)$  is defined in the following sense{\rm :}
\begin{enumerate}
\item[\rm (i)]  $(\rho,u)$ satisfies \eqref{eq:1.1E} with
\eqref{winitial}
in the sense of  distributions,

\smallskip
\item[\rm (ii)] $\rho\geq 0$ and $\big(c, u\big)\in C^1([0,T_0]\times \mathbb{R}^3)$,

\smallskip
\item[\rm (iii)] $u_t+u\cdot\nabla u =0$ when $\rho(t,x)=0$.
\end{enumerate}
\end{theorem}
The corresponding global well-posedness of smooth solutions with small density but possibly large velocity
in some homogeneous Sobolev spaces was proved in Grassin-Serre \cites{MG, serre}.
Chen-Chen-Zhu \cite{Geng1} pinpointed the necessary and sufficient condition
for the formation of singularities of $C^1$ solutions with large initial data allowing
a far-field vacuum for the one-dimensional space:
there exists a compression in the initial data.

For the compressible viscous flow away from a vacuum, the local existence and uniqueness
of classical solutions are known in Nash \cite{nash} and Serrin \cite{serrin1}.
However, if the initial density does not have a strictly positive lower bound,
the arguments used in \cites{nash, serrin1} cannot apply to system $(\ref{eq:1.1})$, owing to the degeneracy
caused by the vacuum or the decay of the density in the far-field  in the momentum equations $(\ref{eq:1.1})_2$
such as
$$
\displaystyle
\underbrace{\rho^\epsilon(u^\epsilon_t+u^\epsilon\cdot \nabla u^\epsilon)}_{Degeneracy \ of  \ time \ evolution}
+\,\nabla p(\rho^\epsilon)
=\epsilon \underbrace{\text{div}\big(2\mu(\rho^\epsilon)D(u^\epsilon)
  +\lambda(\rho^\epsilon)\text{div} u^\epsilon \mathbb{I}_3\big)}_{Degeneracy\ of \ viscosities},
$$
where $D(u)=\frac{1}{2}\big(\nabla u +(\nabla u)^\top\big)$.
For the constant viscous flow ({\it i.e.}, $\delta=0$ in (\ref{fandan})),
in order to establish the local well-posedness of strong solutions with vacuum in $\mathbb{R}^3$,
a remedy was suggested  by Cho-Choe-Kim \cite{CK3} for dealing with the degeneracy of time evolution,
where they imposed  initially a {\it compatibility condition}:
\begin{equation*}
\nabla p(\rho^\epsilon_0)-\text{div} \mathbb{T}_{0} =\sqrt{\rho^\epsilon_{0}} g
\qquad\,\, \text{for some $g\in L^2(\mathbb{R}^3)$}.
\end{equation*}
Later, based on the uniform estimate of the upper bound of the density,
Huang-Li-Xin \cite{HX1} extended this solution to be a global one with small energy for barotropic flow
in $\mathbb{R}^3$.

For the degenerate viscous flow ({\it i.e.}, $\delta>0$ in (\ref{fandan})),
the strong degeneracy of the momentum equations in  \eqref{eq:1.1}  near the vacuum causes new difficulties
for the mathematical analysis of this system.
Based on the B-D entropy introduced by Bresch-Desjardins \cites{bd6,bd3},
some significant results on the weak solutions of the isentropic \textbf{CNS} or  related models
whose viscosity coefficients satisfy the B-D relation
have been obtained; see \cites{bd2,bd,lz,vassu,vayu}.
However, owing to the degenerate mathematical structure and lack of smooth effects on the solutions
when a vacuum appears, many fundamental questions remain open,
including the uniform estimate of the life span of the corresponding strong solutions with respect
to $\epsilon$, the identification of the classes of initial data that
either cause the finite time blow-up or provide the global existence
of the strong solutions, and the well-posedness of solutions with vacuum for system \eqref{eq:1.1} without the
B-D relation for the viscosity coefficients.

In fact, in \cites{CK3,HX1}, the uniform ellipticity of the Lam\'e operator $L$ defined by
\begin{equation*}
Lu := -\alpha \triangle u-(\alpha+\beta)  \nabla \text{div} u
\end{equation*}
plays an essential role in improving the regularity of  $u^\epsilon$, which can be shown as
\begin{equation}\label{c-33}
\big\|\nabla^{k+2}u^\epsilon\big\|_{L^2(\mathbb{R}^3)}
\leq C\epsilon^{-1} \big\|\nabla^k(\rho^\epsilon u^\epsilon_t+\rho^\epsilon u^\epsilon\cdot \nabla u^\epsilon
+\nabla p(\rho^\epsilon))\big\|_{L^2(\mathbb{R}^3)}
\end{equation}
for  $k=0, 1, 2$,
and  some constant $C>0$ independent of $\epsilon$.
However, when $\delta>0$ in  (\ref{fandan}),
the viscosity coefficients approach zero continuously near the vacuum.
This degeneracy makes it difficult to adapt the elliptic approach (\ref{c-33})
to the present case.
Moreover, we need to pay additional attention to deal with the strong nonlinearity
of the variable coefficients of the viscous term due to $\delta>0$,
which is another crucial issue owing to the appearance of a vacuum or the density decay in the far-field.

Recently, when $0<\delta<1$, under the initial compatibility conditions:
\begin{equation}\label{th78zx}
\displaystyle
  \nabla u^\epsilon_0=(\rho^\epsilon_0)^{\frac{1-\delta}{2}} g_1,
   \quad \   \ \  Lu^\epsilon_0= (\rho^\epsilon_0)^{1-\delta}g_2,\quad \ \
 \nabla \big((\rho^\epsilon_0)^{\delta-1}Lu^\epsilon_0\big)=(\rho^\epsilon_0)^{\frac{1-\delta}{2}}g_3
\end{equation}
for some $(g_1,g_2,g_3)\in L^2(\mathbb{R}^3)$, the existence of the unique local classical solution
with far-field vacuum to (\ref{eq:1.1}) in $\R^3$ was established in Xin-Zhu \cite{zz2}
by introducing an elaborate singular elliptic approach
on the two operators $Lu$ and $L(\rho^{\delta-1}u)$.
For the case $\delta=1$,  in Li-Pan-Zhu \cite{sz3}, it was first observed that the degeneracies
of the time evolution and viscosities can be transferred to the possible singularity
of the term  $\nabla \log \rho \cdot \nabla u$.
Then, by establishing some uniform estimates of $\nabla \log \rho^\epsilon$ in $L^6(\mathbb{R}^3)$
and $\nabla^2 \log \rho^\epsilon$ in $H^1(\mathbb{R}^3)$ with respect to the lower bound of the initial density,
the existence of the unique local classical solution with far-field vacuum of system \eqref{eq:1.1} in $\R^2$ was obtained,
which also applies to the two-dimensional shallow water equations.
Later, by introducing some hyperbolic approach and making a full use of weak smooth effect
on the solutions of system (\ref{eq:1.1})
to establish some weighted estimates on the highest order term $\nabla^4u^\epsilon$,
the existence of three-dimensional local classical solutions was obtained in \cite{sz333} when
$1<\delta\leq \min\{3, \frac{\gamma+1}{2}\}$,
and the corresponding global well-posedness in some homogeneous Sobolev spaces  was established by Xin-Zhu \cite{zz}
under some initial smallness assumptions.
See also \cites{cons, Has, hailiang, hyp, sz2, zyj, tyc2, szhu, sz34} for other related results.

\subsection{Vanishing viscosity limit}
Based on the well-posedness theory mentioned above,
a natural followup question is to understand the relation between the regular solutions of inviscid flow in \cite{tms1}
and those of viscous flow for $\delta\geq 0$ in \cites{CK3, sz3, sz333, zz,zz2} with vanishing physical viscosities,
especially for the hard sphere model when $0<\delta<1$.

There is a considerable literature
on the uniform bounds and the vanishing viscosity limit in the whole space.
The idea of regarding inviscid flow as viscous flow with vanishing physical viscosity
dates back to Hugoniot \cite{H1}, Rankine \cite{R1}, Rayleigh \cite{R2}, and Stokes \cite{stoke};
see Dafermos \cite{guy}.
However, it was only in 1951 that Gilbarg \cite{DG} gave the first rigorous convergence analysis
of vanishing physical viscosity limit from the Navier-Stokes equations (1.1)
to the isentropic Euler equations (\ref{eq:1.1E}),
and established the mathematical existence and vanishing viscous limit of the Navier-Stokes shock layers.
The framework on the convergence analysis of piecewise smooth solutions was established
by G$\grave{\text{u}}$es-M$\acute{\text{e}}$tivier-Williams-Zumbrun
\cite{OG}, Hoff-Liu \cite{HL}, and the references cited therein.
Klainerman-Majda \cite{KM} established the convergence of smooth solutions of the Navier-Stokes equations
to solutions of  the Euler equations in $H^s(\mathbb{R}^d)$ for $s>\frac{d}{2}+1$.
In 2009, combining the uniform energy estimates with compactness compensated arguments,
Chen-Perepelitsa \cite{chen4}
provided the first rigorous proof of the vanishing physical viscosity limit
of solutions of the one-dimensional Navier-Stokes equations to a finite-energy entropy solution
of the isentropic Euler equations with relative finite-energy initial data.
Some further results can be found  in Bianchini-Bressan \cite{bianbu} for strictly hyperbolic systems with artificial viscosity,
Chen-Perepelitsa \cite{chen3} and Chen-Schrecker \cite{chen5}  for the spherically symmetric case with artificial viscosity,
as well as Huang-Pan-Wang-Wang-Zhai \cite{HPWWZ} and Germain-LeFloch \cite{Germain}.

However, even in the one-dimensional case, owing to the complex mathematical structure
of the hydrodynamics equations near a vacuum,
the existence results for strong solutions for the viscous and inviscid cases are often established in totally different frameworks,
such as \cites{CK3,sz3,sz333, zz2} for viscous flow and \cites{Geng1, tms1} for inviscid flow.
In fact, the arguments used in \cites{CK3, HX1, sz3,sz333,zz2} essentially rely on the uniform ellipticity
of the Lam\'{e} operator or some related elliptic operators,
and both the desired {\it a priori} estimates of the solutions
and the life spans $T^{\epsilon}$
depend strictly on the real physical viscosities.
For example,
\begin{itemize}
\item  when $\delta=0$ in \cites{CK3,HX1}, $T^{\epsilon}=O(\epsilon)$ and (\ref{c-33});

\smallskip
\item
when $0<\delta<1$  in \cite{zz2}, $T^{\epsilon}=O(\epsilon)$ and
\begin{equation*}
\begin{split}
\displaystyle
&\big\|\nabla^{k+2}\big((\rho^\epsilon)^{\delta-1}  u^\epsilon\big)\big\|_{L^2(\mathbb{R}^3)} \\
& \leq  C\epsilon^{-1}\big\|\nabla^k\big(u^\epsilon_t+ u^\epsilon\cdot \nabla u^\epsilon
  +A\gamma (\rho^\epsilon)^{\gamma-2}\nabla \rho^\epsilon\big)\big\|_{L^2(\mathbb{R}^3)}\\[2pt]
&\quad +C\big(\big\|\nabla^k\big(\nabla (\rho^\epsilon)^{\delta-1}\nabla u^\epsilon\big)\big\|_{L^2(\mathbb{R}^3)}
   +\big\|\nabla^k\big(\nabla^2 (\rho^\epsilon)^{\delta-1}u^\epsilon\big)\big\|_{L^2(\mathbb{R}^3)}\big)
   \qquad \text{for $k=0,1$},\\
\displaystyle
& \big\|\nabla^{2}\big((\rho^\epsilon)^{\delta-1}   \nabla^2 u^\epsilon\big)\big\|_{L^2(\mathbb{R}^3)}\\
&\leq C\epsilon^{-1} \big\|(\rho^\epsilon)^{\delta-1}
  \nabla^2\big((\rho^\epsilon)^{1-\delta}(u^\epsilon_t+ u^\epsilon\cdot \nabla u^\epsilon)
     + (\rho^\epsilon)^{-\delta}\nabla p(\rho^\epsilon)\big)\big\|_{L^2(\mathbb{R}^3)}\\[2pt]
&\quad +C\big\|(\rho^\epsilon)^{\delta-1}\nabla^2(\nabla (\log \rho^\epsilon) \, \nabla u^\epsilon)\big\|_{L^2(\mathbb{R}^3)}\\[2pt]
&\quad +C\big(\big\|\nabla  (\rho^\epsilon)^{\delta-1} \nabla^3 u^\epsilon\big\|_{L^2(\mathbb{R}^3)}
    +\big\|\nabla^2  (\rho^\epsilon)^{\delta-1} \nabla^2 u^\epsilon\big\|_{L^2(\mathbb{R}^3)}\big);\\
\end{split}
\end{equation*}

\item when $\delta=1$   in \cite{sz3},  $T^{\epsilon}=O(\epsilon)$, and
\begin{equation*}\begin{split}
\displaystyle
\big\|\nabla^{k+2}u^\epsilon\big\|_{L^2(\mathbb{R}^2)}
\leq &\, C\epsilon^{-1} \big\|\nabla^k\big(u^\epsilon_t+ u^\epsilon\cdot \nabla u^\epsilon+2\nabla \rho^\epsilon\big)\big\|_{L^2(\mathbb{R}^2)}\\
  &\,+C\big\|\nabla^k\big(\nabla (\log \rho^\epsilon)\,\nabla u^\epsilon\big)\big\|_{L^2(\mathbb{R}^2)}\qquad \mbox{for $k=0, 1, 2$}.
\end{split}
\end{equation*}
\end{itemize}
These indicate that the existing frameworks do not seem to work for the vanishing viscosity limit problem.
Thus, in order to study  the vanishing viscosity limit of strong solutions in the whole space
from the Navier-Stokes to the Euler equations for compressible flow with initial vacuum in some open set
or at the far-field, new ideas are required.

As far as we know, there are only a few results on the inviscid limit problem
for the multidimensional compressible viscous flow with vacuum.
Recently, for the case $\delta=1$ in (\ref{fandan}),
by using the following structure to control the behavior of the fluids velocity $u^\epsilon$:
\begin{equation}\label{zhenji}
u^\epsilon_t+u^\epsilon\cdot\nabla u^\epsilon +\frac{2}{\gamma-1}c^\epsilon\nabla c^\epsilon
+\epsilon Lu^\epsilon=\epsilon \nabla(\log \rho^\epsilon) \, Q(u^\epsilon)
\end{equation}
with
$Q(u)=2\alpha D(u)+\beta\text{div} u\,\mathbb{I}_2$,
the following uniform estimate was obtained in Ding-Zhu \cite{ding}:
\begin{equation*}
\begin{split}
&\sup_{0\leq t\leq T_*}\left(\big\|\left(c^\epsilon, u^\epsilon\right)(t,\cdot)\big\|^2_{H^3(\mathbb{R}^2)}
  +\big\|\nabla \log \rho^\epsilon(t,\cdot)\big\|^2_{H^1(\mathbb{R}^2)}
+\epsilon \big\|\nabla^3 \log \rho^\epsilon(t,\cdot)\big\|^2_{L^2(\mathbb{R}^2)}\right)\\
&\,\,\, +\epsilon
\int^{T_*}_0\big\|\nabla^4 u^\epsilon(t,\cdot)\big\|^2_{L^2(\mathbb{R}^2)}\text{d} t\leq  C_0
\end{split}
\end{equation*}
for some constant $C_0=C_0(A,\gamma,\alpha,\beta,\rho^\epsilon_0,u^\epsilon_0)>0$
and  $T_*>0$ independent of $\epsilon$.
Based on this fact, the corresponding inviscid limit problem for strong solutions of the shallow water  equations in $\R^2$,
{\it i.e.} in (\ref{eq:1.1}),
$$
\gamma=2,\qquad \mu=\alpha \rho,\qquad \lambda= \beta \rho,
$$
was studied under the assumption that $\rho^\epsilon\rightarrow 0$ as $ |x|\rightarrow \infty$.
Some related results can also be found in Geng-Li-Zhu \cite{zhu}.

In this paper, we first observe  a {\it quasi-symmetric hyperbolic--singular elliptic coupled structure} of system \eqref{eq:1.1} with
some singular terms of the first and second orders (see (\ref{eq:cccq2})).
Based on this, we prove that the inviscid flow (\ref{eq:1.1E})
can be regarded as the viscous flow (\ref{eq:1.1}) with vanishing physical viscosities
for the regular solutions  with far-field vacuum in the sense of Definition \ref{d1},
when the viscous stress tensor is of form (\ref{eq:1.3})--(\ref{10000}) for the most physical case $0<\delta<1$.
Our analysis is based mainly on the following structure of $u^\epsilon$:
\begin{equation}\label{sample}
u^\epsilon_t+u^\epsilon\cdot\nabla u^\epsilon +\frac{2}{\gamma-1}c^\epsilon\nabla c^\epsilon
=-\epsilon (\rho^\epsilon)^{\delta-1}  Lu^\epsilon+\frac{2\delta\epsilon}{\delta-1}
(\rho^\epsilon)^{\frac{\delta-1}{2}} \nabla  (\rho^\epsilon)^{\frac{\delta-1}{2}}\, Q(u^\epsilon),
\end{equation}
which is a special quasi-linear parabolic system with some singular coefficients
and source terms  near the vacuum state.
However, compared with the structure of equations \eqref{zhenji} which controls
the behavior of $u^\epsilon$ for the case $\delta=1$ in \cite{ding},
in order to make sure that the estimates and the life span of the solutions
are independent of $\epsilon$, several new difficulties arise:
\begin{enumerate}
\item[\rm (i)]
The source term contains a strong singularity, since
$$
2(\rho^\epsilon)^{\frac{\delta-1}{2}} \nabla (\rho^\epsilon)^{\frac{\delta-1}{2}}\, Q(u^\epsilon)
=(\delta-1)(\rho^\epsilon)^{\delta-1}\nabla (\log \rho^\epsilon) \, Q(u^\epsilon),
$$
whose behavior becomes more singular than that of $\nabla(\log \rho^\epsilon)\, Q(u^\epsilon)$ in \cite{ding}
when  $\rho^\epsilon \rightarrow 0$ as $|x|\rightarrow \infty$, since $\delta<1$.
In fact, the time evolution of $\nabla (\rho^\epsilon)^{\frac{\delta-1}{2}}$ can be controlled
by a symmetric hyperbolic system  with a second-order singular
term $(\rho^\epsilon)^{\frac{\delta-1}{2}}\nabla \text{div} u^\epsilon$ (see (\ref{kuzxc})).
This does not appear in the uniform estimates on $\nabla \log \rho^\epsilon$
for the shallow water equations in \cite{ding};
\item[\rm (ii)]
The coefficient $(\rho^\epsilon)^{\delta-1}$ in front of the Lam\'e operator $L$ tends
to $\infty$ when $\rho^\epsilon \rightarrow 0$,
rather than  $1$ as in \cite{ding} (see \eqref{zhenji}--\eqref{sample}), as $|x|\rightarrow \infty$.
Then we need to pay additional attention to make sure that $(\rho^\epsilon)^{\delta-1} Lu^\epsilon$
is well defined in some weighted
functional space.
\end{enumerate}

We believe that the methodology developed in this paper
will also provide a better understanding for other related vacuum problems for the degenerate viscous flow
in a more general framework, such as the inviscid limit problem for multidimensional entropy weak solutions
in the whole space.

The rest of this paper is divided into seven sections.
In \S 2, we first introduce the notion of regular solutions of the Cauchy problem for
the compressible Navier-Stokes equations (\ref{eq:1.1})
with far-field vacuum and state our main results.

In \S 3, we reformulate the highly degenerate equations  (\ref{eq:1.1}) into
a trackable system (see (\ref{eq:cccq2}) below),
which consists of a symmetric hyperbolic system for $\nabla (\rho^\epsilon)^{\frac{\delta-1}{2}}$
but with a possible singular second-order term of the fluid velocity:
$\frac{\delta-1}{2} (\rho^\epsilon)^{\frac{\delta-1}{2}}\nabla \text{div} u^\epsilon$,
and  a {\it quasi-symmetric hyperbolic}--{\it singular elliptic}  coupled system  for $(c^\epsilon,u^\epsilon)$,
which  contains some possible singular terms of the first order such
as $(\rho^\epsilon)^{\frac{\delta-1}{2}}\nabla (\rho^\epsilon)^{\frac{\delta-1}{2}}\cdot \nabla u^\epsilon$.

In \S 4, we consider the well-posedness with far-field vacuum of the corresponding Cauchy problem of
the reformulated system (\ref{eq:cccq2}) through an elaborate linearization and approximation process,
whose life span is uniformly positive with respect to $\epsilon$.
Moreover, we obtain some  uniform  estimates of $(c^\epsilon,u^\epsilon)$
in $H^3(\mathbb{R}^3)$ that are independent of $\epsilon$.
Denote $h^\epsilon:=(A\gamma)^{\frac{1-\delta}{2(\gamma-1)}}(c^\epsilon)^{\frac{\delta-1}{\gamma-1}}$.
These estimates are achieved by the following five steps:
\begin{enumerate}
\item[1.]
In \S 4.1,  a  uniform elliptic  operator
$$
\frac{(\gamma-1)^2}{4}\big((h^\epsilon)^2+\nu^2\big)Lu^\epsilon
$$
with artificial viscosity coefficients is added to the momentum equations for sufficiently small constant $\nu>0$
so that the global well-posedness of the approximate solutions of the corresponding  linearized problem (\ref{li4})
is established
for $(h^\epsilon, c^\epsilon,u^\epsilon)$ with initial data
\begin{equation*}
 (h^\epsilon,  c^\epsilon,u^\epsilon)|_{t=0}=(h^\epsilon_0, c^\epsilon_0,u^\epsilon_0)(x)
 =((A\gamma)^{-\frac{\iota}{2}}(c^\epsilon_0+\eta)^{\iota},  c^\epsilon_0, u^\epsilon_0)(x),
\end{equation*}
where $c^\epsilon_0=c^\epsilon(0,x)=\sqrt{A\gamma}(\rho^{\epsilon}_0)^{\frac{\gamma-1}{2}}$,
$\iota=\frac{\delta-1}{\gamma-1}<0$,
and $\eta\in (0,1]$ is some constant.

\item[2.]  In \S 4.2, we obtain the uniform  estimates of
$\big(c^\epsilon, u^\epsilon\big)$ in $H^3$ with respect to $(\nu,\eta,\epsilon)$ for
the linearized problem (\ref{li4}).

\item[3.] In \S 4.3,
the approximate solutions of the Cauchy problem of the reformulated  nonlinear system (\ref{middle})
are established by an iteration scheme and the conclusions obtained in \S 4.1--\S4.2,
whose life spans are uniformly positive with respect to $(\nu,\eta, \epsilon)$.

\item[4.]
In \S 4.4, based on the conclusions of \S 4.3, we recover the solution of the nonlinear reformulated
problem (\ref{li4111}) without artificial viscosity by passing to the limit as $\nu \rightarrow 0$.

\item[5.]
In \S 4.5, based on the conclusions of \S 4.4, we recover the solution of the nonlinear reformulated
problem (\ref{eq:cccq2})--(\ref{farnew}) allowing the vacuum state in the far-field
by passing to the limit as $\eta \rightarrow 0$.
\end{enumerate}

Then  in \S 5,  we show that the uniform energy estimates of the reformulated problem obtained in \S 4
indeed imply the desired uniform energy and life span estimates of the original problem.
In \S 6, we establish the vanishing viscosity limit from the degenerate viscous flow
to the inviscid flow with far-field vacuum.
\S 7 is devoted to a non-existence theory for global regular solutions
with $L^\infty$ decay of $u^\epsilon$.
In the appendix, we list some basic lemmas that are used in our proof.
It is worth pointing out that our framework in this paper
can be applied to other physical dimensions, say $1$ and $2$, via some minor modifications.

\section{Main Theorems}
In this section, we state our main results.
Throughout this paper from now on, we adopt the following simplified notation;
most of it is for the standard homogeneous and inhomogeneous Sobolev spaces:
\begin{equation*}\begin{split}
 & \|f\|_s=\|f\|_{H^s(\mathbb{R}^3)},\quad |f|_q=\|f\|_{L^q(\mathbb{R}^3)},\quad \|f\|_{m,q}=\|f\|_{W^{m,q}(\mathbb{R}^3)},\\[5pt]
 & |f|_{C^k}=\|f\|_{C^k(\mathbb{R}^3)},\quad  \|f\|_{XY(t)}=\| f\|_{X([0,t]; Y(\mathbb{R}^3))},\\[5pt]
 & \|(f,g)\|_X=\|f\|_X+\|g\|_X, \quad   X([0,T]; Y)=X([0,T]; Y(\mathbb{R}^3)),   \\[5pt]
 &  D^{k,r}=\{f\in L^1_{\rm loc}(\mathbb{R}^3)\,:\,|f|_{D^{k,r}}=|\nabla^kf|_{r}<\infty\},\\[5pt]
 &D^k=D^{k,2},  \ \ D^{1}=\{f\in L^6(\mathbb{R}^3)\,:\, |f|_{D^1}= |\nabla f|_{2}<\infty\},\\[5pt]
 & |f|_{D^1}=\|f\|_{D^1(\mathbb{R}^3)},\quad  \|f\|_{X_1 \cap X_2}=\|f\|_{X_1}+\|f\|_{X_2}.
\end{split}
\end{equation*}
 A detailed description of the homogeneous Sobolev spaces  can be found in  Galdi \cite{gandi}.

Now we introduce a proper class of solutions, called regular solutions, of the Cauchy problem (\ref{eq:1.1})--(\ref{10000})
with (\ref{initial})--(\ref{far}).

\begin{definition}\label{d1}
Let $T> 0$.
A solution $(\rho^{\epsilon},u^{\epsilon})$ of the Cauchy problem \eqref{eq:1.1}--\eqref{10000}
with \eqref{initial}--\eqref{far} is called a regular solution in $[0,T]\times \mathbb{R}^3$
if $(\rho^{\epsilon},u^{\epsilon})$  is a weak solution in the sense of distributions
and satisfies the following regularity properties{\rm :}
\begin{enumerate}
\item[\rm (i)]
$\rho^{\epsilon}>0,\quad   c^{\epsilon}\in C([0,T]; H^3),
 \quad  \nabla (\rho^{\epsilon})^{\frac{\delta-1}{2}}\in C([0,T]; L^\infty \cap D^2)${\rm ;}

\smallskip
\item[\rm (ii)] $u^{\epsilon}\in C([0,T]; H^3_{\rm loc})\cap  L^\infty([0,T]; H^3),
   \ \ (\rho^{\epsilon})^{\frac{1-\delta}{2}}u^{\epsilon} \in C([0,T]; H^3)$, \\
$(\rho^{\epsilon})^{\frac{\delta-1}{2}}\nabla u^{\epsilon} \in L^2([0,T]; H^3),
   \quad  (\rho^{\epsilon})^{\frac{1-\delta}{2}}u^{\epsilon}_t \in L^\infty([0,T]; H^1)\cap L^2([0,T]; H^2)$.
\end{enumerate}
\end{definition}

\begin{remark}\label{r1}
It follows from Definition {\rm \ref{d1}} and the Gagliardo-Nirenberg inequality
that $\nabla (\rho^{\epsilon})^{\frac{\delta-1}{2}}$ $\in L^\infty$,
which means that the vacuum occurs only in the far-field.
According to the analysis of the structure of system \eqref{eq:1.1} shown in \S {\rm 3} below,
the regular solutions defined above not only select the velocity in a physically reasonable
way when the density approaches zero in the far-field $($see Remark {\rm \ref{zhunbei3}}$)$
but also make the problem trackable through an elaborate linearization and
approximation process $($see \S {\rm 4}$)$.
\end{remark}

Now we are ready to state our first result on the existence and uniform estimates
with respect to $\epsilon$.
Denote the total mass and the total energy in $\R^3$ respectively as
\begin{align*}
m(t):=\int \rho^{\epsilon}  (t, x)\,{\rm d}x,\qquad
E(t):=\int \Big(\frac{1}{2}\rho^{\epsilon} |u^{\epsilon}|^2 +\frac{A(\rho^{\epsilon})^\gamma}{\gamma-1} \Big)(t,x)\,{\rm d} x.
\end{align*}

\begin{theorem}[Existence and Uniform Estimates]\label{th2}
Let $\epsilon \in (0,1]$, and let the physical parameters $(\gamma,\delta, \alpha,\beta)$ satisfy
\begin{equation}\label{canshu}
\gamma>1,\quad 0<\delta<1, \quad \alpha>0, \quad \alpha+\beta  \geq 0.
\end{equation}
Let the initial data $(\rho^{\epsilon}_0, u^{\epsilon}_0)$ satisfy
\begin{equation}\label{th78}
\begin{split}
&\rho^{\epsilon}_0>0, \quad  (c^\epsilon_0, u^{\epsilon}_0)\in H^3,
\quad \epsilon^{\frac{1}{2}}\nabla (\rho^{\epsilon}_0)^{\frac{\delta-1}{2}} \in  D^1\cap D^2,
\quad  \epsilon^{\frac{1}{4}}\nabla (\rho^{\epsilon}_0)^{\frac{\delta-1}{4}} \in  L^4,
\end{split}
\end{equation}
and let the quantity:
\begin{equation}\label{th78--1}
\begin{split}
\mathcal{E}_0=\big\|(c^{\epsilon}_0, u^{\epsilon}_0)\big\|_3
  +\epsilon^{\frac{1}{2}}\big\|\nabla (\rho^{\epsilon}_0)^{\frac{\delta-1}{2}}\big\|_{D^1\cap D^2}
+  \epsilon^{\frac{1}{4}}\big|\nabla (\rho^{\epsilon}_0)^{\frac{\delta-1}{4}}\big|_4
\end{split}
\end{equation}
be uniformly bounded with respect to $\epsilon$.
Then there exist $T_*>0$ and $C_0>0$ independent of $\epsilon$ such that
there exists a unique regular solution $(\rho^{\epsilon}, u^{\epsilon})$ of the Cauchy problem \eqref{eq:1.1}--\eqref{10000}
with \eqref{initial}--\eqref{far} in $ [0,T_*]\times \mathbb{R}^3$
so that the following estimates hold for all $t\in [0, T_*]${\rm :}
\begin{align}
&\sup_{0\leq t\leq T_*}\left(\left\| c^{\epsilon}(t,\cdot)\right\|^2_{3}+\epsilon
\big\|\nabla (\rho^{\epsilon})^{\frac{\delta-1}{2}}(t,\cdot)\big\|^2_{D^1\cap D^2}\right)
+\text{\rm ess}\sup_{0\leq t\leq T_*}\|u^{\epsilon}(t,\cdot)\|^2_{3}\nonumber\\
&+\epsilon
\int^{T_*}_0\sum_{i=1}^4|(\rho^\epsilon)^{\frac{\delta-1}{2}}\nabla^i u^\epsilon(t,\cdot)|^2_{2}\text{\rm d} t
\leq C_0. \label{shangjie1}
\end{align}
Moreover, if $m(0)<\infty$ is additionally assumed,
then the regular solution $(\rho^\epsilon, u^\epsilon)$ obtained above  has  finite total mass and total energy:
\begin{equation*}
\begin{split}
m(t)=m(0)=\int \rho^{\epsilon}_0(x)\,{\rm d}x  <\infty,\qquad
 E(t)\le E(0)=\int \big(\frac{1}{2}\rho^{\epsilon}_0 |u^{\epsilon}_0|^2 +\frac{P(\rho^{\epsilon}_0)}{\gamma-1} \big) \, \text{\rm d}x<\infty
\end{split}
\end{equation*}
for $0\leq t \leq T_*$.
\end{theorem}

\begin{remark}\label{r2}
Regarding the above initial assumption \eqref{th78}, we remark that \eqref{th78}
identifies a class of admissible initial data that provide the unique solvability of
problem \eqref{eq:1.1}--\eqref{10000} with \eqref{initial}--\eqref{far}.
For example,
\begin{equation}\label{Example}
\rho^\epsilon_0(x)=f(x)\chi(\frac{x}{10})+\frac{1}{1+|x|^{2a}},\qquad\,\,  u^\epsilon_0(x)\in H^3(\mathbb{R}^3),
\end{equation}
where
$$
0\leq f(x)\in C^3(\mathbb{R}^3),\qquad \frac{3}{2(\gamma-1)}< a<\frac{1}{2(1-\delta)},
$$
and $\chi(x)\in C^\infty_c(\mathbb{R}^3)$ is  a  truncation function  satisfying
\begin{equation}\label{eq:2.6-77A}
0\leq \chi(x) \leq 1, \qquad  \chi(x)=
 \begin{cases}
1 \;\qquad  \text{if} \ \ |x|\leq 1,\\[3pt]
0   \ \ \ \ \ \ \    \text{if} \ \   |x|\geq 2.
 \end{cases}
 \end{equation}
Moreover, condition $\epsilon^{\frac{1}{4}}\nabla (\rho^{\epsilon}_0)^{\frac{\delta-1}{4}} \in  L^4$
is only  used in the approximation process of the initial data from the non-vacuum flow to the flow with
a far-field vacuum
$($see \eqref{appnonfar} below$)$, which is not used for our energy estimates.
We believe that this condition could be removed if an improved approximation scheme were developed.
\end{remark}

Now, based on the well-posedness results in Theorems \ref{thmakio} and \ref{th2} for both viscous and inviscid flows,
we can show the following asymptotic behavior as $\epsilon \rightarrow 0$:

\begin{theorem}[Inviscid Limit]\label{th3A}
Let $\epsilon\in (0,1]$ and \eqref{canshu}--\eqref{th78--1} hold.
If we additionally assume that there exist functions  $(\rho_0(x),u_0(x))$ defined in $\mathbb{R}^3$
so that
\begin{equation}\label{initialrelation1}
\lim_{\epsilon\rightarrow 0}\big|(c^\epsilon_0-c_0,u^\epsilon_0-u_0)\big|_{2}=0,
\end{equation}
then there exist functions $(\rho,u)$ defined in $[0,T_*]\times \mathbb{R}^3$ such that
\begin{equation}\label{jkkab1}
\sup_{0\leq t \leq T_2}\big\| (c,u)(t,\cdot)\big\|^2_{3}  \leq C  \qquad\mbox{for some constant $C>0$},
\end{equation}
and, for any  constant  $s'\in [0, 3)$,
\begin{equation}\label{shou1A}
\lim_{\epsilon \to 0}\sup_{0\leq t\leq T_*}
\big\|(c^{\epsilon}-c, u^\epsilon-u)(t,\cdot)\big\|_{H^{s'}_{\rm loc}(\mathbb{R}^3)}=0.
\end{equation}
Furthermore, $(\rho, u)$ is  the unique regular
solution of the Cauchy problem \eqref{eq:1.1E}
and \eqref{winitial}
in Theorem $\ref{thmakio}$.
\end{theorem}

The above theorem implies the following result:
\begin{corollary}\label{th3}
Let $\epsilon\in (0,1]$ and \eqref{canshu}--\eqref{th78--1} hold.
Suppose that $(\rho^\epsilon, u^\epsilon)$ is the regular solution of problem \eqref{eq:1.1}--\eqref{10000}
with \eqref{initial}--\eqref{far} in Theorem $\ref{th2}$,  and $(\rho, u)$ is the regular
solution of \eqref{eq:1.1E} with
\eqref{winitial}
in Theorem $\ref{thmakio}$. If
\begin{equation}
(\rho^{\epsilon}, u^{\epsilon})|_{t=0}=(\rho, u)|_{t=0}=(\rho_0, u_0),
\end{equation}
then there exists $T_{*}>0$ independent of $\epsilon$ such that
$(\rho^{\epsilon}, u^{\epsilon})$
converges to $(\rho, u)$ in $[0,T_*]\times \mathbb{R}^3$ as $\epsilon\to 0$ in the sense of distributions.
Moreover,  for any constant  $s'\in [0, 3)$, we also have \eqref{shou1A}.
\end{corollary}

Moreover, we can also obtain the following corollary:

\begin{corollary}\label{example}
Let $\delta\in (0,1)$, and $(\rho_0,u_0)$ satisfy  \eqref{th78asd}.
Then, for every $\epsilon \in (0,1]$, there exist $(\rho^\epsilon_0,u^\epsilon_0)$ satisfying
assumption \eqref{th78} such that
$$
\big\|(c^{\epsilon}_0, u^{\epsilon}_0)\big\|^2_3
+\epsilon \big\|\nabla (\rho^\epsilon_0)^{\frac{\delta-1}{2}}\big\|^2_{D^1\cap D^2}
+\epsilon^{\frac{1}{2} }\big|\nabla (\rho^\epsilon_0)^{\frac{\delta-1}{4}}\big|^2_4 \leq C_0
$$
for some constant $C_0>0$  independent of $\epsilon$ and
$$
\lim_{\epsilon \to 0}\big\|(c^{\epsilon}_0-c_0, u^\epsilon_0
-u_0)\big\|_{3}=0.
$$
Moreover, the corresponding Cauchy problem \eqref{eq:1.1E} and
\eqref{winitial}
can be regarded as a limit problem
of the Cauchy problem \eqref{eq:1.1}--\eqref{10000} with \eqref{initial}--\eqref{far}
as $\epsilon \rightarrow 0$ in the sense of \eqref{shou1A}.
\end{corollary}

Naturally, a further question is whether the solution obtained in Theorem \ref{th2}
can be extended globally in time under the assumption that the initial data is a small perturbation around
some background solution.
Under the assumption that
$$
\rho^\epsilon_0(x)\rightarrow \overline{\rho} \qquad \mbox{as $|x|\rightarrow 0$}
$$
for some constant $\overline{\rho}>0$,
the classical theories, no matter whether for the constant viscous flow ({\it e.g.} \cites{HX1, Wangweike, mat})
or the degenerate viscous flow {{\it e.g.} \cite{decayd}) away from a vacuum,
all indicate that the corresponding background solution for $(\rho^\epsilon,u^\epsilon)$ must be $(\overline{\rho},0)$
with
the following large-time behavior:
\begin{equation}\label{eq:2.15}
\limsup_{t\rightarrow \infty} \big|u^\epsilon(t,\cdot)\big|_{\infty}=0.
\end{equation}
However, when the vacuum appears, the situation for the degenerate viscous flow
is somewhat surprising, since such an extension seems impossible when the initial momentum is nonzero.
More precisely, denote the flow momentum by
$$
\mathbb{P}(t):=\int (\rho^\epsilon u^\epsilon)(t,x)\,{\rm d}x.
$$
\begin{theorem}\label{th:2.20}
 Assume that $0<m(0)<\infty$, $|\mathbb{P}(0)|>0$,  $\epsilon\geq 0$, and \eqref{canshu} hold.
Then there is no global  regular  solution $(\rho^\epsilon,u^\epsilon)$ in the sense of Theorem   {\rm\ref{thmakio}} or {\rm\ref{th2}}
satisfying \eqref{eq:2.15}.
\end {theorem}

Finally, we show that the condition $\epsilon^{\frac{1}{2}}\nabla (\rho^{\epsilon}_0)^{\frac{\delta-1}{2}}\in D^{1}$
in \eqref{th78} can be replaced by other conditions such as
$\epsilon^{\frac{1}{2}}\nabla (\rho^{\epsilon}_0)^{\frac{\delta-1}{2}} \in  L^q \cap D^{1,3}$ for any   constant $q>3$:

\begin{theorem}\label{generaltheorem}
Let  $\epsilon\in(0, 1]$ and \eqref{canshu} hold. Let $q\in (3,\infty]$ be a fixed constant.
Assume that the initial data $( \rho^{\epsilon}_0, u^{\epsilon}_0)$ satisfy
\begin{equation}\label{th78-general}
\rho^{\epsilon}_0>0,\quad  (c^{\epsilon}_0, u^\epsilon_0)\in H^3, \quad
\epsilon^{\frac{1}{2}}\nabla (\rho^{\epsilon}_0)^{\frac{\delta-1}{2}} \in  L^q\cap D^{1,3}\cap D^2, \quad
\epsilon^{\frac{1}{4}}\nabla (\rho^{\epsilon}_0)^{\frac{\delta-1}{4}} \in  L^6,
\end{equation}
and
\begin{equation}\label{th78--1-general}
\begin{split}
\big\|(c^{\epsilon}_0, u^\epsilon_0)\big\|_3
+\epsilon^{\frac{1}{2}}\big\|\nabla (\rho^{\epsilon}_0)^{\frac{\delta-1}{2}}\big\|_{L^q\cap D^{1,3}\cap D^2}
+  \epsilon^{\frac{1}{4}}\big|\nabla (\rho^{\epsilon}_0)^{\frac{\delta-1}{4}}\big|_6
\end{split}
\end{equation}
is uniformly bounded with respect to $\epsilon$.
Then there exist $T_*>0$ and $C>0$, both independent of $\epsilon$,
such that  the  unique regular solution $(\rho^{\epsilon}, u^{\epsilon})$ of the Cauchy problem \eqref{eq:1.1}--\eqref{10000}
with \eqref{initial}--\eqref{far} exists in $[0,T_*]\times \mathbb{R}^3$
with the following estimates{\rm :}
\begin{align}
&\sup_{0\leq t\leq T_*}\left(\big\|c^{\epsilon}(t,\cdot)\big\|^2_{3}+\epsilon
\big\|\nabla (\rho^{\epsilon})^{\frac{\delta-1}{2}}(t,\cdot)\big\|^2_{L^q\cap D^{1,3}\cap D^2}\right)
+\text{\rm ess}\sup_{0\leq t\leq T_*}\|
u^{\epsilon}(t,\cdot)\|^2_{3}\nonumber\\
& +\epsilon
\int^{T_*}_0\sum_{i=1}^4|(\rho^\epsilon)^{\frac{\delta-1}{2}}\nabla^i u^\epsilon(t,\cdot)|^2_{2}\,\text{\rm d} t\leq C_0.
\end{align}
 Moreover, under proper changes to the corresponding assumptions,
 the results obtained in Theorems  {\rm \ref{th3}--\ref{th:2.20}} and Corollaries {\rm \ref{th3}}--{\rm \ref{example}}
still hold.
\end{theorem}

The proof of this theorem is similar to those of Theorems  {\rm \ref{th3}--\ref{th:2.20}}
and Corollaries {\rm \ref{th3}}--{\rm \ref{example}}.
Thus, we omit its details.

\begin{remark}
The conditions in \eqref{th78-general} identify a class of admissible initial data that
provide the unique solvability of problem \eqref{eq:1.1}--\eqref{10000} with \eqref{initial}--\eqref{far}
such as the one shown in \eqref{Example}--\eqref{eq:2.6-77A} with
$$
0\leq f(x)\in C^3(\mathbb{R}^3), \qquad \frac{3}{2(\gamma-1)}< a<\frac{1-3/q}{1-\delta}.
$$
\end{remark}

We remark that our framework in this paper is applicable
to other physical dimensions, say $1$ and $2$,
via some minor modifications.

\section{Reformulation}
In this section, we reformulate the highly degenerate equations  (\ref{eq:1.1})
into a  trackable system that consists of a symmetric hyperbolic system
with a possibly singular second-order term of the fluid velocity,
and  a {\it quasi-symmetric hyperbolic}--{\it singular elliptic}  coupled system.
For simplicity, throughout this section,  we denote $(\rho^\epsilon,u^\epsilon,c^\epsilon, \psi^\epsilon,h^\epsilon)$
by $(\rho,u,c,\psi,h)$ and $(\rho^\epsilon_0,u^\epsilon_0, c^\epsilon_0, \psi^\epsilon_0,h^\epsilon_0)$
by $(\rho_0,u_0,c_0,\psi_0,h_0)$, respectively.

\subsection{New variables}
Let $T>0$ be a fixed finite time.
For $\delta\in (0,1)$,  when $\rho(t,x)>0$ for $(t,x)\in [0,T]\times \mathbb{R}^3$,
$(\ref{eq:1.1})_2$ can be formally rewritten as
\begin{equation}\label{qiyi}
u_t+u\cdot\nabla u +\frac{2A\gamma}{\gamma-1} \rho^{\frac{\gamma-1}{2}} \nabla \rho^{\frac{\gamma-1}{2}}
+\epsilon\rho^{\delta-1}  Lu=\frac{2\delta\epsilon}{\delta-1}\rho^{\frac{\delta-1}{2}} \nabla \rho^{\frac{\delta-1}{2}}\, Q(u),
\end{equation}
where $\rho^{\frac{\gamma-1}{2}}$ is a constant multiple of the local sound speed $c=\sqrt{p'(\rho)}$.
Then, in order to govern velocity $u$ via the above quasilinear parabolic equations with far-field vacuum,
\begin{enumerate}
\item[\rm (i)]
It is necessary to control the behavior of the special source term
$$
\frac{2\delta\epsilon}{\delta-1}\rho^{\frac{\delta-1}{2}} \nabla \rho^{\frac{\delta-1}{2}}\, Q(u)
$$
since $\delta-1<0$ and
$$
\rho(t,x)\rightarrow 0 \qquad \text{as $|x|\rightarrow \infty\, $ for $t\in [0,T]$};
$$
\item[\rm (ii)]
Note that the coefficient $\rho^{\delta-1}$ in front of the Lam\'e operator $L$
tends to $\infty$ as $\rho\rightarrow 0$ in the far-field,
so that it is necessary to show that $\rho^{\delta-1} Lu$ is well defined at least
in the space of continuous functions,
in order that the solution obtained is regular when $t>0$.
\end{enumerate}
Therefore, the three quantities
$$
\rho^{\frac{\gamma-1}{2}}, \quad \nabla \rho^{\frac{\delta-1}{2}}, \quad \rho^{\delta-1} Lu
$$
play a significant role in our analysis on the regularity of the fluid velocity $u$.
In fact, in terms of the fluid velocity $u$, and
\begin{equation}\label{bianliang}
c=\sqrt{A\gamma}\rho^{\frac{\gamma-1}{2}},
\quad \psi=\nabla \rho^{\frac{\delta-1}{2}}=\nabla h=(\psi^{(1)},\psi^{(2)},\psi^{(3)}),
\end{equation}
system (\ref{eq:1.1}) can be rewritten as the following enlarged system:
\begin{equation}
\begin{cases}\label{eq:cccq}
\displaystyle
\psi_t+\nabla (u\cdot \psi)+\frac{\delta-1}{2}\psi\,\text{div} u +\frac{\delta-1}{2} h\nabla \text{div} u=0,\\[8pt]
c_t+u\cdot \nabla c+\frac{\gamma-1}{2}c\,\text{div} u=0,\\[8pt]
 \displaystyle
 u_t+u\cdot\nabla u +\frac{2}{\gamma-1}c\nabla c+\epsilon h^2 Lu=\frac{2\delta\epsilon }{\delta-1}h \psi  Q(u),
 \end{cases}
\end{equation}
where
$$
h=\rho^{\frac{\delta-1}{2}}=(A\gamma)^{-\frac{\iota}{2}}c^\iota, \qquad \iota=\frac{\delta-1}{\gamma-1}<0.
$$

The initial data are given by
\begin{equation}\label{qwe}
(\psi, c,u)|_{t=0}=(\psi_0, c_0,u_0)(x)
:=(\nabla \rho^{\frac{\delta-1}{2}}_0, \sqrt{A\gamma}\rho^{\frac{\gamma-1}{2}}_0, u_0)(x)\qquad \mbox{for $x\in \mathbb{R}^3$},
\end{equation}
so that
\begin{equation}\label{E:3.3}
(\psi_0, c_0,u_0)\rightarrow (0, 0, 0) \qquad \text{as $|x|\rightarrow \infty$}.
\end{equation}

\subsection{Mathematical structure of the reformulated system}
Now we introduce the desired {\it quasi-symmetric hyperbolic--singular elliptic} coupled structure
in order to deal with
the corresponding inviscid limit problem.

The new system (\ref{eq:cccq}) still seems un-trackable for the purpose of  constructing the regular
solutions with far-field vacuum in $H^3$ under the following initial assumption:
\begin{equation}\label{initial-reformulation}
c_0>0,\quad  (c_0, u_0)\in H^3, \quad  \psi_0 \in  D^1\cap D^2.
\end{equation}

First,  even if $h\nabla \text{div} u$ could be controlled by the singular elliptic operator  $h^{2}Lu$
appearing in the momentum equations,
the special source term $\psi$ can not be controlled by a scalar transport equation,
but by a quasilinear hyperbolic system.
It follows from the definition of $\psi$ that, if $\psi\in D^1\cap D^2$,  then
$\partial_i \psi^{(j)}=\partial_j \psi^{(i)}$ in the sense of distributions for $i,j=1,2,3$.
Thus, $(\ref{eq:cccq})_1$
can be rewritten as
\begin{equation}\label{kuzxc}
\psi_t+\sum_{l=1}^3 B_l \partial_l\psi+B\psi+ \frac{\delta-1}{2} h\nabla \text{div} u=0,
\end{equation}
where $B_l=(b^l_{ij})_{3\times 3}$, for  $i,j,l=1,2,3$,
are symmetric  with
$$
b^l_{ij}=u^{(l)}\quad \mbox{for $i=j$;\qquad \,\,\,  $b^l_{ij}=0$ \quad otherwise},
$$
and $B=(\nabla u)^\top+\frac{\delta-1}{2}\text{div}u \mathbb{I}_3$.
This indicates that the subtle source term $\psi$ could actually be controlled by  the
symmetric hyperbolic system with one possible singular source term $\frac{\delta-1}{2} h \nabla \text{div} u$
near the  vacuum.

Second, in order to make sure that the life span and the corresponding energy estimates of
the regular solutions that we will obtain are uniform with respect to $\epsilon$,
we need to introduce more symmetrization arguments, except the above structure for $\psi$.
In fact, letting  $ U=(c, u)$, according to (\ref{kuzxc}),
we rewrite the Cauchy problem (\ref{eq:cccq})--(\ref{E:3.3}) as
\begin{equation}
\label{eq:cccq2}
\begin{cases}
 \underbrace{\displaystyle \psi_t+\sum_{l=1}^3 B_l(u) \partial_l\psi}_{\text{Symmetric hyperbolic}}
  + \underbrace{B(u)\psi}_{\text{First order}}
  + \underbrace{\frac{\delta-1}{2} (A\gamma)^{-\frac{\iota}{2}}c^\iota  \nabla \text{div} u}_{\text{Singular second  order}}=0,\\[15pt]
\displaystyle
 \underbrace{A_0 U_t+\sum^{3}_{j=1}A_j(U)\partial_j U}_{\text{Symmetric hyperbolic}}
 = \underbrace{-\epsilon F(U)}_{\text{Singular elliptic}}+ \underbrace{\epsilon G( \psi, U),}_{\text{Singular first order}}
 \end{cases}
\end{equation}
with the following initial data:
\begin{equation}\label{initialnew}
 ( \psi, c,u)|_{t=0}=(\psi_0, c_0,u_0)(x)
 =(\nabla\rho^{\frac{\delta-1}{2}}_0, \sqrt{A\gamma}\rho^{\frac{\gamma-1}{2}}_0, u_0)(x)\qquad\,\, \mbox{for $x\in \mathbb{R}^3$},
\end{equation}
so that
\begin{equation}\label{farnew}
(\psi_0, c_0,u_0)\rightarrow (0, 0, 0) \qquad \text{as $|x|\rightarrow\infty$},
\end{equation}
\noindent
\smallskip
where $\displaystyle\partial_l \psi=\partial_{x_l}\psi$,
$\displaystyle \partial_j U=\partial_{x_j}U,  i,j,l=1, 2,3$,
\begin{equation}\label{xishu}
\begin{split}
&A_0=\left(\begin{array}{cc}
1&0\\[8pt]
0&a_1\mathbb{I}_3
\end{array}
\right),\qquad
\displaystyle
A_j=\left(\begin{array}{cc}
u_j&\frac{\gamma-1}{2}c e_j\\[8pt]
\frac{\gamma-1}{2}c e_j^\top &a_1u_j\mathbb{I}_3
\end{array}
\right),\,\,\, j=1,2,3,\\[10pt]
&
F(U)=a_1\left(
\begin{array}{c}
 0\\[5pt]
(A\gamma)^{-\iota}c^{2\iota} Lu
\end{array}\right), \quad
  G(\psi,U)=a_1\left(
\begin{array}{c}
 0\\[5pt]
\frac{2\delta}{\delta-1}(A\gamma)^{-\frac{\iota}{2}}c^\iota \psi \, Q(u)
\end{array}\right),
\end{split}
\end{equation}
with $a_1\equiv:\frac{(\gamma-1)^2}{4}$, $e_1=(1,0,0), e_2=(0,1,0)$, and $e_3=(0,0,1)$.

\begin{remark}
The hyperbolic operators $H=(H^1,H^2)$ for $(\psi,U)${\rm :}
\begin{equation*}\begin{split}
H^1(\psi):=&\,\psi_t+\nabla(u\cdot\psi),\\[8pt]
H^2(U):=&\,H^2(c,u)=\left(
\begin{array}{c}
c_t+u\cdot \nabla c+\frac{\gamma-1}{2}c\,\text{\rm div} u\\[8pt]
u_t+u\cdot\nabla u +\frac{2}{\gamma-1}c\nabla c
\end{array}\right)
\end{split}
\end{equation*}
can be  rewritten into the following  symmetric hyperbolic forms:
\begin{equation}\label{symmetricpart}
\psi_t+\sum_{l=1}^3 B_l \partial_l\psi+(\nabla u)^\top \psi,\qquad
A_0 U_t+\sum^{3}_{j=1}A_j(U)\partial_j U,
\end{equation}
which make the $H^3$ estimates of $(c,u)$ possibly independent of $\epsilon$.
However, not every first-order term in system \eqref{eq:cccq2} has been written into the
symmetric structure.
This is the reason why we only say that system \eqref{eq:cccq2} satisfies
the {\it quasi-symmetric}  structure, rather than the {\it symmetric} one.
Hence, new
treatments are needed for the possibly singular source terms
$\frac{\delta-1}{2} h\nabla \text{div} u$ and $\epsilon G(\psi, U)$.
\end{remark}

\section{Uniform Energy Estimates for the Reformulated Problem}
This section is devoted to the establishment of the uniform local-in-time well-posedness (with respect to $\epsilon$)
of strong solutions with far-field vacuum of the reformulated Cauchy problem (\ref{eq:cccq2})--(\ref{farnew}).
Moreover, some uniform estimates of
$\displaystyle\big(c^\epsilon, u^\epsilon\big)$ in $H^3$
with respect to $\epsilon$ can also be established.
For simplicity, in this section, we denote $(\rho^\epsilon,u^\epsilon,c^\epsilon, \psi^\epsilon,h^\epsilon)$
by $(\rho,u,c,\psi,h)$, and $(\rho^\epsilon_0,u^\epsilon_0, c^\epsilon_0, \psi^\epsilon_0,h^\epsilon_0)$
by $(\rho_0,u_0,c_0,\psi_0,h_0)$, respectively.

We first give the definition of strong solutions of the Cauchy problem (\ref{eq:cccq2})--(\ref{farnew}).

\begin{definition}\label{strongsoluiton}
Let $T>0$.
A vector function $(\psi,c,u)$ is called a strong solution in $[0,T]\times \mathbb{R}^3$
if $(\psi,c,u)$ is a weak solution of the Cauchy problem \eqref{eq:cccq2}--\eqref{farnew} in $[0,T]\times \mathbb{R}^3$
in the sense of distributions, all derivatives involved in \eqref{eq:cccq2} are regular distributions,
and \eqref{eq:cccq2} holds almost everywhere in $[0,T]\times \mathbb{R}^3$.
\end{definition}

We now state the main result in this section on the well-posedness of the Cauchy problem
for the reformulated system (\ref{eq:cccq2}).

\begin{theorem}\label{th1}
Let \eqref{canshu} hold and $\epsilon\in (0,1]$.
If the initial data $(\psi_0, c_0,  u_0)$ satisfy
\begin{equation}\label{th78qq}
c_0>0,\quad  (c_0, u_0)\in H^3,
\quad \epsilon^{\frac{1}{2}}\psi_0=\epsilon^{\frac{1}{2}}(A\gamma)^{-\frac{\iota}{2}}\nabla c^\iota_0 \in  D^1\cap D^2,
\quad \epsilon^{\frac{1}{4}}\nabla c^{\frac{\iota}{2}}_0 \in  L^4,
\end{equation}
then there exists $T_*>0$ independent of $\epsilon$ such that there is a unique strong
solution $(\psi,c,u)=((A\gamma)^{-\frac{\iota}{2}}\nabla c^\iota,c,u)$
in $[0,T_*]\times \mathbb{R}^3$ of the Cauchy problem \eqref{eq:cccq2}--\eqref{farnew} satisfying
\begin{equation}\label{reformulateregularity}
\begin{split}
&\psi \in C([0,T];D^1\cap D^2),\quad  c\in C([0,T];H^3),\\
& u\in C([0,T]; H^3_{\rm loc})\cap L^\infty([0,T]; H^3), \quad c^\iota \nabla u\in L^2([0,T]; H^3),
\end{split}
\end{equation}
and the following uniform estimates{\rm :}
\begin{align}
&\sup_{0\leq t \leq T_*}\big(\|c(t,\cdot)\|_3^2+\epsilon \|\psi(t,\cdot)\|^2_{D^1\cap D^2}\big)
+ \text{\rm ess}\sup_{0\leq t \leq T_*}\|u(t,\cdot)\|_3^2 \nonumber\\
&+\epsilon \int_0^{T_*}\sum_{i=1}^4
|c^\iota\nabla^i u(t,\cdot)|_2^2\, {\rm d}t \leq C \label{qiyu}
\end{align}
for some positive constant $C=C(\alpha, \beta, A, \gamma, \delta, c_0, \psi_0,u_0)$
that is independent of $\epsilon$.
\end{theorem}

We now prove this theorem in subsequent Sections \S 4.1--\S 4.5.

\subsection{Linearization with one artificial viscosity}\label{linear2}
In order to proceed with the nonlinear problem (\ref{eq:cccq2})--(\ref{farnew}),
we now consider the following linearized problem:
\begin{equation}\label{li4}
\begin{cases}
 h_t+v\cdot \nabla h+\frac{\delta-1}{2}g\,\text{div} v=0,\\[8pt]
 A_0 U_t+\sum^{3}_{j=1}A_j(V)\partial_j U=-\epsilon F(\nu,h,u)+\epsilon G(h,\nabla h, u),\\[8pt]
 (h, c,u)|_{t=0}=(h_0, c_0,u_0)(x)=((A\gamma)^{-\frac{\iota}{2}}(c_0+\eta)^{\iota}, c_0, u_0)(x)
 \qquad \text{for $x\in \R^3$},
\end{cases}
\end{equation}
with
\begin{equation}\label{4.5a}
(h_0, c_0,u_0)(x)\rightarrow ((A\gamma)^{-\frac{\iota}{2}}\eta^{\iota}, 0, 0)
 \,\, \qquad \text{as $|x|\rightarrow \infty$},
\end{equation}
where  $(\nu,\eta, \epsilon)\in (0,1]\times(0,1]\times (0,1]$  are all positive constants,
 \begin{equation}\label{xishu1}
\begin{split}
F(\nu,h,u):=a_1
 \begin{pmatrix}
      0\\[3pt]
      (h^2+\nu^2) Lu
    \end{pmatrix},
  \quad
G(h,\nabla h,u):=a_1
 \begin{pmatrix}
 0\\[3pt]
\frac{2\delta}{\delta-1}h \nabla h \, Q(u)
\end{pmatrix},
\end{split}
\end{equation}
$V=(\varphi,v)$ with $\varphi$ being a given function
and $v=(v^{(1)},v^{(2)}, v^{(3)})\in \mathbb{R}^3$  a given vector respectively,
and $g$ is a given function satisfying:
\begin{equation}\label{vg}
\begin{split}
&(\varphi,v, g)(0,x)=(c_0, u_0,h_0)(x)=(c_0, u_0, (A\gamma)^{-\frac{\iota}{2}}(c_0+\eta)^{\iota})(x),\\
&\varphi \in C([0,T];H^3),\,\,\, \varphi_t \in C([0,T];H^2),\,\,\,
  g\in L^\infty\cap C ([0,T]\times \mathbb{R}^3),\,\,\, \nabla g \in C([0,T];H^2),\\
&  v\in C([0,T]; H^3)\cap L^2([0,T]; H^4),  \,\,\, v_t\in C([0,T]; H^1)\cap L^2([0,T]; H^2)
\quad\,\mbox{for any $T>0$}.
\end{split}
\end{equation}

Note that, due to the complicated structure of system \eqref{eq:cccq2} near the vacuum,
the linear scheme \eqref{li4} is carefully chosen such that this linear problem can be
solved globally in time, and the desired uniform estimates of the solutions can be established.
According to the analysis in \S {\rm 3}, we first hope to keep the symmetric hyperbolic forms
shown in \eqref{symmetricpart} which, in our desired  linear scheme,  are expected to be
\begin{equation}\label{linearsymmetricpart}
\psi_t+\sum_{l=1}^3 B_l(v) \partial_l\psi,\qquad   A_0 U_t+\sum^{3}_{j=1}A_j(V)\partial_j U
\end{equation}
for $\psi:=\nabla h$.

Next, in order to ensure the global well-posedness of $\psi$ and the desired estimates in some positive time,
the singular source term $\frac{\delta-1}{2} (A\gamma)^{-\frac{\iota}{2}}c^\iota \nabla \text{\rm div} u$
in $\eqref{eq:cccq2}_1$
should be handled carefully. Possible ways to linearize this product term are
$$
\frac{\delta-1}{2} g \nabla \text{\rm div} u, \quad
\text{or} \quad  \frac{\delta-1}{2} (A\gamma)^{-\frac{\iota}{2}}c^\iota \nabla \text{\rm div} v,
\quad \text{or} \quad \frac{\delta-1}{2} g \nabla \text{\rm div} v,
$$
where $V$ and $g$ satisfy assumption \eqref{vg}.
For either of the first two choices, the estimates of $\psi$ in $D^1\cap D^2$
will depend on the upper bound of $g$ or $(A\gamma)^{-\frac{\iota}{2}}c^\iota$, {\it i.e.},
the lower bound of $\rho_0$, which is exactly what we want to avoid.
Therefore, for the linear structure of $\psi$, what we can expect  should be
  \begin{equation}
\label{caixiang1}
\displaystyle
\psi_t+\sum_{l=1}^3 B_l(v) \partial_l\psi+B(v)\psi+\frac{\delta-1}{2} g \nabla \text{div} v=0.
\end{equation}
In fact, for this scheme, one can use
\begin{equation}\label{realuse}
\sum_{k=1}^2 \big\|\nabla^k(g\nabla^2 v)\big\|^2_{L^2([0,T];L^2(\mathbb{R}^3))}
\end{equation}
to replace the upper bound of $g$ in the corresponding estimates,
which can ensure that the desired estimates of $\psi$ are independent of the lower bound
of the initial density and $\epsilon$.

Finally, for the choice of the linearization scheme for $U$,
due to the above discussion, there are at least two requirements that should be satisfied.
On one hand, it needs to keep the symmetric form shown in \eqref{linearsymmetricpart}.
On the other hand, for our final aim -- to approximate the nonlinear problem,
the desired regularity in \eqref{realuse} for $(g,v)$ should be verified
by solution $(c^\iota,u)$ of the linear problem, which can only be provided
by the elliptic operator $c^{2\iota}Lu$.
Then it seems that we should consider the following
equations:
  \begin{equation}
\label{eq:cccq-fenxi-A}
\displaystyle
 A_0 U_t+\sum^{3}_{j=1}A_j(V)\partial_j U=-\epsilon F(\nu, h, u) +\epsilon G(g,\psi, v),
\end{equation}
which actually is still a nonlinear system.
However, even if the corresponding Cauchy problem for system \eqref{eq:cccq-fenxi-A}
is assumed to be globally solved,
we still encounter an obvious difficulty for considering the $L^2$ estimate of $u$.
First, it should be pointed out that, in \eqref{caixiang1} and \eqref{eq:cccq-fenxi-A},
the relationship
$$
\psi=(A\gamma)^{-\frac{\iota}{2}}\nabla c^\iota
$$
between $\psi$ and $c$ has been destroyed owing to term  $g \nabla \text{\rm div} v$
in $\eqref{caixiang1}$.
Second, multiplying by $u$ on both sides of the equation for $u$ in \eqref{eq:cccq-fenxi-A}
and integrating by parts
yield
\begin{equation*}
\begin{split}
&\frac{1}{2} \frac{\rm d}{\rm d t}\big|u\big|^2_2
 +\epsilon\alpha \big|(A\gamma)^{-\frac{\iota}{2}}c^\iota\nabla u\big|^2_2
  +\epsilon(\alpha+\beta)\big|(A\gamma)^{-\frac{\iota}{2}}c^\iota\text{\rm div} u\big|^2_2\\
&=-\int\,\big(v\cdot \nabla u +\frac{2}{\gamma-1}\varphi\nabla c
  +2\epsilon (A\gamma)^{-\frac{\iota}{2}} c^{\iota}
   \underbrace{(A\gamma)^{-\frac{\iota}{2}}\nabla c^{\iota}}_{\neq \psi}\, Q(u)
  -\frac{2\delta\epsilon}{\delta-1}g\psi Q(v) \big)\cdot u\, {\rm d}x.
\end{split}
\end{equation*}
However, $(A\gamma)^{-\frac{\iota}{2}}\nabla c^{\iota}\neq \psi$ in this linear scheme,
which means that there is no way to control term $2\epsilon (A\gamma)^{-\iota} c^{\iota} \nabla c^{\iota}$
in the above energy estimates. In order to overcome this difficulty, in  \eqref{li4},
we  first linearize the equation of $h=(A\gamma)^{-\frac{\iota}{2}} c^\iota$ as:
\begin{equation}\label{h}
h_t+v\cdot \nabla h+\frac{\delta-1}{2}g\,\text{\rm div} v=0,
\end{equation}
and then use $h$ to define $\psi=\nabla h$ again. The linearized equations for $u$ are chosen as
$$
u_t+v\cdot\nabla u +\frac{2}{\gamma-1}\varphi\nabla c+ \epsilon (h^2+\nu^2)  Lu
=\frac{2\epsilon \delta}{\delta-1}h \psi Q(u)
$$
for any positive constant $\nu>0$, where
the appearance of $\nu$ is used to compensate the lack of a lower bound for $h$.
From both equation \eqref{h} for $h$  and relation  $\psi=\nabla h$,
we can obtain a linearized equations \eqref{kuzxclinear} for $\psi$ below,
which, luckily, can be shown to be still good enough
to obtain the desired estimates for $\psi$ $($see Lemma {\rm\ref{3}}$)$.

Now the global well-posedness  of a classical solution
of problem (\ref{li4}) in $[0,T]\times \mathbb{R}^3$ can be obtained by the standard theory \cites{CK3,oar,amj}
at least when $(\nu,\eta, \epsilon)$ are all positive.

\begin{lemma}\label{lem1}
Let $T>0$ and \eqref{canshu} hold.
If $(c_{0},  u_{0})$ satisfy
\begin{equation}\label{zhenginitial}
c_0>0,\quad  (c_0, u_0)\in H^3, \quad \epsilon^{\frac{1}{2}}\nabla c^{\iota}_0 \in  D^1\cap D^2,
\quad \epsilon^{\frac{1}{4}}\nabla c^{\frac{\iota}{2}}_0 \in  L^4,
\end{equation}
then there exists a unique strong solution $(h, c, u)$ of problem \eqref{li4} in $[0,T]\times \mathbb{R}^3$
such that
\begin{equation}\label{reggh}\begin{split}
&h\in L^\infty\cap C ([0,T]\times \mathbb{R}^3),\quad    \nabla h \in C([0,T];H^2),\\
&c \in C([0,T];H^3),\quad  u\in C([0,T]; H^3)\cap L^2([0,T]; H^4).
\end{split}
\end{equation}
\end{lemma}

We now establish the uniform energy estimates, independent of  $(\nu,\eta, \epsilon)$,
of the unique solution $(h,U)$ of the Cauchy problem (\ref{li4}) obtained in Lemma \ref{lem1}.

\subsection{Uniform energy estimates independent of $(\nu,\eta, \epsilon)$}

We first fix $T>0$ and a large enough positive constant $b_0$ (independent of $\epsilon$) such that
\begin{equation}\label{houmian}\begin{split}
 \big\|(c_0, u_0)\big\|^2_{3}
 +\epsilon \big\|\nabla c^{\iota}_0\big\|^2_{D^1\cap D^2}
 +\epsilon^{\frac{1}{2}}\big|\nabla c^{\frac{\iota}{2}}_0\big|^2_4& \leq b_0.
\end{split}
\end{equation}
Then there exists $\eta_{1}>0$ such that, if $0<\eta<\eta_{1}$,
\begin{equation}\label{appnonfar}
\begin{split}
&\eta+\epsilon^{\frac{1}{2}} \big\|\nabla (c_0+\eta)^\iota\big\|_{D^1\cap D^2}
  + \big|(c_0+\eta)^{-\iota}\big|_{\infty}+\big\|(c_0, u_0)\big\|_3\\
&=\eta+\epsilon^{\frac{1}{2}} \big\|\nabla h_0\big\|_{D^1\cap D^2}+\big|h^{-1}_0\big|_{\infty}
  +\big\|(c_0, u_0)\big\|_3 \leq d_0,
\end{split}
\end{equation}
where  we have used the fact that $\epsilon^{\frac{1}{4}}\nabla c^{\frac{\iota}{2}}_0\in L^4$,
and $d_0>0$ is a  constant independent of $(\nu, \eta,\epsilon)$.

We assume that there exist $T^*\in (0,T]$ and a positive constant $d_1$ such that $1< d_0\leq d_1$,
and
\begin{equation}\label{jizhu1}
\sup_{0\leq t \leq T^*}\Big(\epsilon \big\|\nabla g(t,\cdot)\big\|^2_{D^1\cap D^2}+\big\| V(t,\cdot)\big\|^2_{3}\Big)
+\int_{0}^{T^*} \epsilon \sum_{i=1}^4 \big|g\nabla^i v(t,\cdot)\big|^2_2{\dd t} \leq d^2_1,
\end{equation}
where $T^*$ and  $d_1$ will be determined  later (see \eqref{determine}), which
depend only on $d_0$ and the fixed constants $(A, \alpha, \beta, \gamma, \delta, T)$.

Next, a series of uniform local-in-time estimates independent of $(\nu,\eta, \epsilon)$
will be listed in Lemmas  \ref{3}--\ref{2}.
Hereinafter, we use $C\geq 1$ to denote  a generic constant
depending only on the fixed constants $(A, \alpha, \beta, \gamma, \delta, T)$.

\subsubsection{Uniform energy estimates on  $\psi$.}
In order to  deal with the singular elliptic operator $h^2Lu$,
we first need to make some proper  estimates of $\psi$.

\begin{lemma}\label{3}
Let $(h,c, u)$ be the unique strong solution of problem \eqref{li4} in $[0,T] \times \mathbb{R}^3$.
Then
\begin{equation}\label{psi}
\epsilon \big\|\psi(t)\big\|^2_{D^1\cap D^2}\leq Cd^2_0  \qquad\,\, \text{for $0\leq t\leq T_1=\min(T^*, (1+d^2_1)^{-1})$}.
\end{equation}
\end{lemma}

\begin{proof}
Since $\psi=\nabla h$,  and equation $(\ref{li4})_1$ holds,
$\psi$ satisfies the following equations:
\begin{equation}\label{kuzxclinear}
\psi_t+\sum_{l=1}^3 B_l(v) \partial_l\psi+B^*(v)\psi
+\frac{\delta-1}{2} \big(g\nabla \text{div} v+\nabla g \,\text{div} v\big)=0,
\end{equation}
where $B^*(v)=(\nabla v)^\top$.

Next, let $\zeta=(\zeta_1,\zeta_2,\zeta_3)$
with $1\leq |\zeta|\leq 2$ and $\zeta_i=0,1,2$.
Applying operator $\partial_{x}^{\zeta} $ to $(\ref{kuzxclinear})$,
then multiplying by $2\partial_{x}^{\zeta} \psi$, and integrating over $\mathbb{R}^3$ yield
\begin{equation}\label{zhenzhen}
\frac{\rm d}{{\rm d} t}\big|\partial_{x}^{\zeta} \psi\big|^2_2
\leq \Big(\sum_{l=1}^{3}\big|\partial_{l}B_l\big|_\infty+\big|B^*\big|_\infty\Big)\big|\partial_{x}^{\zeta}\psi\big|^2_2
+\sum_{l=1}^3\big|\Theta_l\big|_2 \big|\partial_{x}^{\zeta}\psi\big|_2,
\end{equation}
where
\begin{equation*}\begin{split}
\Theta_1 :=& -\partial_{x}^{\zeta} (B^*\psi)+B^*\partial_{x}^{\zeta}
\psi,\quad\,\,
 \Theta_2 :=\sum_{l=1}^3 \big(-\partial_{x}^{\zeta} (B_l \partial_l
\psi) +B_l \partial_l\partial_{x}^{\zeta}\psi\big),\\
 \Theta_3 := &\frac{\delta-1}{2}\partial_{x}^{\zeta}\big(g\nabla \text{div} v+\nabla g\,\text{div} v\big).
 \end{split}
\end{equation*}

It is direct to see
$$
\frac{\rm d}{\rm d t}\big\|\psi\big\|_{D^1\cap D^2}
\leq C \big\|v\big\|_3 \big\|\psi\big\|_{D^1\cap D^2}
+C\big(\big\|\nabla g\big\|_{D^1\cap D^2}\big\|v\big\|_3+\big|g\nabla^3 v\big|_2+\big|g\nabla^4 v\big|_2\big),
$$
which, along with  Gronwall's inequality and (\ref{jizhu1}), implies
\begin{equation*}
\begin{split}
\big\|\psi(t)\big\|_{D^1\cap D^2}
\leq&\, C\Big(\big\|\psi_0\big\|_{D^1\cap D^2}+\int_0^t\big(d^2_1\epsilon^{-\frac{1}{2}}+\big|g\nabla^3 v\big|_2
     +\big|g\nabla^4 v\big|_2\big)\text{d}s\Big)\exp(Cd_1t)\\
\leq &\, C\epsilon^{-\frac{1}{2}}(d_0+1)\qquad \text{for $0\leq t\leq T_1=\min(T^*, (1+d^2_1)^{-1})$}.
\end{split}
\end{equation*}
\end{proof}

\subsubsection{Uniform energy estimates on $u$}
Based on the estimates for $\psi=\nabla h$ obtained   in Lemma \ref{3},
we now give the corresponding uniform estimates for $u$. Denote $a_2=\frac{2a_1\delta}{\delta-1}$.
\begin{lemma}\label{2}
Let $(h,c, u)$ be the unique strong solution to problem \eqref{li4} in $[0,T] \times \mathbb{R}^3$. Then
\begin{equation}\label{diyi}
\big\|U(t)\big\|^2_{3}
+\int_{0}^{t} \epsilon \sum_{i=1}^4 \big(\big|h\nabla^i u\big|^2_{2}+\nu^2\big|\nabla^i u\big|^2_{2}\big)\text{\rm d}s\leq Cd^2_0
\qquad\,\, \text{for $0\leq t \leq T_1$}.
\end{equation}
\end{lemma}

\begin{proof}
We divide the proof into four steps.

1.
Applying  $\partial_{x}^{\zeta}$ to
$\eqref{li4}_2$, then multiplying by $2\partial_{x}^{\zeta} U$ and
integrating over $\mathbb{R}^3$ yield
\begin{equation}\label{E:3.29*}
\begin{split}
&\frac{\rm d}{{\rm d} t}\int \partial_{x}^{\zeta} U\cdot (A_0\partial^{\zeta}_x U)\,\dd x +2\epsilon a_1 \int (h^2+\nu^2)\Big(\alpha
|\nabla\partial^{\zeta}_x u|^2 +\big(\alpha+\beta\big)|\text{div} \partial^{\zeta}_x u|^2\Big)\,\dd x\\
&=\int\text{div} A|\partial^{\zeta}_x U|^2\,\dd x
 -2\int\sum^3_{l=1}\Big(\partial^{\zeta}_x (A_l\partial_l U)-A_l \partial^{\zeta}_x \partial_l U\Big)\cdot \partial^{\zeta}_x U\,\dd x\\
&\quad -2\epsilon a_1\alpha  \int \Big(\big(\nabla h^2 \cdot \nabla \partial^{\zeta}_x u\big) \cdot\partial^{\zeta}_x u
    -\big(\partial^{\zeta}_x(h^2\triangle u)-h^2\partial^{\zeta}_x \triangle u\big)\cdot \partial^{\zeta}_x u\Big)\,\dd x \\
&\quad -2\epsilon a_1 (\alpha+\beta)\int\big(\nabla h^2\text{div}\partial^{\zeta}_x u\big) \cdot\partial^{\zeta}_x u\,\dd x \\
&\quad +2\epsilon a_1(\alpha+\beta)
 \int\Big(\partial^{\zeta}_x\big(h^2\nabla \text{div} u\big)-h^2\partial^{\zeta}_x\nabla \text{div} u\Big)\cdot \partial^{\zeta}_x u\,\dd x \\
&\quad +\epsilon a_2\int \Big(\nabla h^2 \, \partial^{\zeta}_x Q
  + \partial^{\zeta}_x\big(\nabla h^2\, Q\big)-\nabla h^2 \, \partial^{\zeta}_x Q\Big)\cdot \partial^{\zeta}_x u\,\dd x=:\sum^{8}_{i=1}I_i.
\end{split}
\end{equation}

2. We need to consider the terms on the right-hand side of
\eqref{E:3.29*} when $|\zeta|\leq 3$.
It follows from the Gagliardo-Nirenberg inequality and  H\"older's inequality that
\begin{equation}\label{guji1}
I_1=\int\text{div} A \big|\partial^{\zeta}_x U\big|^2\,\dd x
\leq  C\big|\text{div} A\big|_{\infty}\big|\partial^{\zeta}_x U\big|^2_2
\leq  Cd_1\big|\partial^{\zeta}_x U\big|^2_2 \qquad\,\, \text{for $|\zeta|\leq 3$}.
\end{equation}

Similarly, for $I_2$, using Lemma \ref{zhen1},
we have
\begin{equation}\label{guji2}
\begin{split}
I_2=&-2\sum_{l=1}^3\int\big(\partial^{\zeta}_x (A_l\partial_l U)-A_l \partial^{\zeta}_x \partial_l U\big)\cdot \partial^{\zeta}_x U\,{\rm d} x\\
  \leq&\,  C\big|\nabla V\big|_\infty\big|\nabla U\big|^2_2\leq Cd_1\big|\nabla U\big|^2_2 \,\,\qquad \text{for $|\zeta|=1$};\\[6pt]
I_2\leq &\, C\big(\big|\nabla V\big|_\infty\big|\nabla^2 U\big|_2
   +\big|\nabla^2V\big|_3\big|\nabla U\big|_6\big)\big|\nabla^2 U\big|_2\\
\leq &\, Cd_1\big\|\nabla U\big\|^2_1 \,\,\qquad \text{for $|\zeta|=2$};\\[6pt]
I_2\leq  &\, C\big(\big|\nabla V\big|_\infty\big|\nabla^3 U\big|_2+\big|\nabla^3 V\big|_2\big|\nabla U\big|_\infty\big)\big|\nabla^3 U\big|_2\\
\leq &\, Cd_1\big\|\nabla U\big\|^2_2 \,\,\qquad \text{for $|\zeta|=3$}.
\end{split}
\end{equation}

For $I_3$ and $I_5$, it follows  from Lemma \ref{3}, the Gagliardo-Nirenberg inequality, the H\"older inequality, and the Young inequality that
\begin{equation}\label{guji3}
\begin{split}
I_3+I_5
=&-2\epsilon a_1 \int \big(\alpha\nabla h^2\, \nabla \partial^{\zeta}_x u
        +(\alpha+\beta)\nabla h^2\text{div} \partial^{\zeta}_x u\big) \cdot\partial^{\zeta}_x u\,\dd x \\
\leq &\, C\epsilon a_1 \big|\psi\big|_\infty \big|h\nabla \partial^{\zeta}_x u\big|_2 \big|\partial^{\zeta}_x u\big|_2\\[1.5mm]
\leq &\, \frac{\epsilon a_1\alpha}{16}\big|h\nabla \partial^{\zeta}_x u\big|^2_2+C d^2_0\big|\partial^{\zeta}_x u\big|^2_2
\qquad\,\, \text{for $|\zeta|\leq 3$}.
\end{split}
\end{equation}

Similarly, for $I_4$ and $I_7$,  it follows from Lemma \ref{3} that
\begin{equation}\label{guji5}
\begin{split}
I_4=&\,2\epsilon a_1\alpha\int\Big(\partial^{\zeta}_x(h^2\triangle u)-h^2\partial^{\zeta}_x \triangle u\Big)\cdot \partial^{\zeta}_x u\,\dd x \\
 \leq &\,  C\epsilon a_1\alpha \big|\psi\big|_{\infty}\big|h\triangle u\big|_2|\nabla u|_2\\
 \leq  &\,  \frac{\epsilon a_1\alpha}{16}\big|h\nabla^2 u\big|^2_2+C d^2_0 \big|\nabla u\big|^2_2 \qquad\,\,  \text{for $|\zeta|=1$};\\[2mm]
I_4\leq &\, C\epsilon a_1 \alpha\big(\big|\psi\big|_{\infty}\big|h\nabla^3 u\big|_2
   +\big|\psi\big|^2_{\infty}\big|\triangle u\big|_2+\big|\nabla \psi\big|_3\big|h\nabla^2 u\big|_6\big)\big|\nabla^2 u\big|_2\\
 \leq &\,\frac{\epsilon a_1\alpha}{16}\big|h\nabla^3 u\big|^2_2+ C d^2_0\big\|\nabla^2 u\big\|^2_1 \qquad\,\, \text{for $|\zeta|=2$};\\[2mm]
I_4\leq &\, C\epsilon a_1 \alpha\big(\big|\nabla^2\psi\big|_{2}\big|h\nabla^2 u\big|_{\infty}+\big|\nabla \psi\big|_{3}\big|h\nabla^3 u\big|_6
  +\big|\psi\big|_{\infty}\big|\nabla \psi|_3\big|\nabla^2 u\big|_6\big)\big|\nabla^3 u\big|_2\\
&+C\epsilon  a_1\alpha\big(\big|\psi\big|^2_{\infty}\big|\nabla^3 u\big|_{2}
   +\big|\psi\big|_{\infty}\big|h\nabla^4 u\big|_2\big)\big|\nabla^3 u\big|_2\\
 \leq &\frac{\epsilon a_1\alpha}{16}\big(\big|h\nabla^3 u\big|^2_2+\big|h\nabla^4 u\big|^2_2\big)
    +Cd^2_0\big\|\nabla^2 u\big\|^2_1 \qquad\,\,  \text{for $|\zeta|=3$};\\
 I_7=&\,\epsilon a_2\int \big(\nabla h^2\, \partial^{\zeta}_x Q(u)\big)\cdot \partial^{\zeta}_x u\,\dd x
  \leq   C\epsilon a_1 \big|\psi\big|_\infty\big|h\nabla^{|\zeta|+1} u\big|_2\big|\partial^{\zeta}_x u\big|_2\\
  \leq &\, \frac{\epsilon a_1\alpha}{16}\big|h\nabla^{|\zeta|+1} u\big|^2_2
      +C d^2_0\big|\partial^{\zeta}_x u\big|^2_2 \qquad\,\,  \text{for $|\zeta|\leq 3$}.
\end{split}
\end{equation}
Because of the same structure, the above  estimates of $I_4$ also hold for the term:
$$
I_6=2\epsilon a_1(\alpha+\beta)\int\big(\partial^{\zeta}_x(h^2\nabla \text{div} u)
  -h^2\partial^{\zeta}_x\nabla \text{div} u\big)\cdot \partial^{\zeta}_x u\,\dd x.
$$

3. For $I_8$, it follows from Lemma \ref{3} that
\begin{equation}\label{guji7}
\begin{split}
I_8=&\, \epsilon a_2 \int \Big(\partial^{\zeta}_x\big(\nabla h^2\, Q(u)\big)-\nabla h^2\, \partial^{\zeta}_x Q(u)\Big)\cdot\partial^{\zeta}_x u\,\dd x\\
 \leq&\,  C\epsilon a_2 \big(\big|\psi\big|^2_\infty\big|\nabla u\big|^2_2
    +\big|\nabla \psi\big|_6\big|h\nabla u\big|_2\big|\nabla u\big|_3\big)\\
\leq  &\, \frac{\epsilon a_1\alpha}{16}\big|h\nabla u\big|^2_2+C d^2_0 \big\|\nabla u\big\|^2_1 \qquad\,\,\text{for $|\zeta|=1$};\\[2mm]
I_8\leq&\, C\epsilon a_2\Big( \big|\nabla^2 \psi\big|_2\big|h\nabla u\big|_{\infty}
    + \big|\psi\big|_\infty\big|\nabla \psi\big|_{2}\big|\nabla u\big|_\infty \Big)\big|\nabla^2 u\big|_2\\
&+ C\epsilon a_2\Big( \big|\nabla \psi\big|_{3}\big|h\nabla^2 u\big|_{6}
  + \big|\psi\big|^2_{\infty}\big|\nabla^2 u\big|_2\Big)\big|\nabla^2 u\big|_2\\
 \leq &\,\frac{\epsilon a_1\alpha}{16}\big(\big|h\nabla^2 u\big|^2_2+\big|h\nabla^3 u\big|^2_2\big)
   +Cd^2_0\big\|\nabla u\big\|^2_2 \qquad\,\,  \text{for $|\zeta|=2$}.
\end{split}
\end{equation}

On the other hand, when $|\zeta|=3$,
\[
\begin{split}
&\partial^{\zeta}_x\big(\nabla h^2 Q(u)\big)-\nabla h^2 \partial^{\zeta}_x Q(u)\\
&=
\sum_{i=1}^3 C_{i1}\partial^{\zeta^i}_x\nabla h^2 \partial^{\zeta-\zeta^i}_xQ(u)
+\sum_{i=1}^3
C_{i2}\partial^{\zeta-\zeta^i}_x\nabla h^2 \partial^{\zeta^i}_x
Q(u)+\partial^{\zeta}_x \nabla h^2 Q(u),
\end{split}
\]
where  $C_{ij}$, $i=1,2,3, j=1,2$, are all constants,  and $\zeta=\zeta^1+\zeta^2+\zeta^3$
with $\zeta^i\in \mathbb{R}^3$ as a
multi-index satisfying $|\zeta^i|=1$, $i=1, 2, 3$.
Then $I_8:=I_{81}+I_{82}+I_{83}$ can be estimated as follows:
\begin{equation}\label{DZ12cc}
\begin{split}
I_{81}=&\,\epsilon a_2\int\sum_{i=1}^3C_{i1}\big(\partial^{\zeta^i}_x\nabla h^2\partial^{\zeta-\zeta^i}_x Q(u)\big)\cdot \partial^{\zeta}_x u\,\dd x\\
\leq &\, C\epsilon a_1\big(\big|\psi\big|^2_\infty\big|\nabla^3 u\big|_2
           +\big|\nabla \psi\big|_3\big|h\nabla^3 u\big|_6\big)\big|\nabla^3 u\big|_2\\
\leq &\,\frac{\epsilon a_1\alpha}{16}\big|h\nabla^4 u\big|^2_2+Cd^2_0\big|\nabla^3 u\big|^2_2,\\[1mm]
I_{82}=&\,\epsilon a_2\int\sum_{i=1}^3C_{i2}\big(\partial^{\zeta-\zeta^i}_x\nabla h^2 \partial^{\zeta^i}_x Q(u)\big)\cdot \partial^{\zeta}_x u\,\dd x \\
\leq &\, C\epsilon a_1\big(\big|h\nabla^2u\big|_{\infty}\big|\nabla^2 \psi\big|_2
       +\big|\psi\big|_\infty\big|\nabla \psi\big|_{6}\big|\nabla^2 u\big|_3\big)\big|\nabla^3 u\big|_2\\
 \leq &\, \frac{\epsilon a_1\alpha}{16}\big(\big|h\nabla^3 u\big|^2_2+\big|h\nabla^4 u\big|^2_2\big)+Cd^2_0\big\|\nabla^2 u\big\|^2_1,\\[2mm]
I_{83}=&\,2\epsilon a_1 \int\big(\partial^{\zeta}_x \nabla h^2 Q(u)\big)\cdot\partial^{\zeta}_x u\,\dd x\\
\leq &\, C\epsilon a_1\big(\big|h\nabla^2u\big|_{\infty}\big|\nabla^3 u\big|_2
      +\big|h\nabla^4 u\big|_2\big|\nabla u\big|_\infty\big)\big|\nabla^2 \psi\big|_2\\
&+C\epsilon a_1\big(\big|\psi\big|_\infty\big|\nabla^2 \psi\big|_2
   +\big|\nabla \psi\big|_3\big|\nabla \psi\big|_6\big)\big|\nabla u\big|_\infty\big|\nabla^3 u\big|_2\\
 \leq &\,\frac{\epsilon a_1\alpha }{16}\big(\big|h\nabla^3 u\big|^2_2+\big|h\nabla^4 u\big|^2_2\big)+Cd^2_0\big\|\nabla^2 u\big\|^2_1,
\end{split}
\end{equation}
where we have performed integration by parts for $I_{83}$.

\smallskip
4. From \eqref{E:3.29*}--\eqref{DZ12cc},
along with the Gronwall inequality, we have
\begin{equation}\label{DZ34}
\big\|U(t)\big\|^2_{3}+\epsilon\int^t_0\sum_{i=1}^4 \big(\big|h\nabla^i u\big|^2_{2}+\nu^2\big|\nabla^i u\big|^2_{2}\big)\text{d} s
\leq
\big\|U_0\big\|^2_{3}\exp(C d^2_1t)\leq Cd^2_0 \quad\mbox{for $0\leq t\leq T_1$}.
\end{equation}
This completes the proof.
\end{proof}

\medskip
Then, from Lemmas \ref{3}--\ref{2}, for
$$
0 \leq t \leq T_1=\min (T^*, (1+d_1)^{-2}),
$$
we have
\begin{equation*}
\epsilon \big\|\psi(t)\big\|^2_{D^1\cap D^2}+\big\| U(t)\big\|^2_{3}
+  \epsilon \sum_{i=1}^4 \int_{0}^{t}\big(\big|h\nabla^i u\big|^2_{2}+\nu^2 \big|\nabla^i u\big|^2_{2}\big)\dd s\leq Cd^2_0.
\end{equation*}
Therefore, defining
\begin{equation}\label{determine}
T^*=\min (T, (1+d^{2}_1)^{-1}),\quad d_1=C^{\frac{1}{2}}d_0,
\end{equation}
we have
\begin{equation}\label{jkk}
\epsilon \big\|\psi(t)\big\|^2_{D^1\cap D^2}+\big\| U(t)\big\|^2_{3}
+\epsilon \sum_{i=1}^4\int_{0}^{t}  \big(\big|h\nabla^i u\big|^2_{2}+\nu^2\big|\nabla^i u\big|^2_{2}\big)\text{\rm d}s\leq d^2_1
\quad \mbox{for $0\leq t \leq T^*$}.
\end{equation}
In other words, given fixed $d_0$ and $T$,
there are positive constants $T^*$ and $d_1$, depending only on $d_0$ and $T$,
such that, if \eqref{jizhu1} holds for $(g,v)$,
then \eqref{jkk} holds for the strong solution $(h,\psi,c,u)$ of problem
\eqref{li4} in $[0, T^*]\times \mathbb{R}^3$.

\subsection{Construction of the nonlinear approximation solutions}
In this subsection, based on the assumption that $\nu>0$ and  $h_0\leq \eta^\iota$ for some $\eta>0$,
we now give the local-in-time well-posedness of the following nonlinear Cauchy problem:
 \begin{equation}
\label{middle}
\begin{cases}
\displaystyle
 \displaystyle h_t+u\cdot \nabla h+\frac{\delta-1}{2}h\text{div} u=0,\\[4pt]
\displaystyle
 A_0 U_t+\sum^{3}_{j=1}A_j(U)\partial_j U=-\epsilon F(\nu,h, u)+\epsilon G(h,\psi, u),\\[4pt]
  (h,  c,u)|_{t=0}=((A\gamma)^{-\frac{\iota}{2}}(c_0+\eta)^\iota,  c_0,u_0)(x)\qquad\mbox{for $x\in \mathbb{R}^3$},
\end{cases}
\end{equation}
where $\psi=\nabla h$ satisfies
\begin{equation}\label{psiequation}
\psi_t+\sum_{l=1}^3 B_l(u) \partial_l\psi
  +B(u)\psi + \frac{\delta-1}{2}h  \nabla \text{div} u=0.
  \end{equation}

\begin{theorem}\label{th1zx111}
Let \eqref{canshu} hold and $(\nu,\eta, \epsilon)\in (0,1]\times(0,1]\times (0,1]$.
If the initial data  $(c_{0},h_0,  u_{0})$ satisfy \eqref{zhenginitial}, then
there exist $T_*>0$ and a unique classical solution $(h,c, u)$
of problem \eqref{middle} in $[0,T_*]\times \mathbb{R}^3$
such that
\eqref{reggh} and the uniform a priori estimates \eqref{jkk}
with $T^*$ replaced by $T_*$ hold, where $T_*$ is independent of $(\nu,\eta, \epsilon)$.
\end{theorem}

The proof is based on an iteration scheme and the conclusions obtained in \S 4.1--\S 4.2.
As in \S 4.2, we define the same constants $d_i$, $i=0,1$.

Let $(h^0,c^0, u^0)$ be the solution of the Cauchy problem in $(0,\infty)\times \mathbb{R}^3$:
\begin{equation}\label{zheng6}
\left\{\begin{aligned}
\displaystyle
 &X_t+u_0 \cdot \nabla X=0,\\[4pt]
\displaystyle
&Y_t+u_0 \cdot \nabla Y=0, \\[4pt]
\displaystyle
&Z_t- \epsilon (X^2+\nu^2)\triangle Z=0,\\[4pt]
\displaystyle
&(X,Y,Z)|_{t=0}
=((A\gamma)^{-\frac{\iota}{2}}(c_0+\eta)^\iota,c_0,u_0)(x) \qquad \text{for $x\in \mathbb{R}^3$}.
\end{aligned}
\right.
\end{equation}
Choose $T^{**}\in (0,T^*]$ small enough such that, for $0\leq t \leq T^{**}$,
\begin{equation}\label{jizhu}
\epsilon \big\|\nabla h^0(t)\big\|^2_{D^1\cap D^2}
+\big\| (c^0,u^0)(t)\big\|^2_{3}
+\epsilon \sum_{i=1}^4  \int_{0}^{t}\big(\big|h^0\nabla^i u^0\big|^2_{2}+\nu^2\big|\nabla^i u^0\big|^2_{2}\big)\text{d}s\leq d^2_1.
\end{equation}

\begin{proof} The desired existence, uniqueness, and time-continuity can be proved in the following three steps.

\smallskip
\paragraph{{\rm 1}. \em  Existence.}
Let the first step of our iteration be
$(g,\varphi,v)=(h^0,c^0,u^0)$.  Then we can obtain a classical solution $(h^1,c^1,u^1)$ of problem (\ref{li4})
so that $\psi^1:=\nabla h^1$ is defined.
Inductively, we construct approximate sequence $(h^{k+1}, c^{k+1},u^{k+1})$ as follows:
Given $(h^k, c^k,  u^{k})$ for $k\geq 1$, define $( h^{k+1},c^{k+1},u^{k+1})$  by solving the following problem:
\begin{equation}\label{li6}
\begin{cases}
\displaystyle
h^{k+1}_t+u^{k}\cdot \nabla h^{k+1}+\frac{\delta-1}{2}h^k\text{div} u^k=0,\\[4pt]
\displaystyle
A_0 U^{k+1}_t+\sum^3_{l=1} A_l(U^k)\partial_l U^{k+1}
=-\epsilon F(\nu,h^{k+1}, u^{k+1})+ \epsilon G(h^{k+1},\psi^{k+1}, u^{k+1}),\\[4pt]
\displaystyle
( h^{k+1}, c^{k+1}, u^{k+1})|_{t=0}
=((A\gamma)^{-\frac{\iota}{2}}(c_0+\eta)^\iota, c_0,u_0)(x) \,\,\qquad \mbox{for $x\in \mathbb{R}^3$},
 \end{cases}
\end{equation}
where $\psi^{k+1}=\nabla h^{k+1}$ satisfy the following equations ({\it cf.} \eqref{kuzxclinear}):
\begin{equation}\label{4.36a}
\psi^{k+1}_t+\sum^3_{l=1} B_l(u^k) \partial_l\psi^{k+1}+B^*(u^k)\psi^{k+1}
+\frac{\delta-1}{2} \nabla\big(h^k \text{div} u^k\big)=0.
\end{equation}
Then the solution sequences  $(h^{k}, c^{k}, u^{k})$, $k=1,2,...,$  satisfy the uniform estimates (\ref{jkk}).

Next, we prove that the whole sequence $(h^k,\psi^k, c^k,
u^k)$ converges strongly to a limit $(h,\psi,c,u)=(h,\nabla h, c, u)$ in some strong sense.
Denote
\begin{equation*}\begin{split}
&\overline{h}^{k+1}:=h^{k+1}-h^k,\quad \overline{\psi}^{k+1}:=\psi^{k+1}-\psi^k=\nabla  \overline{h}^{k+1},\\
&\overline{c}^{k+1}:=c^{k+1}-c^k,\quad  \overline{u}^{k+1}:=u^{k+1}-u^k, \quad  \overline{U}^{k+1}=
(\overline{c}^{k+1}, \overline{u}^{k+1}).
\end{split}
\end{equation*}
Then, from (\ref{li6}), we have
\begin{equation}\label{eq:1.2w}
\left\{\begin{aligned}
\displaystyle
& \overline{h}^{k+1}_t+u^k\cdot \nabla\overline{h}^{k+1} +\overline{u}^k\cdot\psi^{k}
  +\frac{\delta-1}{2}\big(\overline{h}^{k} \text{div}u^{k-1} +h^{k}\text{div}\overline{u}^k\big)=0,\\[1pt]
\displaystyle
&\overline{\psi}^{k+1}_t+\sum^3_{l=1} B_l(u^k) \partial_l\overline{\psi}^{k+1}+B^*(u^k)\overline{\psi}^{k+1}\\[1pt]
&\quad =-\frac{\delta-1}{2}\big(\overline{h}^{k} \nabla\text{div}u^{k-1} +h^{k}\nabla \text{div}\overline{u}^k\big)\\[1pt]
&\quad\quad -\sum^3_{l=1} B_l(\overline{u}^k) \partial_l \psi^{k}-B^*(\overline{u}^k)\psi^k
  -\frac{\delta-1}{2}\big(\overline{\psi}^{k} \text{div}u^{k-1} +\psi^{k}\text{div}\overline{u}^k\big),\\[2pt]
\displaystyle
&A_0\overline{U}^{k+1}_t+\sum_{l=1}^3A_l(U^k)\partial_l\overline{U}^{k+1}+
\epsilon F(\nu, h^{k+1}, \overline{u}^{k+1}) \\[1pt]
&\quad=-\sum_{l=1}^3 A_l(\overline{U}^k)\partial_l U^k-\epsilon\big(F(\nu, h^{k+1}, u^k)-F(\nu, h^{k}, u^k)\big)\\[1pt]
&\quad\quad +
\epsilon\Big( G(h^{k+1},\overline{\psi}^{k+1}, u^{k+1})+G(h^{k+1},\psi^{k}, \overline{u}^{k+1})+G(\overline{h}^{k+1}, \psi^{k}, u^{k})\Big).
\end{aligned}
\right.
\end{equation}

First, for $\overline{h}^{k+1}$, we state the following lemma:
\begin{lemma}\label{cancel} We have
$$
\overline{h}^{k+1} \in L^\infty([0,T^{**}];H^3(\mathbb{R}^3)) \qquad\,\, \text{for $k=1,2,...$}.
$$
\end{lemma}

\begin{remark}
This lemma is important for our limit process from the linear problem to the nonlinear one,
and helps us to deal with the cancellation of the most singular terms in the limit process.
Its proof can be found in Remark {\rm\ref{cancelproof}} at the end of this subsection.
\end{remark}

Based on Lemma \ref{cancel},  multiplying
$(\ref{eq:1.2w})_1$ by $2\overline{h}^{k+1}$ and then integrating over
$\mathbb{R}^3$, we have
\begin{equation}\label{difference1}
\begin{split}
 \epsilon\frac{\rm d}{{\rm d} t}\big|\overline{h}^{k+1}\big|^2_2
\leq&\, C\epsilon \big(\big|\nabla u^k\big|_\infty \big|\overline{h}^{k+1}\big|_2
  +\big|\psi^k\big|_6\big|\overline{u}^k\big|_3\big)\big|\overline{h}^{k+1}\big|_2\\
    &\,+C\epsilon\big(\big|\overline{h}^k\big|_2\big|\nabla u^{k-1}\big|_\infty
      +\big|h^k \nabla \overline{u}^k\big|_2\big)\big|\overline{h}^{k+1}\big|_2\\
\leq &\, C\sigma^{-1}\epsilon \big|\overline{h}^{k+1}\big|^2_2
     +\sigma\big(\big\|\overline{u}^k\big\|^2_1+\epsilon \big|\overline{h}^k\big|^2_2
      +\epsilon \big|h^k \nabla \overline{u}^k\big|^2_2\big),
\end{split}
\end{equation}
where $\sigma\in(0, \frac{1}{10})$ is a constant to be determined later.

Then multiplying
$(\ref{eq:1.2w})_2$ by $2\overline{\psi}^{k+1}$ and integrating over
$\mathbb{R}^3$ yield
\begin{equation}\label{difference2}
\begin{split}
\epsilon\frac{\rm d}{{\rm d} t} \big|\overline{\psi}^{k+1}\big|^2_2
\leq &\, C\epsilon\big|\nabla u^k\big|_\infty\big|\overline{\psi}^{k+1}\big|^2_2
+C\epsilon\big(\big|\nabla^2 u^{k-1}\big|_3\big|\overline{h}^k\big|_6
   +\big|h^k\nabla^2 \overline{u}^{k}\big|_2\big)\big|\overline{\psi}^{k+1}\big|_2\\
&+C\epsilon\big(\big|\overline{u}^k\big|_6\big|\nabla \psi^k\big|_3
  +\big|\nabla \overline{u}^k\big|_2\big|\psi^k\big|_\infty
   +\big|\overline{\psi}^k\big|_2\big|\nabla u^{k-1}\big|_\infty\big)\big|\overline{\psi}^{k+1}\big|_2\\
\leq &\, C\sigma^{-1}\epsilon\big|\overline{\psi}^{k+1}\big|^2_2
  +\sigma\big(\big\|\overline{u}^k\big\|^2_1+\epsilon\big|\overline{\psi}^k\big|^2_2
  +\epsilon\big|h^k \nabla^2 \overline{u}^k\big|^2_2\big).
\end{split}
\end{equation}

Next,  multiplying
$(\ref{eq:1.2w})_3$ by $2\overline{U}^{k+1}$ and integrating over
$\mathbb{R}^3$, we have
\begin{equation}\label{suanshi1}
\begin{split}
&\frac{\rm d}{{\rm d}t}\int\overline{U}^{k+1}\cdot
 (A_0\overline{U}^{k+1})\,\dd x  +2\epsilon
a_1\int \alpha \big(\big|h^{k+1}\nabla \overline{u}^{k+1}\big|^2
     +\nu^2\big|\nabla \overline{u}^{k+1}\big|^2\big)\dd x\\
&+2\epsilon
a_1\int (\alpha+\beta)\big(\big|h^{k+1}\text{div} \overline{u}^{k+1}\big|^2
     +\nu^2\big|\text{div} \overline{u}^{k+1}\big|^2\big)\dd x\\
=&\,  \int \text{div} A(U^k)|\overline{U}^{k+1}|^2_2\,\dd x
  -2\sum^3_{l=1}\int  \big(A_l(\overline{U}^k)\partial_l U^k\big)\cdot \overline{U}^{k+1}\dd x\\
 &\,-2\epsilon a_1\int \Big(\alpha \big(\nabla (h^{k+1})^2\cdot \nabla \overline{u}^{k+1}\big)
      +(\alpha+\beta)\big(\nabla (h^{k+1})^2\text{div} \overline{u}^{k+1}\big)\Big)\cdot \overline{u}^{k+1}\dd x\\
 &\,+ 2\epsilon a_1 \int\Big(-\overline{h}^{k+1}(h^{k+1}+h^{k})Lu^k
   +\frac{2\delta}{\delta-1}h^{k+1}\overline{\psi}^{k+1}\, Q(u^{k+1})\Big)\cdot \overline{u}^{k+1}\dd x\\
 &\,+\frac{4\epsilon a_1\delta}{\delta-1}\int  \Big(h^{k+1}\psi^{k}\, Q(\overline{u}^{k+1})
   +\overline{h}^{k+1}\psi^{k}\, Q(u^{k})\Big) \cdot \overline{u}^{k+1}\dd x\\
\leq&\,
C\big|\nabla U^k\big|_{\infty}\big|\overline{U}^{k+1}\big|_2\big(\big|\overline{U}^{k+1}\big|_2
   + \big|\overline{U}^k\big|_2\big)
  + C\epsilon \big|\psi^{k+1}\big|_{\infty}\big|h^{k+1}\nabla \overline{u}^{k+1}\big|_2\big|\overline{u}^{k+1}\big|_2\\
&\,+ C\epsilon \big(\big|\overline{h}^{k+1}\big|^2_6\big|Lu^k\big|_2
  +\big|\overline{h}^{k+1}\big|_2\big|h^k Lu^k\big|_{6}
    +\big|\overline{\psi}^{k+1}\big|_2\big|h^{k+1}\nabla u^{k+1}\big|_6\big)\big|\overline{u}^{k+1}\big|_3\\
&\,+ C\epsilon \big(\big|h^{k+1}\nabla \overline{u}^{k+1}\big|_2\big|\psi^k\big|_\infty
    +\big|\psi^k\big|_\infty \big|\nabla u^{k}\big|_\infty\big|\overline{h}^{k+1}\big|_{2}\big)\big|\overline{u}^{k+1}\big|_2,
\end{split}
\end{equation}
where we have used the following fact:
\begin{equation}\label{minusplus}
\begin{split}
\int \overline{h}^{k+1}h^{k+1}Lu^k\cdot \overline{u}^{k+1}\dd x
=&\,\int \overline{h}^{k+1}\big(\overline{h}^{k+1}+h^k\big)Lu^k\cdot \overline{u}^{k+1}\dd x\\
\leq &\, C\big(\big|\overline{h}^{k+1}\big|^2_6\big|Lu^k\big|_2
  +\big|\overline{h}^{k+1}\big|_2\big|h^kLu^k\big|_6\big)\big|\overline{u}^{k+1}\big|_3.
\end{split}
\end{equation}

Now, applying operator $\partial_{x}^{\zeta}$ ($|\zeta|=1$) to $(\ref{eq:1.2w})_3$, multiplying by
$2\partial_{x}^\zeta\overline{U}^{k+1}$, and integrating over
$\mathbb{R}^3$ yield
\begin{equation}\label{DZ11xyz-1}
\begin{split}
&\frac{\rm d}{{\rm d}t}\int \partial_{x}^\zeta\overline{U}^{k+1}\cdot
(A_0\partial_{x}^\zeta\overline{U}^{k+1})\,\dd x
 +2\epsilon a_1\int \alpha \big(\big|h^{k+1}\nabla\partial^\zeta_x\overline{u}^{k+1}\big|^2
    +\nu^2\big|\nabla\partial^\zeta_x\overline{u}^{k+1}\big|^2\big)\dd x\\
&\quad +2\epsilon
a_1\int (\alpha+\beta)\big(\big|h^{k+1}\text{div}\partial^{\zeta}_x\overline{u}^{k+1}\big|^2
  +\nu^2\big|\text{div}\partial^{\zeta}_x\overline{u}^{k+1}\big|^2\Big)\dd x\\
&=\int \text{div} A(U_k)\big|\partial^\zeta_x \overline{U}^{k+1}\big|^2\dd x
  -2\sum^3_{l=1}\int\partial_{x}^\zeta\big(A_l(\overline{U}^k)\partial_l U^k\big)
   \cdot\partial^{\zeta}_x \overline{U}^{k+1}\dd x \\
&\quad +
2\sum^3_{l=1}\int\Big(A_l(U^k)\partial_l\partial^{\zeta}_x\overline{U}^{k+1}-\partial_{x}^\zeta\big( A_l(U^k)\partial_l\overline{U}^{k+1}\big)
 \Big)\cdot\partial^\zeta_x\overline{U}^{k+1}\,\dd x\\
&\quad +2\epsilon a_1\int \partial^\zeta_x\Big(-\overline{h}^{k+1}(h^{k+1}+h^{k})Lu^k
    +G(h^{k+1},\overline{\psi}^{k+1}, u^{k+1})\Big) \cdot \partial^{\zeta}_x \overline{u}^{k+1}\,\dd x \\
&\quad +2\epsilon a_1\int \partial^\zeta_x\Big( G(h^{k+1},\psi^{k}, \overline{u}^{k+1})
   +G(\overline{h}^{k+1}, \psi^{k}, u^{k})\Big) \cdot \partial^\zeta_x\overline{u}^{k+1}\,\dd x\\
& :=J^*.
\end{split}
\end{equation}
It follows from integration by parts and   H\"older's inequality  that
 \begin{equation}\label{DZ11xyz}
\begin{split}
 J^* \leq&\, C\big(\big|\nabla U^k\big|_{\infty}\big(\big|\nabla \overline{U}^{k+1}\big|_2+\big|\nabla\overline{U}^k\big|_2\big)
   +\big|\overline{U}^k\big|_6\big|\nabla^2 U^k\big|_{3}\big)\big|\nabla \overline{U}^{k+1}\big|_2 \\[1mm]
&+C\epsilon\big(\big|\overline{h}^{k+1}\big|_6\big|h^{k+1}\nabla^2 \overline{u}^{k+1}\big|_{2}\big|Lu^k\big|_3
+ \big|h^k Lu^k\big|_\infty\big|\overline{\psi}^{k+1}\big|_2\big|\nabla \overline{u}^{k+1}\big|_2\big)\\[1mm]
&+C\epsilon\big(\big|\overline{h}^{k+1}\big|_3\big|h^{k}\nabla^3 u^{k}\big|_{6}\big|\nabla \overline{u}^{k+1}\big|_2
+\big|\overline{h}^{k+1}\big|_6 \big|\psi^k\big|_\infty \big|\nabla^2 u^k\big|_3\big|\nabla \overline{u}^{k+1}\big|_2\big)\\[1mm]
&+ C\epsilon \big(\big|\overline{\psi}^{k+1}\big|_2  \big|\nabla u^{k+1}\big|_\infty
+\big|\psi^{k}\big|_\infty  \big|\nabla \overline{ u}^{k+1}\big|_2\big)\big|h^{k+1}\nabla^2 \overline{u}^{k+1}\big|_2\\[1mm]
&+C\epsilon \big(\big|\nabla \psi^{k}\big|_3\big|\overline{h}^{k+1}\big|_{6}
 +\big|\psi^{k}\big|_\infty\big|\overline{\psi}^{k+1}\big|_{2}\big)\big|\nabla u^k\big|_\infty\big|\nabla \overline{u}^{k+1}\big|_2.
\end{split}
\end{equation}

Combining (\ref{suanshi1})--(\ref{DZ11xyz}) with the Young inequality, we have
\begin{equation}\label{gogo13}\begin{split}
&\frac{\rm d}{{\rm d} t}\big\|\overline{U}^{k+1}\big\|^2_1
 +\epsilon a_1 \big(\nu^2 \big\|\nabla \overline{u}^{k+1}\big\|^2_{1}
  +\big|h^{k+1}\nabla \overline{u}^{k+1}\big|^2_{2}+\big|h^{k+1}\nabla^2 \overline{u}^{k+1}\big|^2_{2}\big)\\[1mm]
&\leq \Theta^k_{\sigma,\epsilon}(t)
  \big(\big\|\overline{U}^{k+1}\big\|^2_1+\epsilon \big|\overline{\psi}^{k+1}\big|^2_{2}+\epsilon\big|\overline{h}^{k+1}\big|^2_{2}\big)
  +\sigma\big\|\overline{U}^{k}\big\|^2_1,
\end{split}
\end{equation}
where $\Theta^k_{\sigma,\epsilon}(t)$ satisfies
$$
\int_0^t \Theta^k_{\sigma,\epsilon}(s)\,\text{d}s\leq C+C_\sigma t \qquad\,\, \text{for $t\in (0,T^{**}]$}.
$$

Finally, let
\begin{equation*}
\Gamma^{k+1}(t)=\sup_{0\leq s \leq t}\big(\epsilon\big|\overline{h}^{k+1}(s)\big|^2_{2}+\epsilon\big|\overline{\psi}^{k+1}(s)\big|^2_{2}
  +\big\|\overline{U}^{k+1}(s)\big\|^2_{1}\big).
\end{equation*}
Combining (\ref{difference1})--(\ref{difference2}) and  (\ref{gogo13}) with the Gronwall inequality leads to
\begin{equation*}\begin{split}
&\Gamma^{k+1}(t)
 + \epsilon a_1  \int_{0}^{t}\big(\nu^2 \big\|\nabla \overline{u}^{k+1}\big\|^2_{1}
   +\big|h^{k+1}\nabla \overline{u}^{k+1}\big|^2_{2}+\big|h^{k+1}\nabla^2 \overline{u}^{k+1}\big|^2_{2}\big)\text{d}s\\
&\leq  C\sigma \Big( \epsilon \int_{0}^{t} \big(\big|h^{k}\nabla \overline{u}^{k}\big|^2_{2}
            +\big|h^{k}\nabla^2 \overline{u}^{k}\big|^2_{2}\big)\text{d}s
  +t\sup_{0\leq s \leq t}\Gamma^{k}(t)\Big)\exp{(C+C_\sigma t)}.
\end{split}
\end{equation*}

Choose $\sigma>0$ and $T_* \in (0,T^{**})$ small enough
such that
$$
C\sigma \exp{C}\leq \min(\frac{1}{8}, \frac{a_1 }{8}), \qquad
\exp(C_\sigma T_*) \leq 2.
$$
It is clear that $T_*>0$ is independent of $(\nu, \eta, \epsilon)$.
Then we have
\begin{equation*}\begin{split}
\sum_{k=1}^{\infty}\Big(  \Gamma^{k+1}(T_*)
+\epsilon a_1\int_{0}^{T_*}
\big(\nu^2 \big\|\nabla \overline{u}^{k+1}\big\|^2_{1}
 +\big|h^{k+1}\nabla \overline{u}^{k+1}\big|^2_{2}+\big|h^{k+1}\nabla^2\overline{u}^{k+1}\big|^2_{2}\big)\text{d}s\Big)
\leq C<\infty.
\end{split}
\end{equation*}
It follows from
\begin{equation*}\begin{split}
\lim_{k\mapsto \infty} \big|\overline{h}^{k+1}\big|_{\infty}
 \leq &\,C\lim_{k\mapsto\infty} \big|\overline{\psi}^{k+1}\big|^{\frac{1}{2}}_{2}\big|\nabla \overline{\psi}^{k+1}\big|^{\frac{1}{2}}_{2}
 \leq C\lim_{k\mapsto \infty} \big|\overline{\psi}^{k+1}\big|^{\frac{1}{2}}_{2}=0,\\
\lim_{k\mapsto \infty}\big|\overline{\psi}^{k+1}\big|_{6}
 \leq &\, C\lim_{k\mapsto\infty} \big|\overline{\psi}^{k+1}\big|^{\frac{2}{3}}_{\infty}\big|\overline{\psi}^{k+1}\big|^{\frac{1}{3}}_{2}
  \leq C\lim_{k\mapsto\infty} \big|\overline{\psi}^{k+1}\big|^{\frac{1}{3}}_{2}=0
\end{split}
\end{equation*}
that the whole sequence $(h^k, \psi^k,c^k,u^k)$
converges to a limit $(h, \psi, c,u)$ in the following strong sense:
\begin{equation}\label{str}
\begin{split}
&h^k \rightarrow h\,\,\, \text{in $L^\infty([0,T_*]\times \mathbb{R}^3)$},
\quad  \psi^k\rightarrow\psi \,\,\, \text{in $L^\infty([0,T_*];L^6(\mathbb{R}^3))$},\\
&(c^k, u^k)\rightarrow (c, u)\,\,\, \text{in $L^\infty([0,T_*];H^1(\mathbb{R}^3))$}.
\end{split}
\end{equation}

From the uniform estimates (\ref{jkk}),  we have the following weak/weak* convergence:
\begin{equation}\label{weakconvergence}
\begin{split}
\psi^k\rightharpoonup \psi \qquad &\text{weakly*  in $L^\infty([0,T_*];D^1\cap D^2)$},\\
U^k\rightharpoonup U \qquad &\text{weakly* in $L^\infty([0,T_*];H^3)$},\\
\nabla u^{k} \rightharpoonup  \nabla u \qquad &\text{weakly in $L^2([0,T_*];H^3)$},
\end{split}
\end{equation}
which, along with the weakly lower semi-continuity of norms, implies  that $(h, \psi,c, u)$ satisfies
\begin{equation*}
\begin{split}
\epsilon \big\|\psi(t)\big\|^2_{D^1\cap D^2}+\big\| U(t)\big\|^2_{3}
 +\int_{0}^{t} \epsilon \sum_{i=1}^4 \nu^2\big|\nabla^i u\big|^2_{2}\text{d}s\leq C d^2_0.
\end{split}
\end{equation*}
Then it follows from the above uniform estimates, the strong convergence in (\ref{str}),
and the weak convergence in (\ref{weakconvergence}) that
\begin{equation}\label{weakconvergence2}
\begin{split}
h^k\nabla^i u^k \rightharpoonup  h\nabla^i u \qquad \text{weakly in $L^2([0,T_*];L^2)\,\,$ for $i=1,2,3,4$,}
\end{split}
\end{equation}
which, along with the  weakly lower semi-continuity of norms again, implies that $(h, \psi,c, u)$ satisfies
the uniform estimates (\ref{jkk}).
Moreover, it follows from  (\ref{str})--(\ref{weakconvergence2}) and  the uniform estimates (\ref{jkk})
that $(h,\psi,U)$  satisfies   \eqref{middle}--\eqref{psiequation}
in the sense of distributions, respectively,  with the following regularity:
\begin{equation}\label{rjkqq}
\begin{split}
&h\in L^\infty ([0,T_*]\times \mathbb{R}^3),\quad\,\,    \psi \in L^\infty ([0,T_*];H^2),\\
&c \in L^\infty([0,T_*];H^3),\quad \quad u\in L^\infty([0,T_*]; H^3)\cap L^2([0,T_*]; H^4).
\end{split}
\end{equation}

Finally, we  need to verify the following  relation:
\begin{equation}\label{relationrecover}
\psi^{\epsilon}=\nabla  h.
\end{equation}
Denote
$\psi^*=\nabla  h$ and $ \overline{\psi}^*=\psi-\psi^*$.
Then, using equations  $(\ref{middle})_1$--$(\ref{middle})_2$, we have
\begin{equation}\label{liyhn-recover}
\begin{cases}
\displaystyle
\overline{\psi}^*_t+\sum_{l=1}^3 B_l(u) \partial_l\overline{\psi}^*+B(u)\overline{\psi}^*=0,\\[12pt]
\displaystyle
\overline{\psi}^*|_{t=0}=0   \qquad \text{in $\mathbb{R}^3$},
 \end{cases}
\end{equation}
which, together with a standard energy method, implies
$$
\overline{\psi}^*=0\ \qquad \text{for $(t,x)\in [0,T_*]\times \mathbb{R}^3$}.
$$
Then the relation in (\ref{relationrecover}) is verified.

\medskip
\paragraph{{\rm 2}. \em  Uniqueness.}
 Let $(h_1, c_1,u_1)$ and $(h_2,  c_2,u_2)$ be two solutions
 of the Cauchy problem (\ref{middle}) obtained in Step 1, and let $\psi_i=\nabla h_i$, $i=1,2$.
Define
\begin{equation*}\begin{cases}
\displaystyle \overline{h}:=h_1-h_2,\quad \overline{\psi}:=\psi_1-\psi_2=\nabla h_1-\nabla h_2,\\[3pt]
\displaystyle
\overline{c}:=c_1-c_2,\quad \  \overline{u}:=u_1-u_2, \quad \overline{U}:=U_1-U_2.
\end{cases}
\end{equation*}
Then $(\overline{h}, \overline{\psi}, \overline{c},\overline{u})$ satisfies the
following problem:
 \begin{equation}\label{zhuzhu}
\left\{\begin{aligned}
\displaystyle
&\overline{h}_t+u_1\cdot \nabla\overline{h}+\overline{u}\cdot\psi_2
+\frac{\delta-1}{2}\big(\overline{h} \text{div}u_2 +h_1\text{div}\overline{u}\big)=0,\\
 \displaystyle
&\overline{\psi}_t+\sum_{l=1}^3 B_l(u_1)
\partial_l\overline{\psi}+B(u_{1})\overline{\psi}
+\frac{\delta-1}{2}\big(\overline{h} \nabla\text{div}u_2 +h_1\nabla \text{div}\overline{u}\big)\\
\displaystyle
&\quad =-\sum_{l=1}^3B_l(\overline{u})
\partial_l\psi_{2}-B(\overline{u}) \psi_{2},\\[2pt]
\displaystyle
& A_0\overline{U}_t+\sum_{l=1}^3A_l(U_1)\partial_l\overline{U}+
\epsilon F(\nu,h_1, \overline{u}) \\
 \displaystyle
&\quad =-\sum_{l=1}^3 A_l(\overline{U})\partial_l U_2-\epsilon\big(F(\nu,h_1, u_2)-F(\nu,h_2, u_2)\big)\\
\displaystyle
&\qquad +
\epsilon\Big( G(h_1,\overline{\psi}, u_1)+G(h_1,\psi_2, \overline{u})+G(\overline{h}, \psi_2, u_2)\Big),\\
 \displaystyle
&(\overline{h}, \overline{\psi}, \overline{U})|_{t=0}=(0,0,0)\qquad\,\, \mbox{for $x\in\mathbb{R}^3$}.
\end{aligned}
\right.
\end{equation}

Denote
$$
\Phi(t)=\epsilon\big(\big|\overline{h}(t)\big|^2_{ 2}
  +\big|\overline{\psi}(t)\big|^2_{2}\big)+\big\|(\overline{c},\overline{u})(t)\big\|^2_1.
$$
Similarly to the derivation of (\ref{difference1})--(\ref{difference2}) and (\ref{gogo13}), we can show
\begin{equation}\label{gonm}\begin{split}
\frac{\rm d}{{\rm d} t}\Phi(t)+\epsilon \big(\big|h_1\nabla \overline{u}(t)\big|^2_2
     +\big|h_1\nabla^2 \overline{u}(t)\big|^2_2+\nu^2\big\|\nabla \overline{u}(t)\big\|^2_1\big)
\leq G(t)\Phi (t),
\end{split}
\end{equation}
where $ \int_{0}^{t}G(s)\,\dd s\leq C$ for $0\leq t\leq T_*$.
From the
Gronwall inequality, we conclude that
$$
\overline{h}=\overline{\psi}=\overline{c}=\overline{u}=0,
$$
which gives the desired  uniqueness.

\medskip
\paragraph{{\rm 3}. \em  Time-continuity.} The time-continuity  follows easily from  the same  procedure as in Lemma \ref{lem1}.

This completes the proof.
\end{proof}

\smallskip
\begin{remark}\label{cancelproof}
For completeness, we now give the proof of Lemma {\rm\ref{cancel}}, which is  quite direct.

\medskip
\noindent
{\it Proof of Lemma {\rm\ref{cancel}}}.
We define
$\chi_R(x)=\chi(\frac{x}{R})$
with $\chi(x)\in C^\infty_c(\mathbb{R}^3)$ as a truncation function
satisfying \eqref{eq:2.6-77A}.

Denote $\overline{h}^{k+1,R}=\overline{h}^{k+1} \chi_R$. From $(\ref{eq:1.2w})_1$, we have
\begin{align}
\displaystyle
& \overline{h}^{k+1,R}_t+u^k\cdot \nabla\overline{h}^{k+1,R} +\frac{\delta-1}{2}(\overline{h}^{k,R} \text{div}\,u^{k-1}
  +h^{k}\text{div}\,\overline{u}^k\chi_R)\nonumber\\[2pt]
&=u^k\overline{h}^{k+1}\cdot \nabla \chi_R-\overline{u}^k\cdot\psi^{k}\chi_R. \label{localversion}
\end{align}
Then multiplying the above equation by $2\overline{h}^{k+1,R}$ and integrating over $\mathbb{R}^3$ yield
 \begin{align}
\displaystyle
\frac{\rm d}{{\rm d} t}\big|\overline{h}^{k+1,R}\big|_2
\leq &\, C\big|\nabla u^k\big|_\infty\big|\overline{h}^{k+1,R}\big|_2
    +C\big(\big|\overline{h}^{k}\big|_\infty \big|\text{div}\,u^{k-1}\big|_2
    +\big|h^{k}\big|_\infty\big|\text{div}\,\overline{u}^k\big|_2\big)\nonumber\\[3pt]
&\,+C\big(\big|u^k\big|_2\big|\overline{h}^{k+1}\big|_\infty
    +\big|\overline{u}^k\big|_2\big|\psi^{k}\big|_\infty\big)\nonumber\\[3pt]
\leq &\, \hat{C}\big|\overline{h}^{k+1,R}\big|_2+\hat{C}, \label{localversion2}
\end{align}
where $\hat{C}>0$ is a constant depending on the generic constant $C$ and $(d_0, d_1,\eta)$, but
is independent of $R$.
Then, using the Gronwall inequality, we have
$$
\big|\overline{h}^{k+1,R}(t)\big|_2\leq  \hat{C}\exp{(\hat{C}T^{**})} \qquad \text{for $(t,R)\in [0,T^{**}]\times [0,\infty)$}.
$$
Then $\overline{h}^{k+1}\in L^\infty([0,T^{**}]; L^2(\mathbb{R}^3))$, which, along with
$\overline{h}^{k+1}=h^{k+1}-h^k$ and
$$
\nabla h^k=\psi^k\in L^\infty([0,T^{**}];D^1\cap D^2(\mathbb{R}^3)),
$$
implies that $\overline{h}^{k+1}\in L^\infty([0,T^{**}]; H^3(\mathbb{R}^3))$.
\end{remark}

\subsection{Passing to the limit as $\nu \rightarrow 0$}
Now we consider the following nonlinear problem for $(h,c,u)$:
\begin{equation}
\label{li4111}
\begin{cases}
 \displaystyle
h_t+u\cdot \nabla h+\frac{\delta-1}{2}h\text{div} u=0,\\[2pt]
\displaystyle
A_0 U_t+\sum^{3}_{j=1}A_j(U)\partial_j U=-\epsilon F(h, u)+\epsilon G(h,\psi, u),\\[4pt]
  \displaystyle
(h, c,u)|_{t=0}=(h_0,  c_0,u_0)(x)
=((A\gamma)^{-\frac{\iota}{2}}(c_0+\eta)^{\iota}, c_0, u_0)(x)\qquad\mbox{for $x\in \mathbb{R}^3$},
\end{cases}
\end{equation}
where $\psi=\nabla h$ satisfying equations \eqref{psiequation}.

The main result in this subsection can be stated as follows:

\begin{theorem}\label{lem111}
Let \eqref{canshu} hold, and let $(\eta,\epsilon)\in (0,1]\times (0,1]$.
Assume that the initial data  $(c_{0}, u_{0})$ satisfy
\eqref{zhenginitial}.
Then there exist $T_*>0$ independent of $(\eta,\epsilon)$ and  a unique strong solution $(h,c, u)$
of problem \eqref{li4111} in $[0,T_*]\times \mathbb{R}^3$ satisfying the uniform  estimates \eqref{jkk},
\begin{equation}\label{reggh111}\begin{split}
&h\in L^\infty\cap C ([0,T_*]\times \mathbb{R}^3),\quad   \psi \in C([0,T_*];H^2),\\
& c \in C([0,T_*];H^3),\quad   u\in C([0,T_*]; H^3)\cap L^2([0,T_*]; H^4),
\end{split}
\end{equation}
and
$$
\frac{C}{2}\big(1+|c_0|_\infty\big)^\iota \leq h^{\eta, \epsilon}(t,x)
\leq C\eta^\iota \qquad \text{for $(t,x)\in [0,T_*]\times \mathbb{R}^3$},
$$
where $C>0$ is a constant independent of $(\eta, \epsilon)$.
\end{theorem}

\begin{proof} We prove the existence, uniqueness, and time-continuity in two steps.

\smallskip
\paragraph{{\rm 1}. \em  Existence.}
First, from Theorem  \ref{th1zx111},
for every $(\nu,\eta, \epsilon)\in (0,1]\times(0,1]\times (0,1]$,
there exist $T_*>0$ independent of $(\nu,\eta,\epsilon)$ and   a unique  strong solution
$$
(h^{\nu,\eta, \epsilon}, c^{\nu,\eta, \epsilon}, u^{\nu,\eta, \epsilon})(t,x) \qquad\mbox{in $[0,T_*]\times \mathbb{R}^3$}
$$
of problem (\ref{middle}) satisfying the estimates in  (\ref{jkk}), which are independent of
$(\nu,\eta,\epsilon)$.
Moreover, applying the characteristic method and the standard energy estimates for the transport equations,
and using equations $(\ref{middle})_1$--$(\ref{middle})_2$ for $h^{\nu,\eta, \epsilon}$ and $\psi^{\nu,\eta, \epsilon}$,
we have
\begin{equation}\label{related}
\begin{split}
&\frac{C}{2} \big(1+|c_0|_\infty\big)^\iota \leq h^{\nu,\eta, \epsilon}(t,x) \leq C\eta^\iota \qquad
  \text{for $(t,x)\in [0,T_*]\times \mathbb{R}^3$},\\
&\big\|h^{\nu,\eta, \epsilon}(t)\big\|_{ D^1}+ \big|h^{\nu,\eta, \epsilon}_t(t)\big|_{2}
\leq \hat{C}(\eta, \alpha, \beta, \gamma, \delta, T_*,c_0,u_0)
  \qquad \text{for $0\leq t\leq T_*$}.
\end{split}
\end{equation}

Then,  by virtue of the uniform estimates in (\ref{jkk}) independent of  $(\nu,\eta, \epsilon)$,
the estimates in (\ref{related}) independent of $(\nu, \epsilon)$,
and  the compactness in Lemma \ref{aubin} (see \cite{jm}),
we obtain that, for any $R> 0$,  there exists a subsequence of solutions (still denoted by)
$(h^{\nu,\eta, \epsilon}, c^{\nu,\eta, \epsilon},  u^{\nu,\eta, \epsilon})$ converging
to  a limit $(h^{\eta, \epsilon}, c^{\eta, \epsilon}, u^{\eta, \epsilon})$ as  $ \nu \rightarrow 0$
in the following  strong sense:
\begin{equation}\label{ert1}\begin{split}
&(h^{\nu,\eta, \epsilon}, c^{\nu,\eta, \epsilon}, u^{\nu,\eta, \epsilon} )
  \rightarrow (h^{\eta, \epsilon}, c^{\eta, \epsilon},  u^{\eta, \epsilon})
\qquad \text{ in $C([0,T_*];H^2(B_R))$}
\end{split}
\end{equation}
for any finite constant $R>0$, where $B_R$ is a ball centered at the origin with radius $R$.

Again, it follows from the uniform estimates in (\ref{jkk}) and (\ref{related})
that there exists a further subsequence (of the subsequence chosen above) of solutions (still denoted by)
$(h^{\nu,\eta, \epsilon}, c^{\nu,\eta, \epsilon},  u^{\nu,\eta, \epsilon})$ converging to
$(h^{\eta, \epsilon}, c^{\eta, \epsilon},  u^{\eta, \epsilon})$ as $\nu \rightarrow 0$ in the following sense:
\begin{equation}\label{ruojixian}
\begin{split}
(c^{\nu,\eta, \epsilon},  u^{\nu,\eta, \epsilon})\rightharpoonup  (c^{\eta, \epsilon},  u^{\eta, \epsilon})
         \qquad &\text{weakly* in $L^\infty([0,T_*];H^3)$},\\
h^{\nu,\eta, \epsilon} \rightharpoonup h^{\eta, \epsilon} \qquad &\text{weakly* in $L^\infty([0,T_*]\times \mathbb{R}^3)$},\\
\psi^{\nu,\eta, \epsilon}=\nabla h^{\nu,\eta, \epsilon} \rightharpoonup \psi^{\eta, \epsilon} \qquad &\text{weakly* in $L^\infty([0,T_*];H^2)$},\\
\nabla u^{\nu,\eta, \epsilon} \rightharpoonup  \nabla u^{\eta, \epsilon} \qquad &\text{weakly in $L^2([0,T_*];H^3)$},
\end{split}
\end{equation}
which, along with  the lower semicontinuity of weak/weak* convergence of the norms,
implies  that $(h^{\eta, \epsilon},  \psi^{\eta, \epsilon},  c^{\eta, \epsilon},  u^{\eta, \epsilon})$ also satisfies the corresponding
estimates in (\ref{jkk}) and (\ref{related}), except those estimates on $h^{\eta,\epsilon}\nabla^i u^{\eta,\epsilon}$ for $i=1,2,3,4$.

Together with the uniform estimates on $(h^{\eta, \epsilon}, u^{\eta, \epsilon})$ obtained above,
the strong convergence in (\ref{ert1}), and the weak/weak* convergence in  (\ref{ruojixian}), we obtain that, for $i=1,2,3,4$,
\begin{equation}\label{ruojixian2}
\begin{split}
h^{\nu, \epsilon,\eta}\nabla^i u^{\nu, \eta, \epsilon} \rightharpoonup  h^{\eta, \epsilon}\nabla^i u^{\eta, \epsilon}
\qquad &\text{weakly in $L^2([0,T_*]; L^2)$},
\end{split}
\end{equation}
which, along with the lower semicontinuity of weak convergence again, implies
that $(h^{\eta, \epsilon}, u^{\eta, \epsilon})$ also satisfies the uniform estimates
on $h^{\eta,\epsilon}\nabla^i u^{\eta,\epsilon}$ for $i=1,2,3,4$.

Now we show that $(h^{\eta, \epsilon}, c^{\eta, \epsilon},  u^{\eta, \epsilon})$ is a weak solution of problem (\ref{li4111})
in the sense of distributions.
First, multiplying the equations for the time evolution of $u$  in $(\ref{li4111})$
by a test function  $w(t,x)=(w^1,w^2,w^3)\in C^\infty_c ([0,T_*)\times \mathbb{R}^3)$
on both sides and integrating over $[0,T_*]\times \mathbb{R}^3$, we have
\begin{equation}\label{zhenzheng1}
\begin{split}
&\int_0^{T_*} \int_{\R^3}  \Big(u^{\nu,\eta, \epsilon} \cdot w_t - (u^{\nu,\eta, \epsilon} \cdot \nabla) u^{\nu,\eta, \epsilon} \cdot w
  +\frac{1}{\gamma-1}\big(c^{\nu,\eta, \epsilon}\big)^2\text{div} w \Big)\text{d}x\text{d}t\\
&\quad +\int_{\R^3} u_0(x) \cdot w(0,x)\,\dd x\\
&=\int_0^{T_*} \int \epsilon \Big( \big((h^{\nu,\eta, \epsilon})^2+\nu^2\big) Lu^{\nu,\eta, \epsilon}
   -\frac{2\delta}{\delta-1}h^{\nu,\eta, \epsilon}\psi^{\nu,\eta, \epsilon} \cdot Q(u^{\nu,\eta, \epsilon})\Big)\cdot w\, \text{d}x\text{d}t.
\end{split}
\end{equation}
Combining the uniform estimates obtained above with the strong convergence in (\ref{ert1})
and the weak convergence in (\ref{ruojixian})--(\ref{ruojixian2}),
and letting $\nu \rightarrow 0$ in (\ref{zhenzheng1}), we obtain
\begin{equation}\label{zhenzhengxx}
\begin{split}
&\int_0^{T_*} \int_{\R^3}  \Big(u^{\eta, \epsilon} \cdot w_t - (u^{\eta, \epsilon} \cdot \nabla) u^{\eta, \epsilon} \cdot w
  +\frac{1}{\gamma-1}\big(c^{\eta, \epsilon}\big)^2\text{div} w \Big)\text{d}x\text{d}t\\
&\quad +\int_{\R^3} u_0(x) \cdot w(0,x)\, \text{d} x\\
&=\int_0^{T_*} \int_{\R^3} \epsilon \Big( (h^{\eta, \epsilon})^2 Lu^{\eta, \epsilon}
  -\frac{2\delta}{\delta-1}h^{\eta, \epsilon}\psi^{\eta, \epsilon} \cdot Q(u^{\eta, \epsilon}) \Big)\cdot w\, \text{d}x\text{d}t.
\end{split}
\end{equation}

Similarly, we can use the same argument to show that $(h^{\eta, \epsilon},  c^{\eta,\epsilon},  u^{\eta,\epsilon})$
satisfies the other equations in  $(\ref{li4111})$ and the corresponding initial data in the sense of distributions.
Thus, $(h^{\eta, \epsilon}, c^{\eta, \epsilon},  u^{\eta, \epsilon})$ is a weak solution of problem (\ref{li4111})
in the sense of distributions satisfying the  following regularity:
\begin{equation}\label{zheng}
\begin{split}
&h^{\eta, \epsilon}\in (L^\infty\cap C) ([0,T_*]\times \mathbb{R}^3),\quad   \psi^{\eta, \epsilon} \in L^\infty([0,T_*];H^2),\\
& c^{\eta, \epsilon} \in L^\infty([0,T_*];H^3),\quad   u^{\eta, \epsilon}\in L^\infty([0,T_*]; H^3)\cap L^2([0,T_*]; H^4).
\end{split}
\end{equation}
Therefore, this  solution $(h^{\eta, \epsilon}, c^{\eta, \epsilon},  u^{\eta, \epsilon})$ of problem \eqref{li4111}
is actually a strong solution.
The relation $\psi^{\eta, \epsilon}=\nabla h^{\eta, \epsilon}$ can be verified
by the same argument used in the proof of (\ref{relationrecover}).

\medskip
\paragraph{{\rm 2}. \em  Uniqueness and time-continuity.}
Owing to the lower bound estimate of $h^{\eta, \epsilon}$:
$$
h^{\eta, \epsilon}(t,x)\geq \frac{C}{2} \big(1+|c_0|_\infty\big)^\iota\qquad\,\, \text{for $(t,x)\in [0,T_*]\times \mathbb{R}^3$},
$$
the uniqueness and time-continuity of the strong solutions obtained above
can be proved via the similar argument to Theorem  \ref{th1zx111}; hence we omit its details.
\end{proof}

\subsection{Passing to the limit as $\eta \rightarrow 0$}
Based on the conclusions obtained in the above subsections,
we are now ready to prove Theorem \ref{th1}.

\begin{proof} We divide the proof into three  steps.

\smallskip
\paragraph{{\rm 1}. \em  Existence.}
As in \S 4.2, we define constants $d_i$, $i=0,1$.
For any $\eta\in (0,1)$, set
$$
h^\eta_0=(A\gamma)^{-\frac{\iota}{2}}\big(c_0+\eta\big)^{\iota},
\quad  \psi^\eta_{ 0}=(A\gamma)^{-\frac{\iota}{2}}\nabla \big(c_0+\eta\big)^{\iota}.
$$
Then there exists $\eta_{1}>0$ such that, if $0<\eta<\eta_{1}$,
\begin{equation*}\begin{split}
\eta+\epsilon^{\frac{1}{2} } \|\psi^\eta_{0}\|_{D^1\cap D^2}+|(h^{\eta}_0)^{-1}|_{\infty}+\epsilon^{\frac{1}{4} }|\nabla c^{\frac{\iota}{2}}_0|_4
+\|(c_0,u_0)\|_3\leq&\, \overline{d}_0,
\end{split}
\end{equation*}
where we have used the fact that $\epsilon^{\frac{1}{4}}\nabla c^{\frac{\iota}{2}}_0 \in  L^4$,
and $\overline{d}_0$ is a positive constant independent of $\eta$.
Therefore, taking $(h^\eta_{0},  c_0, u_0)$ as the initial data,
problem (\ref{li4111})  admits a unique strong  solution $(h^{\eta, \epsilon},  c^{\eta, \epsilon},  u^{\eta, \epsilon})$
in $[0,T_*]\times \mathbb{R}^3$ satisfying the local estimates in (\ref{jkk}) with $d_0$ replaced by $\overline{d}_0$,
and the life span $T_*$ is also independent of $(\eta,\epsilon)$.
Moreover, we also know that
$$
h^{\eta, \epsilon}\geq \frac{C}{2} (1+|c_0|_\infty)^\iota\qquad \text{for $(t,x)\in [0,T_*]\times \mathbb{R}^3$},
$$
where $C>0$ is a constant independent of $(\eta, \epsilon)$.

We first state the following lemma:
\begin{lemma}\label{localupperbound}
For any $R_0>0$ and $(\eta,\epsilon) \in (0,1]\times (0,1]$,
there exists a constant $b_{R_0}>0$ such that
\begin{equation}\label{localbound}
h^{\eta,\epsilon}(t,x)\leq b_{R_0} \qquad \mbox{for any $(t,x)\in [0,T_*] \times B_{R_0}$},
\end{equation}
where $b_{R_0}>0$ is a constant independent of $(\eta,\epsilon)$.
\end{lemma}

\begin{proof} It suffices to consider the case when $R_0$ is sufficiently large.

First, from the  initial assumptions on $c_0$:
$$
c_0(x)>0,\quad c_0 \in  H^3,
$$
we obtain that the initial vacuum occurs only in the far-field.
Then, for every $R'>2$, there exists a constant $C_{R'}$ independent of $(\eta,\epsilon)$ such that
$$
c^\eta_0(x)  \geq C_{R'}+\eta>0 \qquad\,\, \mbox{for any $x\in  B_{R'}$},
$$
which implies
\begin{equation}\label{local1}
h^\eta_0(x)  \leq (A\gamma)^{-\frac{\iota}{2}}(C_{R'}+2\eta)^\iota \leq (A\gamma)^{-\frac{\iota}{2}}C^{\iota}_{R'} \qquad
\,\,\mbox{for any $x\in  B_{R'}$}.
\end{equation}

Second, let $x^{\eta,\epsilon}(t;x_0)$ be the particle path starting from $x_0$ at $t=0$:
\begin{equation}\label{local2}
\frac{\rm d}{\text{d}t}x^{\eta,\epsilon}(t;x_0)=u^{\eta,\epsilon}(t,x(t;x_0)),\quad x^{\eta,\epsilon}(0;x_0)=x_0.
\end{equation}
Denote by $B(t,R')$ the closed regions that are the images of $B_{R'}$  under the flow map (\ref{local2}):
\begin{equation*}
B(t,R')=\{x^{\eta,\epsilon}(t;x_0)\;:\; x_0\in B_{R'}\}.
\end{equation*}
From equation $ (\ref{li4111})_1$, we have
\begin{equation}\label{local3}
h^{\eta,\epsilon}(t,x)
=h^\eta_0(x_0)\exp\Big(-\frac{\delta-1}{2}\int_{0}^{t}\textrm{div} u^{\eta,\epsilon}(s,x^{\eta,\epsilon} (s;x_0))\,\text{d}s\Big).
\end{equation}
According to (\ref{jkk}), for $ 0\leq t \leq T_*$,
\begin{equation}\label{local4}
\begin{split}
&\int_0^t\big|\textrm{div} u^{\eta,\epsilon}(t,x^{\eta,\epsilon}(t;x_0))\big|\,\text{d}s
  \leq \int_0^t\big|\nabla u^{\eta,\epsilon}\big|_\infty\,\text{d}s\\
&\leq  \int_0^t \big\|\nabla u^{\eta,\epsilon}\big\|_2\,\text{d}s
 \leq t^{\frac{1}{2}}\Big(\int_0^t \big\|\nabla u^{\eta,\epsilon}\big\|^2_2\,\text{d}s\Big)^{\frac{1}{2}}
 \leq d_1T^{\frac{1}{2}}_*.
\end{split}
\end{equation}
Thus,  by  \eqref{local1} and  \eqref{local3}--\eqref{local4}, we see that, for $ 0\leq t \leq T_*$,
\begin{equation}\label{local5}
h^{\eta,\epsilon}(t,x)\leq  (A\gamma)^{-\frac{\iota}{2}}C^*C^\iota_{R'} \qquad\,\, \mbox{for any $x\in  B(t,R')$},
\end{equation}
where $C^*=\exp\big(\frac{1}{2}d_1T^{\frac{1}{2}}_*\big)$.

Finally, from  problem (\ref{local2}) and the estimates in (\ref{jkk}), we have
\begin{equation*}
|x_0-x|=|x_0-x^{\eta,\epsilon}(t;x_0)|
\leq  \int_0^t| u^{\eta,\epsilon}(\tau,x^{\eta,\epsilon}(\tau;x_0))|\,\text{d}\tau\leq d_1t\leq 1\leq  \frac{R'}{2}
\end{equation*}
for all $(t,x)\in [0,T_*]\times B_R$,  which implies that
$B_{R'/2} \subset B(t,R')$.
Therefore, we choose
$$
R'=2R_0,\qquad b_{R_0}=(A\gamma)^{-\frac{\iota}{2}}C^*C^\iota_{R'}
$$
to complete the proof.
\end{proof}

Then, for any $R>0$, it follows from Lemmas \ref{localupperbound} and \ref{aubin} that
there exists a subsequence (still denoted by) $(h^{\eta, \epsilon},  \psi^{\eta, \epsilon}, c^{\eta, \epsilon},  u^{\eta, \epsilon})$ such that
\begin{equation}\label{ert}
\begin{split}
(h^{\eta,\epsilon},c^{\eta, \epsilon},  u^{\eta, \epsilon})
  \rightarrow (h^\epsilon, c^{\epsilon},  u^{ \epsilon}) \qquad &\text{in $C([0,T_*];H^2(B_R))$},\\
\psi^{\eta, \epsilon}\rightarrow \psi^{ \epsilon}\qquad  &\text{in $C([0,T_*];H^1(B_R))$}.
\end{split}
\end{equation}
This, together with Lemma \ref{localupperbound}, yields the following lemma:

\begin{lemma}\label{localupperbound11}
For any $R_0>0$ and $\epsilon \in (0,1]$, there exists a constant $b_{R_0}>0$ such that
\begin{equation}\label{localbound111}
h^{\epsilon}(t,x)\leq b_{R_0} \qquad\,\, \mbox{for any $(t,x)\in [0,T_*] \times B_{R_0}$},
\end{equation}
where $b_{R_0}$ is independent of $\epsilon$.
\end{lemma}

Next, note that estimates  (\ref{jkk}) are independent of $(\eta,\epsilon)$.
Then there exists a subsequence (still denoted by) $(h^{\eta, \epsilon},\psi^{\eta, \epsilon}, c^{\eta, \epsilon},  u^{\eta, \epsilon})$
converging to a limit $(h^{\epsilon}, \psi^{\epsilon}, c^{\epsilon},  u^{ \epsilon})$ in the weak or weak* sense:
\begin{equation}\label{ruojixianas}
\begin{split}
(c^{\eta, \epsilon},  u^{\eta, \epsilon})\rightharpoonup  (c^{\epsilon},  u^{ \epsilon}) \qquad &\text{weakly* in $L^\infty([0,T_*];H^3)$},\\
\psi^{\eta,\epsilon}  \rightharpoonup  \psi^\epsilon \qquad &\text{weakly* in $L^\infty([0,T_*];D^1\cap D^2)$},\\
 u^{\eta,\epsilon}  \rightharpoonup   u^\epsilon \qquad &\text{weakly in $L^2([0,T_*];H^4)$}.
\end{split}
\end{equation}
From the lower semicontinuity of the weak convergence,  $(h^{\epsilon}, \psi^{\epsilon},c^{\epsilon},  u^{ \epsilon})$ also
satisfies the corresponding  estimates (\ref{jkk}), except those of $h^\epsilon\nabla^i  u^\epsilon$ for $i=1,2,3,4$.

Together with the uniform estimates of $(h^{\epsilon}, \psi^{\epsilon},c^{\epsilon},  u^{ \epsilon})$ obtained above,
the strong convergence in (\ref{ert}),  and  the weak or weak* convergence in  (\ref{ruojixianas}),
we obtain that, for $i=1,2,3,4$,
\begin{equation}\label{ruojixianas2}
h^{\epsilon,\eta}\nabla^i u^{\epsilon,\eta} \rightharpoonup  h^{\epsilon}\nabla^i u^{\epsilon}\qquad \text{weakly in $L^2([0,T_*]; L^2)$},
\end{equation}
which, along with the lower semicontinuity of weak or weak* convergence again,
implies that $(h^{ \epsilon}, u^{ \epsilon})$ also satisfies the uniform estimates of $h^{\epsilon}\nabla^i u^{\epsilon}$ for $i=1,2,3,4$.

Thus, it is easy to show that $(h^\epsilon,  c^{\epsilon}, u^{\epsilon}) $ solves the following Cauchy problem
in the sense of distributions:
\begin{equation}
\label{eq:cccq2-reformlate}
\begin{cases}
\displaystyle h^\epsilon_t+u^\epsilon\cdot \nabla h^\epsilon+\frac{\delta-1}{2}h^\epsilon\text{div} u^\epsilon=0,\\[3pt]
\displaystyle
A_0 U^\epsilon_t+\sum^{3}_{j=1}A_j(U^\epsilon)\partial_j U^\epsilon
  = -\epsilon F(h^\epsilon, u^\epsilon)+ \epsilon G(h^\epsilon, \psi^\epsilon, u^\epsilon),
 \end{cases}
\end{equation}
with the following initial data
\begin{equation}\label{initialnew-reformulate}
 (h^\epsilon, c^\epsilon,u^\epsilon)|_{t=0}=(h_0, c_0,u_0)(x)
 =(\rho^{\frac{\delta-1}{2}}_0,  \sqrt{A\gamma}\rho^{\frac{\gamma-1}{2}}_0, u_0)(x) \qquad\mbox{for $x\in \mathbb{R}^3$},
\end{equation}
so that
\begin{equation}\label{farnew-reformulate}
(\rho_0, u_0)\rightarrow (0, 0) \qquad \text{as $|x|\rightarrow \infty$}.
\end{equation}
Moreover, in this step, even though the vacuum appears in the far-field, $\psi^{\epsilon}$ satisfies
$\partial_i (\psi^{\epsilon})^{(j)}=\partial_j (\psi^{\epsilon})^{(i)}$, $i,j=1,2,3$, and the following equation:
\begin{equation}\label{psiepsilon}
\displaystyle \psi^\epsilon_t+\sum_{l=1}^3 B_l(u^\epsilon) \partial_l\psi^\epsilon+ B(u^\epsilon)\psi^\epsilon
+ \frac{\delta-1}{2}h^\epsilon\nabla \text{div} u^\epsilon=0
\end{equation}
in the sense of distributions.

Finally, we need to verify the following  relations
\begin{equation}\label{relationback}
\psi^{\epsilon}=\nabla  h^{\epsilon}, \qquad\,\,  h^{\epsilon}=(A\gamma)^{-\frac{\iota}{2}}(c^\epsilon)^\iota.
\end{equation}
The first relation can be verified by the same argument used in the proof of (\ref{relationrecover}).
For the second, we denote
$$
h^*=(A\gamma)^{-\frac{\iota}{2}}(c^\epsilon)^\iota, \qquad\,\, \overline{h}^*=h^*-h^\epsilon.
$$
Then using equations  $(\ref{eq:cccq2-reformlate})_1$ and $(\ref{eq:cccq2-reformlate})_3$ leads to
\begin{equation}\label{liyhn-recoverback}
\begin{cases}
\displaystyle
\overline{h}^*_t+u^\epsilon\cdot \nabla \overline{h}^*+\frac{\delta-1}{2}\overline{h}^*\text{div} u^\epsilon=0,\\[3pt]
\displaystyle
\overline{h}^*|_{t=0}=0    \qquad \text{in $\mathbb{R}^3$},
\end{cases}
\end{equation}
which, together with the standard energy method, implies
$$
\overline{h}^*=0\ \qquad \text{for $(t,x)\in [0,T_*]\times \mathbb{R}^3$}.
$$
Then the relations in (\ref{relationback}) have been verified.

\medskip
\paragraph{{\rm 2}. \em  Uniqueness.}
Owing to the lower bound estimate of $h^{ \epsilon}$:
$$
h^{ \epsilon}(t,x)\geq \frac{C}{2} \big(1+|c_0|_\infty\big)^\iota \qquad \text{for $(t,x)\in [0,T_*]\times \mathbb{R}^3$},
$$
the uniqueness of the strong solutions obtained above can be proved
via a similar argument to that used in the proof of Theorem  \ref{th1zx111}, and hence we omit the details.

\medskip
\paragraph{{\rm 3}. \em  Time-continuity.}
First, from the uniform estimates in (\ref{jkk}) and the classical Sobolev embedding theorem,
we obtain that, for any $s'\in (0,3)$ and $s{''}\in (0,1)$,
\begin{equation}\label{zhengW}
\begin{split}
c^{ \epsilon} \in C([0,T_*]; H^{s'}\cap \text{weak-}H^3),
\qquad\, \nabla \psi^{\epsilon}\in C([0,T_*]; L^2 \cap \text{weak-}H^{s{''}}).
\end{split}
\end{equation}

Using similar arguments to the proof of Lemmas \ref{3}--\ref{2} yields
\begin{equation}\label{liu03}
\displaystyle \lim_{t\rightarrow 0}\sup \|c^{ \epsilon} (t)\|_3 \leq \|c_0\|_3,\qquad
\displaystyle \lim_{t\rightarrow 0}\sup \|\psi^{ \epsilon} (t)\|_{D^1\cap D^2} \leq \|\psi_0\|_{D^1\cap D^2},
\end{equation}
which, together with  Lemma \ref{zheng5} and  (\ref{zhengW}), implies  that  $(c^{ \epsilon},\psi^{ \epsilon})$ is
right-continuous at $t=0$ in $H^3$ and $D^1\cap D^2$, respectively.
The time reversibility of the equations in $(\ref{eq:cccq2-reformlate})$ for $(c^{\epsilon},\psi^{\epsilon})$
yields
\begin{equation}\label{xian}
c^{\epsilon} \in C([0,T_*]; H^3),\qquad  \psi^{ \epsilon}\in C([0,T_*]; D^1\cap D^2).
\end{equation}

For velocity $u^\epsilon$, from the equations in $(\ref{eq:cccq2-reformlate})$ for $u^\epsilon$,
Lemma \ref{localupperbound11}, and the classical Sobolev embedding theorem, we conclude that
$u^\epsilon \in C([0,T_*]; H^3_{\rm loc}(\mathbb{R}^3))$.
\end{proof}

\section{Proof of Theorem \ref{th2}}
Based on the conclusion of Theorem \ref{th1}, we are ready to establish the local-in-time well-posedness
for the regular solutions of the original Cauchy problem (\ref{eq:1.1})--(\ref{10000})
with (\ref{initial})--(\ref{far}) shown in Theorem \ref{th2}.
For simplicity, in this section, we denote $(\rho^\epsilon,u^\epsilon,c^\epsilon, \psi^\epsilon,h^\epsilon)$
as $(\rho,u,c,\psi,h)$, and $(\rho^\epsilon_0,u^\epsilon_0, c^\epsilon_0, \psi^\epsilon_0,h^\epsilon_0)$
as $(\rho_0,u_0,c_0,\psi_0,h_0)$, respectively.

\begin{proof} We divide the proof into two steps.

\smallskip
\paragraph{{\rm 1}. \em  Existence and uniform regularity.}
It follows from the initial assumptions (\ref{th78}) and Theorem \ref{th1}
that there exists $T_{*}> 0$ such that the Cauchy problem (\ref{eq:cccq2})--(\ref{farnew})
has a unique strong solution $(\psi, c,u)=((A\gamma)^{-\frac{\iota}{2}}\nabla c^\iota, c,u)$
with the regularity properties (\ref{reformulateregularity}) and  the uniform estimates (\ref{qiyu}).
Denote $\rho=a_3c^{\frac{2}{\gamma-1}}$ with $a_3=(A\gamma)^{-\frac{1}{\gamma-1}}$.
Then we obtain that $\psi=\nabla \rho^{\frac{\delta-1}{2}}=\nabla h$ and
\begin{equation}\label{reg2}
0<\sqrt{A\gamma}\rho^{\frac{\gamma-1}{2}}=c\in C^1([0,T_{*}]\times \mathbb{R}^3),
\qquad (u, \nabla u) \in C((0,T_{*}]\times \Omega)
\end{equation}
for any bounded smooth domain $\Omega\subset \mathbb{R}^3$.

First, multiplying $(\ref{eq:cccq})_2$ by
$
\frac{\partial \rho}{\partial c}=\frac{2a_3}{\gamma-1}c^{\frac{3-\gamma}{\gamma-1}}
$,
we obtain the continuity equation in (\ref{eq:1.1}).

Then multiplying $(\ref{eq:cccq})_3$ by $\rho=a_3c^{\frac{2}{\gamma-1}}$ leads to
the momentum equations in (\ref{eq:1.1}).

Next, we need to show that $(\rho, u)$ also satisfies
\begin{equation}\label{continuous}
\rho^{\frac{1-\delta}{2}}u_t \in L^\infty([0,T]; H^1)\cap L^2([0,T]; H^2),
\qquad  \rho^{\frac{1-\delta}{2}}u \in C([0,T]; H^3).
\end{equation}
In fact, from the momentum equations  $(\ref{eq:1.1})_2$ and the positivity: $\rho(t,x)>0$ for $(t,x)\in [0,T_*]\times \mathbb{R}^3$,
we have
\begin{equation}\label{bianxing}
\rho^{\frac{1-\delta}{2}}\big(u_t+u\cdot\nabla u +\frac{2A\gamma}{\gamma-1}\rho^{\frac{\gamma-1}{2}} \nabla\rho^{\frac{\gamma-1}{2}}\big)
 +\epsilon\rho^{\frac{\delta-1}{2}}  Lu
=\frac{2\delta\epsilon}{\delta-1}\nabla\rho^{\frac{\delta-1}{2}}\cdot  Q(u),
\end{equation}
which, along with the uniform estimates (\ref{qiyu}) and the classical Sobolev embedding theorem,
implies (\ref{continuous}).

Therefore, $(\rho,u)$ satisfies (\ref{eq:1.1})--(\ref{10000}) with (\ref{initial})--(\ref{far})
in the sense of distributions, the regularity properties in Definition \ref{d1},
and the uniform estimates in Theorem \ref{th2}.

\medskip
\paragraph{{\rm 2}. \em  Conservation laws.}
First,  we show that the solution $(\rho, u)$ obtained above still satisfies
$$
\rho  \in L^\infty([0,T_*];L^{1}(\mathbb{R}^3)) \qquad \text{if $\rho_0\in L^{1}$}.
$$

Let $f: \mathbb{R}^{+} \to \mathbb{R}^+$ be a nonincreasing $C^2$ function satisfying
\begin{align*}
f(s)=\begin{cases}
1 & s \in [0,\frac{1}{2}],\\[2pt]
e^{-s} & s \ge 1.
\end{cases}
\end{align*}
Denote $f_R(x)=f(\frac{|x|}{R})$. Using the regularity of the solution and the definition of $f$,
we have
\begin{align*}
\int \big(\rho +|\rho_t|+ |\text{div}(\rho u)|  \big)f_R(x)\, {\rm d}x
+\frac{1}{R}\int \big(\big| \rho u f' (\tfrac{|x|}{R})\big|+ \rho |u| f(\tfrac{|x|}{R})\big)\, {\rm d}x\le C(R)
\end{align*}
for any fixed $R>1$, where $C(R)>0$ is a constant depending only on $R$.

Next, since the continuity equation \eqref{eq:1.1}$_1$ holds, we can multiply  \eqref{eq:1.1}$_1$
by $ f_R(x)$ and integrate with respect to $x$ to obtain
\begin{align}\label{xccv}
\frac{\rm d}{{\rm d} t} \int \rho  f_R(x) \,{\rm d}x=-\int \text{div} (\rho u)  f_R(x)\, {\rm d}x.
\end{align}
Then it follows from integration by parts and H\"older's inequality that
\begin{align*}
-\int \text{div} (\rho u) f_R(x)\,{\rm d}x= \int \rho u \cdot \tfrac{x }{|x|R}  f' (\tfrac{|x|}{R})\,{\rm d}x
\le&\, C \int \rho |u| f(\tfrac{|x|}{R}) \tfrac{1}{R}\,{\rm d}x
\le  C  \frac{|u|_{\infty} }{R} \int \rho    f_R (x)\,{\rm d}x,
\end{align*}
which,  along with (\ref{xccv}) and Gronwall's inequality, implies
\begin{align*}
\sup_{0\leq t \leq T_*} \int \rho  f_R(x)\, {\rm d}x \le C \int \rho_0  f_R(x)\, {\rm d}x + C \le C,
\end{align*}
with $C$ independent of $R$.
Note that
$\rho  f_R(x)  \to \rho$ as $R\to \infty$ for all $x\in \mathbb{R}^3$.
Thus, by  Fatou's lemma, we have
\begin{equation}\label{rhox}
\sup_{0\leq t \leq T_*} \int \rho\,  {\rm d}x
\le
\sup_{0\leq t \leq T_*}  \liminf_{R \to \infty}\int \rho  f_R(x)\,{\rm d}x \le C,
\end{equation}
where $C$ is a constant depending only on $(\alpha, \beta, \gamma, \delta)$ and the initial data $(\rho_0, u_0)$.

Then $(\rho,u)(t,x)$ in $[0,T_*]\times \mathbb{R}^3$ obtained above has finite mass $m(t)$ and  finite  energy $E(t)$.
Indeed,
\begin{equation}\label{finite}
\displaystyle
E(t)=\int \Big( \frac{1}{2}\rho|u|^2+\frac{p}{\gamma-1}\big)(t,x)\, {\rm d}x  \leq C\big(\big|u\big|^2_2+\big|\rho\big|_1\big)<\infty.
\end{equation}

For the conservation laws, since $\rho u\in W^{1,1}(\mathbb{R}^3)$, we have
\begin{equation}\label{conmass}
\frac{\rm d}{{\rm d} t}m(t)=0.
\end{equation}
On the other hand, multiplying equation $(\ref{eq:1.1})_2$ by $u$ on both sides and  integrating over $\mathbb{R}^3$,
via integration by parts and the continuity equation, we have
\begin{equation}\label{conenergy}\begin{split}
\frac{\rm d}{{\rm d} t} E(t)+\int \rho^\delta\big(\alpha  \big|\nabla u\big|^2_2+(\alpha+\beta)\big|\text{div} u\big|^2_2\big)\, {\rm d}x=0,
\end{split}
\end{equation}
where we have used the fact that $\rho u |u|^2, up, \ \rho^\delta \nabla u \cdot u \in W^{1,1}(\mathbb{R}^3)$, and
$$
p_t+u\cdot \nabla p+ p\,\text{div} u=-(\gamma-1)p\,\text{div} u \qquad \text{in $[0,T_*]\times \mathbb{R}^3$}.
$$
This completes the proof.
\end{proof}

\section{Vanishing Viscosity Limit as $\epsilon \rightarrow 0$}

In this section, we establish the vanishing viscosity limit results
stated in Theorem \ref{th3A} and also give the proofs of Corollaries \ref{th3}--\ref{example}.

\smallskip
\subsection{Proof of Theorem \ref{th3A}}
We divide the proof into five steps.

\medskip
1. We first denote by
$$
(\psi^\epsilon, U^\epsilon)=( \psi^\epsilon, c^\epsilon,u^\epsilon)^\top
=( \nabla (\rho^{\epsilon})^{\frac{\delta-1}{2}}, \sqrt{A\gamma}(\rho^{\epsilon})^{\frac{\gamma-1}{2}},u^\epsilon)^\top
$$
as the solution of the following problem:
\begin{equation}
\begin{cases}
\label{E:1.1}
\displaystyle
 \psi^\epsilon_t+\sum_{l=1}^3 B_l(u^\epsilon) \partial_l\psi^\epsilon
  +B(u^\epsilon)\psi^\epsilon+\frac{\delta-1}{2} h^\epsilon\nabla \text{div} u^\epsilon=0,\\[3pt]
\displaystyle
 A_0 U^\epsilon_t+\sum^{3}_{j=1}A_j(U^\epsilon)\partial_j U^\epsilon
 =-\epsilon F(h^\epsilon, u^\epsilon)+\epsilon G(h^\epsilon,\psi^\epsilon, u^\epsilon),\\[3pt]
 (\psi^\epsilon, c^\epsilon,u^\epsilon)|_{t=0}
 =( \nabla (\rho^{\epsilon}_0)^{\frac{\delta-1}{2}}, \sqrt{A\gamma}(\rho^{\epsilon}_0)^{\frac{\gamma-1}{2}},u^\epsilon_0)(x)
  \qquad \mbox{for $x\in\mathbb{R}^3$},
\end{cases}
\end{equation}
where $h^\epsilon=(\rho^\epsilon)^{\frac{\delta-1}{2}}=(A\gamma)^{-\frac{\iota}{2}}(c^\epsilon)^\iota$.
From Theorem \ref{th1}, there exists $T_1>0$ independent of $\epsilon$
such that $(\psi^\epsilon,  U^\epsilon)$ of \eqref{E:1.1} satisfies
\begin{equation}\label{jkka}
  \sup_{0\leq t \leq T_1}\Big(\big\|(c^\epsilon,u^\epsilon)(t,\cdot)\big\|_3^2
   +\epsilon \big\|\psi^\epsilon(t,\cdot)\big\|^2_{D^1\cap D^2}\Big)
   +\epsilon \int_0^{T_1}\sum_{i=1}^4
    \big|h^\epsilon\nabla^i u^\epsilon(t,\cdot)\big|_2^2\,\text{d}t \leq C_0
\end{equation}
for some positive constant $C_0=C_0(\alpha, \beta, A, \gamma, \delta, c_0, \psi_0,u_0)$ that is independent of $\epsilon$.

\medskip
2. Notice that the initial data $(\rho^\epsilon_0,u^\epsilon_0)$ satisfy \eqref{th78}--\eqref{th78--1}
and that there exist a vector function  $(\rho_0(x),u_0(x))$ defined in $\mathbb{R}^3$  such that
$$
\lim_{\epsilon\rightarrow 0}\big|(c^\epsilon_0-c_0,u^\epsilon_0-u_0)\big|_{2}=0.
$$
Then
$
(c_0,u_0)\in H^3(\mathbb{R}^3),
$
due to the lower semi-continuity of the weak convergence.

Based on \cite{tms1}, regarding $(\rho_0,u_0)$ as the initial data,
we denote by
$$
U=(c,u)=(\sqrt{A\gamma}\rho^{\frac{\gamma-1}{2}},u)
$$
as the regular solution of the Cauchy problem \eqref{eq:1.1E} with
\eqref{winitial},
which can be written into the following symmetric system:
\begin{equation}
\label{E:4.2}\begin{cases}
\displaystyle
A_0U_t+\sum_{j=1}^3A_j(U) \partial_j U=0,\\[3pt]
\displaystyle
U(0,x)=U_0=(\sqrt{A\gamma}\rho^{\frac{\gamma-1}{2}}_0,u_0),
\end{cases}
\end{equation}

From \cite{tms1} (see Theorem \ref{thmakio}), we know that there exits $T_{2}$ such that
there is a unique regular solution $U$ of problem \eqref{E:4.2} satisfying
\begin{equation}\label{jkkab}
\sup_{0\leq t \leq T_2}\| U(t,\cdot)\|^2_{3}  \leq C
\end{equation}
for some positive constant $C=C(A, \gamma, c_0, u_0)$.

\smallskip
3. Denote $T_*=\min\{T_1,T_2\}>0$.
Then, for any bounded smooth domain $\Omega$, due to Lemma \ref{localupperbound11} and
the Aubin-Lions lemma  (see \cite{jm}) ({\it i.e.}  Lemma \ref{aubin}),
there exist a subsequence (still denoted by) $(c^{ \epsilon},  u^{ \epsilon})$
and a limit vector function $(c^*,u^*)$ satisfying
\begin{equation}\label{ertfinalq}
(c^{ \epsilon},  u^{ \epsilon})\rightarrow (c^*,u^*) \qquad \text{in $C([0,T_*];H^2(\Omega))$}.
\end{equation}

Note that estimates  (\ref{jkk}) are independent of $\epsilon$. Then
there exists a subsequence (still denoted by) $(c^{ \epsilon},  u^{\epsilon})$ converging
to the  limit function $(c^*,u^*)$ in the weak or weak* sense:
\begin{equation}\label{ruojixianass}
(c^{\epsilon},  u^{ \epsilon})\rightharpoonup  (c^*,u^*) \qquad\,\, \text{weakly* in $L^\infty([0,T_*];H^3)$}.
\end{equation}
From the lower semi-continuity of the norms in the weak convergence,
$(c^*,u^*)$ also satisfies the corresponding  estimates (\ref{jkk}), except those on $h$ and $\psi$.

\medskip
4. Now we show that $(c^*,u^*)$ is a weak solution  of (\ref{E:4.2}) in the sense of distributions.
First, multiplying the equations in  $(\ref{E:1.1})$ for the fluid velocity $u^\epsilon$ by a test function
$w(t,x)=(w^1,w^2,w^3)\in C^\infty_c ([0,T^*)\times \mathbb{R}^3)$ on both sides
and integrating over $[0,T^*]\times \mathbb{R}^3$ yield
\begin{equation}\label{zhenzheng1q}
\begin{split}
&\int_0^{T^*}\int\Big(u^{\epsilon}\cdot w_t -(u^{\epsilon} \cdot \nabla) u^{\epsilon} \cdot w
   +\frac{1}{\gamma-1}\big(c^{\epsilon}\big)^2
   \text{div} w \Big)\text{d}x\text{d}t+\int u^\epsilon_0(x) \cdot w(0,x)\,\dd x\\
&=\int_0^{T^*}\int \epsilon \Big( (h^{\epsilon})^2 Lu^{\epsilon}
  -\frac{2\delta}{\delta-1}h^{\epsilon}\psi^{\epsilon} \cdot Q(u^{ \epsilon}) \Big)\cdot w
  \,\text{d}x\text{d}t.
\end{split}
\end{equation}
Combining the uniform estimates obtained above, the strong convergence in (\ref{ertfinalq}),
the weak convergence in (\ref{ruojixianass}), and \eqref{initialrelation1},
and  letting $\epsilon \rightarrow 0$ in (\ref{zhenzheng1q}), we have
\begin{equation*}
\int_0^{T^*} \int  \Big(u^* \cdot w_t - (u^* \cdot \nabla) u^* \cdot w
+\frac{1}{\gamma-1}(c^*)^2\text{div} w \Big)\text{d}x\text{d}t
+\int u_0(x) \cdot w(0,x)\,\dd x=0,
\end{equation*}
where we have used Lemma \ref{localupperbound11} and
\begin{equation*}
\begin{split}
&\int_0^{T^*} \int \epsilon \Big( (h^{\epsilon})^2 Lu^{\epsilon}
   -\frac{2\delta}{\delta-1}h^{\epsilon}\psi^{\epsilon} \cdot Q(u^{ \epsilon}) \Big)\cdot w\, \text{d}x\text{d}t\\[2mm]
&\leq C\epsilon \big(\big\|h^{\epsilon}\big\|^2_{L^\infty(\supp\,w)} \big|\nabla^2 u^\epsilon\big|_2
  +\big\|h^{\epsilon}\big\|_{L^\infty(\supp\,w)}\big|\psi^\epsilon\big|_\infty
   \big|\nabla u^{\epsilon}\big|_2\big)\big|w\big|_2 \rightarrow 0
  \qquad \text{as $\epsilon \rightarrow 0$}.
\end{split}
\end{equation*}
Similarly, we can use the same argument to show that $(c^*,u^*)$ also satisfies the equation in  $(\ref{E:4.2})$
for $c=\sqrt{A\gamma}\rho^{\frac{\gamma-1}{2}}$ and the initial condition in the sense of distributions.
Thus, $(c^*,u^*)$ is a weak solution of the Cauchy problem (\ref{E:4.2}) in the sense of distributions
satisfying the  following regularity:
\begin{equation}\label{zheng-b}
(c^*, u^*)\in L^\infty([0,T_*]; H^3).
\end{equation}

5. Finally, the uniqueness obtained in \cite{tms1} yields
that the whole family $(c^\epsilon,u^\epsilon)$ converges to $(c,u)=(c^*,u^*)$ in the sense of distributions
or the strong convergence shown in (\ref{ertfinalq}).

This completes the proof of Theorem \ref{th3A}.

\medskip
The proof of Corollary \ref{th3} is similar to that of Theorem \ref{th3A}. Here we omit its details.

\subsection{Proof of Corollary \ref{example}}
We divide the proof into three steps.

\medskip
\paragraph{{\rm 1}. \em  Some auxiliary functions.}
First, let  $(\rho_0,u_0)$ satisfy
\begin{equation}\label{invisidinitial}
\rho_0\geq 0, \quad\,\, \rho^{\frac{\gamma-1}{2}}_0\in H^3,\quad\,\, u_0\in H^3.
\end{equation}

Next, denote
$$
f(x)=\frac{1}{1+|x|^{2a}}
$$
for some constant  $a$ satisfying
\begin{equation}\label{range} \frac{3}{2(\gamma-1)}< a<\frac{1}{2(1-\delta)}.\end{equation}
It is easy to check that $f(x)$ satisfies
\begin{equation}\label{fcondition}
f>0, \quad\,\, f^{\frac{\gamma-1}{2}}\in H^3, \quad\,\,  \nabla f^{\frac{\delta-1}{2}}\in D^1\cap D^2,
\quad\,\, \nabla f^{\frac{\delta-1}{4}}\in L^4.
\end{equation}

Finally, let  $\chi(x)\in C^\infty_c(\mathbb{R}^3)$ be a  truncation function  satisfying (\ref{eq:2.6-77A})
and denote $\chi_{\epsilon^q}(x)=\chi(\epsilon^q x)$ for $x\in \mathbb{R}^3$.
Then we can define the functions $(\rho^\epsilon_0,u^\epsilon_0)$ by
$$
\big(\rho^\epsilon_0\big)^{\frac{\gamma-1}{2}}=\rho^{\frac{\gamma-1}{2}}_0\chi_{\epsilon^q}+\epsilon^r  f^{\frac{\gamma-1}{2}},
\qquad  u^\epsilon_0=u_0,
$$
where  $q$ and $r$ are both positive constants to be determined later.

\smallskip
\paragraph{{\rm 2}. \em  The uniform bound of $\big\|(\rho^\epsilon_0)^{\frac{\gamma-1}{2}}\big\|_3$ and the strong convergence.}
We use $C>0$ to denote a constant depending only on $(\rho_0,u_0, f)$, $\delta$,
and  $\gamma$ in the rest of this section.

First, we have the following formula:
\begin{equation}\label{expansion1}
\begin{split}
&\partial^\zeta_x( \rho^{\frac{\gamma-1}{2}}_0\chi_{\epsilon^q})- \rho^{\frac{\gamma-1}{2}}_0\partial^\zeta_x\chi_{\epsilon^q}
  -\partial^\zeta_x \rho^{\frac{\gamma-1}{2}}_0 \cdot \chi_{\epsilon^q}\\
&=\sum_{1\le i,j,k\le 3}l_{ijk}\Big(C_{1ijk} \partial^{\zeta^i}_x \rho^{\frac{\gamma-1}{2}}_0\cdot  \partial^{\zeta^j+\zeta^k}_x \chi_{\epsilon^q}
  +C_{2ijk}  \partial^{\zeta^j+\zeta^k}_x \rho^{\frac{\gamma-1}{2}}_0 \cdot \partial^{\zeta^i}_x \chi_{\epsilon^q}\Big),
\end{split}
\end{equation}
where $\zeta=\zeta^1+\zeta^2+\zeta^3$ for three multi-indexes $\zeta^i\in \mathbb{R}^3$, $i=1,2,3$,
satisfying $|\zeta^i|=0$ or $1$, and $C_{1ijk}$ and $C_{2ijk}$  are all constants.

Then it is direct to show
\begin{equation}\label{expansion2}
\begin{split}
\big|\nabla^k \chi_{\epsilon^q}\big|_\infty\leq C \qquad &\text{for $k=0,1,2,3$},\\
\big|\nabla^k \chi_{\epsilon^q}\big|_\infty \rightarrow 0 \quad \text{as $\epsilon \rightarrow 0$} \qquad &\text{for $k=1,2,3$},\\
\int_{ |x|\geq \frac{1}{\epsilon^q}} \big|\nabla^k\rho^{\frac{\gamma-1}{2}}_0\big|^2\text{d}x\rightarrow 0\quad \text{as $\epsilon \rightarrow 0$}
  \qquad &\text{for $k=0,1,2,3$},
  \end{split}
\end{equation}
which, together with (\ref{expansion1}), implies that $\big\|(\rho^\epsilon_0)^{\frac{\gamma-1}{2}}\big\|_3\leq C$
and
$$
\big\|(\rho^\epsilon_0)^{\frac{\gamma-1}{2}}-\rho^{\frac{\gamma-1}{2}}_0\big\|_3\rightarrow 0
\qquad \text{as $\epsilon \rightarrow 0$}.
$$

\smallskip
\paragraph{{\rm 3}. \em  The uniform bound
of $\epsilon^{\frac{1}{2}} \big\|\psi^\epsilon_0\big\|_{D^1\cap D^2}+\epsilon^{\frac{1}{4}}\big|n^\epsilon_0\big|_{4}$
with $n^\epsilon_0=\nabla \big(\rho^\epsilon_0\big)^{\frac{\delta-1}{4}}$.}
From the definition of $(\rho^\epsilon_0)^{\frac{\gamma-1}{2}}$, we have
\begin{equation}\label{gongshi}
\begin{split}
\psi^\epsilon_0=&\nabla \big(\rho^\epsilon_0\big)^{\frac{\delta-1}{2}}
  =\iota\big(\rho^{\frac{\gamma-1}{2}}_0\chi_{\epsilon^q}
   +\epsilon^r  f^{\frac{\gamma-1}{2}}\big)^{\iota-1}\big(\nabla (\rho^{\frac{\gamma-1}{2}}_0\chi_{\epsilon^q})
    +\epsilon^r \nabla f^{\frac{\gamma-1}{2}}\big),\\
n^\epsilon_0=&\nabla \big(\rho^\epsilon_0\big)^{\frac{\delta-1}{4}}
 =\frac{\iota}{2}\big(\rho^{\frac{\gamma-1}{2}}_0\chi_{\epsilon^q}+\epsilon^r  f^{\frac{\gamma-1}{2}}\big)^{\frac{\iota}{2}-1}\big(\nabla (\rho^{\frac{\gamma-1}{2}}_0\chi_{\epsilon^q})
     +\epsilon^r \nabla f^{\frac{\gamma-1}{2}}\big).
\end{split}
\end{equation}
Denote $B_{\epsilon,q}=B_{\frac{2}{\epsilon^q}}$ and  $B^C_{\epsilon,q}=B^C_{\frac{2}{\epsilon^q}}=\mathbb{R}^3/B_{\frac{2}{\epsilon^q}}$
in the rest of this section for simplicity, and
\begin{equation}\label{division}\begin{split}
&\big\|\psi^\epsilon_0\big\|_{D^1\cap D^2}
=\big\|\psi^\epsilon_0\big\|_{D^1\cap D^2(B_{\epsilon,q})}+\big\|\psi^\epsilon_0\big\|_{D^1\cap D^2(B^C_{\epsilon,q})};\\
&\big|n^\epsilon_0\big|_{4}=\big\|n^\epsilon_0\big\|_{L^4(B_{\epsilon,q})}+\big\|n^\epsilon_0\big\|_{L^4(B^C_{\epsilon,q})}.
\end{split}
\end{equation}
Notice that
\begin{equation}\label{flocal}
f(x)\geq \frac{\epsilon^{2aq}}{\epsilon^{2aq}+2^{2a}}\qquad\,\, \text{for $x\in B_{\epsilon,q}$}.
\end{equation}
Then it follows that
\begin{equation}\label{cal1}\begin{split}
\epsilon^{\frac{1}{4}}\big|n^\epsilon_0\big|_{4}
 =&\,\epsilon^{\frac{1}{4}}\big\|n^\epsilon_0\big\|_{L^4(B_{\epsilon,q})}
   +\epsilon^{\frac{1}{4}}\big\|n^\epsilon_0\big\|_{L^4(B^C_{\epsilon,q})}\\
\leq &\, C\epsilon^{\frac{1}{4}}\Big\|
 \Big(\epsilon^r  \big( \frac{\epsilon^{2aq}}{\epsilon^{2aq}+2^{2a}}\big)^{\frac{\gamma-1}{2}}\Big)^{\frac{\iota}{2}-1}
  \big(\nabla (\rho^{\frac{\gamma-1}{2}}_0\chi_{\epsilon^q} )+\epsilon^r \nabla f^{\frac{\gamma-1}{2}}\big)\Big\|_{L^4(B_{\epsilon,q})   }\\
&\,+C\epsilon^{\frac{1}{4}}
  \Big\|\big(\epsilon^r  f^{\frac{\gamma-1}{2}}\big)^{\frac{\iota}{2}-1}\epsilon^r \nabla f^{\frac{\gamma-1}{2}}\Big\|_{L^4(B^C_{\epsilon,q})}\\
\leq &\, C\epsilon^{\frac{1}{4}-(r+aq(\gamma-1))(1-\frac{\iota}{2})}\leq C,
\end{split}
\end{equation}
under the condition:
\begin{equation}\label{cond1}
0<r+aq(\gamma-1)<\frac{1}{4-2\iota}.
\end{equation}

Next, for $\psi^\epsilon_0$,  we have
\begin{equation}\label{cal2}\begin{split}
\nabla \psi^\epsilon_0=&\, \iota\big(\rho^{\frac{\gamma-1}{2}}_0\chi_{\epsilon^q}+\epsilon^r  f^{\frac{\gamma-1}{2}}\big)^{\iota-1}
   \big(\nabla^2(\rho^{\frac{\gamma-1}{2}}_0\chi_{\epsilon^q} )+\epsilon^r \nabla^2 f^{\frac{\gamma-1}{2}}\big)\\
&+ \iota(\iota-1)\big(\rho^{\frac{\gamma-1}{2}}_0\chi_{\epsilon^q}+\epsilon^r  f^{\frac{\gamma-1}{2}}\big)^{\iota-2}
   \big(\nabla\big(\rho^{\frac{\gamma-1}{2}}_0\chi_{\epsilon^q} \big)+\epsilon^r \nabla f^{\frac{\gamma-1}{2}}\big)^2,\\
\nabla^2 \psi^\epsilon_0=&\, \iota\big(\rho^{\frac{\gamma-1}{2}}_0\chi_{\epsilon^q}+\epsilon^r  f^{\frac{\gamma-1}{2}}\big)^{\iota-1}\big(\nabla^3(\rho^{\frac{\gamma-1}{2}}_0\chi_{\epsilon^q})
    +\epsilon^r \nabla^3 f^{\frac{\gamma-1}{2}}\big)\\
&+ 3\iota(\iota-1)\big(\rho^{\frac{\gamma-1}{2}}_0\chi_{\epsilon^q}+\epsilon^r  f^{\frac{\gamma-1}{2}}\big)^{\iota-2}\\
&\quad \times\big(\nabla (\rho^{\frac{\gamma-1}{2}}_0\chi_{\epsilon^q} )
   +\epsilon^r \nabla f^{\frac{\gamma-1}{2}}\big)
    \cdot \big(\nabla^2(\rho^{\frac{\gamma-1}{2}}_0\chi_{\epsilon^q})+\epsilon^r \nabla^2 f^{\frac{\gamma-1}{2}}\big)\\
&+ \iota(\iota-1)(\iota-2)\big(\rho^{\frac{\gamma-1}{2}}_0\chi_{\epsilon^q}+\epsilon^r  f^{\frac{\gamma-1}{2}}\big)^{\iota-3}
  \big(\nabla (\rho^{\frac{\gamma-1}{2}}_0\chi_{\epsilon^q})  +\epsilon^r \nabla f^{\frac{\gamma-1}{2}}\big)^3.
\end{split}
\end{equation}
Notice that
\begin{equation*}
\begin{split}
( f^{\frac{\gamma-1}{2}})^{\iota-1} \nabla^2 f^{\frac{\gamma-1}{2}}
=&\,\frac{1}{\iota}\Big(\nabla^2( f^{\frac{\gamma-1}{2}})^\iota-\frac{4(\iota-1)}{\iota}(\nabla f^{\frac{\delta-1}{4}})^2\Big),\\
 (f^{\frac{\gamma-1}{2}})^{\iota-1}\nabla^3 f^{\frac{\gamma-1}{2}}
=&\,\frac{1}{\iota}\Big(\nabla^3( f^{\frac{\gamma-1}{2}})^\iota-3\iota(\iota-1)(f^{\frac{\gamma-1}{2}})^{\iota-2}
     \nabla f^{\frac{\gamma-1}{2}} \nabla^2 f^{\frac{\gamma-1}{2}}\\
 &\quad\,\,\, -\iota(\iota-1)(\iota-2) (f^{\frac{\gamma-1}{2}})^{\iota-3}(\nabla f^{\frac{\gamma-1}{2}})^3\Big),
 \end{split}
\end{equation*}
which,  together with  (\ref{cal2}) and the similar argument used in the derivation of (\ref{cal1}), implies
\begin{equation*}
\begin{split}
\epsilon^{\frac{1}{2}} \big\|\psi^\epsilon_0\big\|_{D^1\cap D^2}
\leq C\epsilon^{\frac{1}{2}-(r+aq(\gamma-1))(3-\iota)}\leq C,
 \end{split}
\end{equation*}
under the condition:
\begin{equation}\label{cond2}
0<r+aq(\gamma-1)<\frac{1}{6-2\iota}.
\end{equation}

Thus, following (\ref{cond1}) and (\ref{cond2}), we can just choose
$$
r_0=\frac{1}{6(3-\iota)},\qquad\, q_0=\frac{1}{6a(\gamma-1)(3-\iota)},
$$
and $ (\rho^\epsilon_0,u^\epsilon_0)$ can be given as
$$
\big(\rho^\epsilon_0\big)^{\frac{\gamma-1}{2}}=\rho^{\frac{\gamma-1}{2}}_0\chi_{\epsilon^{q_0}}
+\epsilon^{r_0}  \Big(\frac{1}{1+|x|^{2a}}\Big)^{\frac{\gamma-1}{2}},
\qquad  u^\epsilon_0=u_0,
$$
where $a$ satisfies (\ref{range}).

This completes the proof of Corollary \ref{example}.

\medskip
\section{Nonexistence of Global Solutions with $L^\infty$ Decay}

This section is devoted to the proof of Theorem \ref{th:2.20}. Denote the total kinetic energy as
$$
E_{\rm k}(t)= \frac{1}{2}\int (\rho |u|^2)(t,x)\,\dd x.
$$
For simplicity, in this section, we denote $(\rho^\epsilon,u^\epsilon)$ as $(\rho,u)$,
and $(\rho^\epsilon_0,u^\epsilon_0)$ as $(\rho_0,u_0)$.

First, the conservation of momentum can be verified.
\begin{lemma}\label{lemmak}
Let \eqref{canshu}  hold and $\epsilon\geq 0$, and let $(\rho,u)$ be the regular solution obtained
in Theorems {\rm\ref{thmakio}} and  {\rm\ref{th2}}.
Then
$$
\mathbb{P}(t)=\mathbb{P}(0) \qquad  \text{for $t\in [0,T]$}.
$$
\end{lemma}

\begin{proof}
The momentum equations imply
\begin{equation}\label{deng1}
\mathbb{P}_t=-\int \text{div}(\rho u \otimes u)\,\dd x -\int \nabla p\,\dd x
+\int \text{div}\mathbb{T}\,\dd x=0,
\end{equation}
where we have used the fact that
$$
(\rho u\otimes u, \rho^\gamma, \rho^\delta \nabla u) \in W^{1,1}(\mathbb{R}^3).
$$
\end{proof}

Now we are ready to prove Theorem \ref{th:2.20}.
Let $T>0$ be any constant.
It  follows from the definitions of $m(t)$, $\mathbb{P}(t)$, and $E_{\rm k}(t)$ that
$$
 |\mathbb{P}(t)|\leq \int \rho(t,x)|u|(t,x)\,\dd x \leq  \sqrt{2m(t)E_{\rm k}(t)},
$$
which, together with \eqref{conmass} and Lemma \ref{lemmak}, implies
$$
0<\frac{|\mathbb{P}(0)|^2}{2m(0)}\leq E_{\rm k}(t)\leq \frac{1}{2} m(0)|u(t)|^2_\infty
\qquad\,\, \text{for $t\in [0,T]$}.
$$
Then there exists a positive constant $C_u$ such that
$$
|u(t)|_\infty\geq C_u  \qquad \text{for $t\in [0,T]$}.
$$
Thus, we have obtained the desired conclusion as shown in Theorem \ref{th:2.20}.

\begin{remark}\label{zhunbei3}
In the sense of the three  fundamental  conservation laws in fluid dynamics,
the definition of regular solutions with vacuum is consistent with the physical background
of the compressible Navier-Stokes equations.
This is the reason why we say in Remark {\rm \ref{r1}} that the regular solutions
defined above select the fluid velocity in a physically reasonable way.\end{remark}

\appendix
\section{Some Basic Lemmas}

In this appendix, for self-containedness, we list  some basic lemmas that are frequently used in our proof.
The first is the  well-known Gagliardo-Nirenberg inequality.

\begin{lemma}[\cite{oar}]\label{lem2as}\
For $p\in [2,6]$, $q\in (1,\infty)$, and $r\in (3,\infty)$,
there exists a generic constant $C> 0$ that may depend on $q$ and $r$ such that,
for any $f\in H^1(\mathbb{R}^3)$ and $g\in L^q(\mathbb{R}^3)\cap D^{1,r}(\mathbb{R}^3)$,
\begin{equation}\label{33}
|f|^p_p \leq \, C |f|^{\frac{6-p}{2}}_2 |\nabla f|^{\frac{3p-6}{2}}_2,\quad
|g|_\infty\leq\, C |g|^{\frac{q(r-3)}{3r+q(r-3)}}_q |\nabla g|^{\frac{3r}{3r+q(r-3)}}_r.
\end{equation}
\end{lemma}

Some special  cases of the inequalities are
\begin{equation}\label{ine}\begin{split}
|u|_6\leq C|u|_{D^1},\quad\,\, |u|_{\infty}\leq C|u|^{\frac{1}{2}}_6|\nabla u|^{\frac{1}{2}}_{6},
\quad\,\, |u|_{\infty}\leq C\|u\|_{W^{1,r}}.
\end{split}
\end{equation}

The second lemma is some compactness results obtained via the Aubin-Lions lemma.

\begin{lemma}[\cite{jm}]\label{aubin}
Let $X_0\subset X\subset X_1$ be three Banach spaces.
Suppose that $X_0$ is compactly embedded in $X$, and $X$ is continuously embedded in $X_1$.
Then the following statements hold{\rm :}
\begin{enumerate}
\item[(i)] If $J$ is bounded in $L^p([0,T];X_0)$ for $1\leq p < \infty$,
and $\frac{\partial J}{\partial t}$ is bounded in $L^1([0,T];X_1)$,
then $J$ is relatively compact in $L^p([0,T];X)$.

\smallskip
\item[(ii)] If $J$ is bounded in $L^\infty([0,T];X_0)$  and $\frac{\partial J}{\partial t}$ is bounded in $L^p([0,T];X_1)$ for $p>1$,
then $J$ is relatively compact in $C([0,T];X)$.
\end{enumerate}
\end{lemma}

The following three lemmas contain some Sobolev
inequalities on the product estimates, the interpolation estimates,
the commutator inequality, {\it etc.},  which can be found in many
references; {\it cf.} Majda \cite{amj}.  We omit their proofs.

\begin{lemma}
\label{DZ5}
Let functions $u, v\in H^s$ and $s \geq 2$.
Then $u\cdot v\in H^s$, and there exists a constant $C_s$ depending only on $s$ such that
\[
\|u\cdot v\|_s\leq C_s\|u\|_s\|v\|_s.
\]
\end{lemma}

\begin{lemma}
\label{DZ37}
Let $u\in H^s$.
Then, for any $s'\in[0, s]$, there exists a
constant $C_s$ depending only on $s$ such that
\begin{equation}
\|u\|_{s'}\leq C_s\|u\|^{1-\frac{s'}{s}}_0\|u\|^{\frac{s'}{s}}_s.
\end{equation}
\end{lemma}

\begin{lemma}
\label{zhen1}
Let  $r$, $a$, and $b$  be constants such that
$$
\frac{1}{r}=\frac{1}{a}+\frac{1}{b},\qquad 1\leq a,\ b, \ r\leq \infty.
$$
Then, for any $s\geq 1$, if $f, g \in W^{s,a} \cap  W^{s,b}(\mathbb{R}^3)$,
\begin{equation}\begin{split}\label{ku11}
&\big|\nabla^s(fg)-f \nabla^s g\big|_r
\leq C_s\big(\big|\nabla f\big|_a \big|\nabla^{s-1}g\big|_b+\big|\nabla^s f\big|_b\big|g\big|_a\big),\\
\end{split}
\end{equation}
\begin{equation}\begin{split}\label{ku22}
&\big|\nabla^s(fg)-f \nabla^s g\big|_r
\leq C_s\big(\big|\nabla f\big|_a \big|\nabla^{s-1}g\big|_b+\big|\nabla^s f\big|_a\big|g\big|_b\big),
\end{split}
\end{equation}
where $C_s> 0$ is a constant depending only on $s$,
$\nabla^s f$ is the set of  all $\partial^\zeta_x f$  with $|\zeta|=s\ge 1$,
and $\zeta=(\zeta_1,\zeta_2,\zeta_3)\in \mathbb{R}^3$ is a multi-index.
\end{lemma}

The final lemma is useful to improve weak convergence to strong convergence.
\begin{lemma}
\label{zheng5}
If the function sequence $\{w_n\}^\infty_{n=1}$ converges weakly  to $w$ in an Hilbert space $X$,
then it converges strongly to $w$ in $X$ if and only if
$$
\|w\|_X \geq \lim \text{sup}_{n \rightarrow \infty} \|w_n\|_X.
$$
\end{lemma}

\bigskip
\bigskip
{\bf Acknowledgement:}
The research of Geng Chen was supported in part by the US National Science Foundation Grant DMS-1715012.
The research of Gui-Qiang G. Chen was supported in part by
the UK Engineering and Physical Sciences Research Council Awards
EP/L015811/1 and EP/V008854/1, and the Royal Society--Wolfson Research Merit Award (UK).
The research of Shengguo Zhu was supported in part by
the China National Natural Science Foundation under Grant 1210010137, the Royal Society--Newton International Fellowships NF170015,
the Newton International Fellowships Alumni AL/201021 and AL/211005,
the Monash University-Robert Bartnik Visiting Fellowships,
and the Shanghai Frontier Science Research Center for Modern Analysis.

\end{document}